\documentclass [12pt]{amsart}

\usepackage[top=100pt,bottom=60pt,left=96pt,right=92pt]{geometry}
\usepackage[utf8]{inputenc}
\usepackage[T1]{fontenc}
\usepackage[catalan,german,spanish,dutch,british]{babel}
\usepackage{lmodern}
\usepackage[pdftex]{graphicx}
\usepackage{tikz}
\usetikzlibrary{matrix}
\usepackage{pinlabel}
\usepackage{faktor} 

\usepackage{mathtools} 

\usepackage[mathscr]{eucal}
\usepackage{amsmath}
\usepackage{bm} 
\usepackage{amsthm}
\usepackage{amssymb}
\usepackage{amsfonts} 
\usepackage{mathrsfs}
\usepackage{amscd}
\usepackage{enumitem} 
\usepackage{xcolor}
\usepackage{layout}
\usepackage{stmaryrd} 
\usepackage{float} 
\usepackage{longtable}
\usepackage{txfonts} 
\usepackage{setspace} 
\usepackage{subcaption} 
\usepackage{verbatim}
\usepackage{nicefrac}
\usepackage{textcomp}

\usepackage[normalem]{ulem} 
\usepackage{tikz-cd} 

\usepackage{chngcntr} 

\usepackage{hyperref} 

\hypersetup{
    colorlinks=true,
    linkcolor=blue,
    filecolor=blue,      
    urlcolor=blue,
    citecolor = blue,
    }
 
\numberwithin{equation}{section}
\counterwithin{figure}{section}

\newtheorem{theorem}{Theorem}[section]
\newtheorem{proposition}[theorem]{Proposition}
\newtheorem{corollary}[theorem]{Corollary}
\newtheorem{lemma}[theorem]{Lemma}

\newtheorem{definition}[theorem]{Definition}
\newtheorem{remark}[theorem]{Remark}
\newtheorem{example}[theorem]{Example}

\theoremstyle{plain}

\newcommand{\purge}[1]{} 

\newcommand{\vungoc}{V\~u Ng\d{o}c}

\def\epsilon{\varepsilon}
\def\phi{\varphi}
\def\theta{\vartheta}

\newcommand{\al}{\alpha}
\newcommand{\be}{\beta}
\newcommand{\ga}{\gamma}

\newcommand{\ka}{\kappa}
\newcommand{\lam}{\lambda}

\newcommand{\si}{\sigma}

\newcommand{\om}{\omega}

\newcommand{\Ga}{\Gamma}
\newcommand{\De}{\Delta}

\newcommand{\gti}{{\ti{g}}}

\newcommand{\pti}{{\ti{p}}}

\newcommand{\zbar}{\bar{z}}

\def\C{{\mathbb C}}
\def\D{{\mathbb D}}

\def\F{{\mathbb F}}

\def\N{{\mathbb N}}

\def\R{{\mathbb R}}
\def\mbS{{\mathbb S}} 
\def\T{{\mathbb T}}

\newcommand{\mcC}{\mathcal C}

\newcommand{\mcF}{\mathcal F}
\newcommand{\mcG}{\mathcal G}
\newcommand{\mcH}{\mathcal H}

\newcommand{\mcJ}{\mathcal J}

\newcommand{\mcL}{\mathcal L}

\newcommand{\mcS}{\mathcal S}

\newcommand{\mcX}{\mathcal X}

\newcommand{\IFF}{\Leftrightarrow}

\newcommand{\ti}{\tilde}

\DeclarePairedDelimiter{\abs}{\lvert}{\rvert}

\DeclareMathOperator{\Hyp}{Hyp}

\DeclareMathOperator{\Id}{Id}

\DeclareMathOperator{\Mod}{mod}

\DeclareMathOperator{\Span}{Span}

\DeclareMathOperator{\tr}{tr}

\def\slashii#1{\setbox0=\hbox{$#1$}             
\dimen0=\wd0                                 
\setbox1=\hbox{\sl/} \dimen1=\wd1            
\ifdim\dimen0>\dimen1                        
\rlap{\hbox to \dimen0{\hfil\sl/\hfil}}   
#1                                        
\else                                        
\rlap{\hbox to \dimen1{\hfil$#1$\hfil}}   
\hbox{\sl/}                               
\fi}                                         %
\def\slashiii#1{\setbox0=\hbox{$#1$}#1\hskip-\wd0\hbox to\wd0{\hss\sl/\/\hss}}
\allowdisplaybreaks

\begin{document}

\title{Creating hyperbolic-regular singularities in the presence of an $\mathbb S^1$-symmetry}

\author{Yannick Gullentops $\&$ Sonja Hohloch}

\date{\today}

\begin{abstract}
On a 4-dimensional compact symplectic manifold, we study how suitable perturbations of a toric system to a family of completely integrable systems with $\mbS^1$-symmetry lead to various hyperbolic-regular singularities. We compute and visualise associated phenomena like flaps, swallowtails, and $k$-stacked tori for $k \in \{2, 3, 4\}$ and give an upper bound for $k$ in our family of systems.
\end{abstract}

\maketitle

\tableofcontents


\section{Introduction}

Hamiltonian systems are an important class of dynamical systems since they appear naturally in many different contexts from mathematics over physics to chemistry and biology. 
Hamiltonian systems always have at least one conservation law, namely conservation of energy, but there are surprisingly many systems that display additional conservation laws and/or symmetries.
Intuitively a Hamiltonian systems is `completely integrable' if it has a `maximal number' of `conserved quantities/symmetries' (see Definition \ref{def:integrable} for the proper definition specialised to dimension four). 

In this paper, we focus on a certain family of integrable systems on a $4$-dimensional manifold that has an underlying $\mbS^1$-symmetry. This type of symmetry appears in many natural systems, for instance the coupled spin oscillator, the coupled angular momenta, the Lagrange, Euler, and Kovalevskaya
spinning tops, and the spherical pendulum. Over the course of the past decade, there has been a lot of progress in various aspects of completely intergable systems with $\mbS^1$-symmetry on $4$-dimensional manifolds, for instance, in symplectic classification (Pelayo $\&$ \vungoc\ \cite{Pelayo2009ConstructingIS,Pelayo2009SemitoricIS}, Palmer $\&$ Pelayo $\&$ Tang \cite{palmpelaytangsemitoric}), computation of invariants (Alonso $\&$ Dullin $\&$ Hohloch \cite{ALONSO2019131,Alonso_2019}, Alonso $\&$ Hohloch \cite{Alonso_2021}), study of specific examples (Hohloch $\&$ Palmer \cite{HolAndPalmFamily}, De Meulenaere $\&$ Hohloch \cite{meulenaere2019family}), bifurcation behavior (Le Floch $\&$ Palmer \cite{floch2019semitoric}), quantum aspects (Le Floch $\&$ \vungoc\ \cite{floch2021inverse}), extension of $\mbS^1$-actions (Hohloch $\&$ Palmer \cite{hohloch2021extending}), and so on.

So far, there are symplectic classifications in the following situations of interest for this paper: 
\begin{itemize}
 \item 
 for toric systems (see Definition \ref{def:toric}), it was established by Delzant (cf.\ Theorem \ref{theo:delzant}),
 \item
for semitoric systems (see Definition \ref{def:semitoric}), it was achieved by Pelayo $\&$ \vungoc\ \cite{Pelayo2009ConstructingIS,Pelayo2009SemitoricIS} and Palmer $\&$ Pelayo $\&$ Tang \cite{palmpelaytangsemitoric}.
\end{itemize}

\begin{figure}[h]
    \centering
\begin{subfigure}[t]{.48\textwidth}
\centering
     \includegraphics[scale = .3]{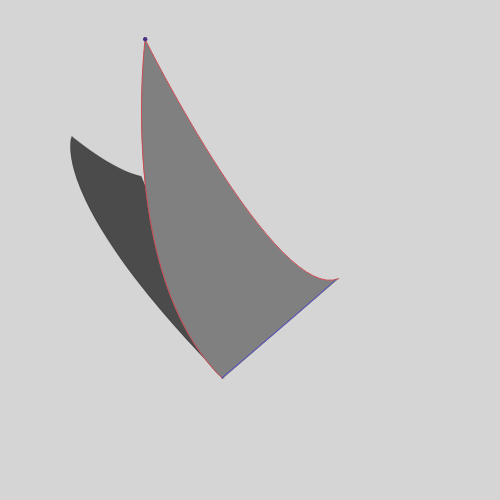}
     \caption{A flap displayed by an unfolded bifurcation diagram. The flap branches off the darker grayish background along the red line.}
     \label{fig:unfolded flap}
  \end{subfigure}\quad
\begin{subfigure}[t]{.48\textwidth}
\centering    
    \includegraphics[scale = .2]{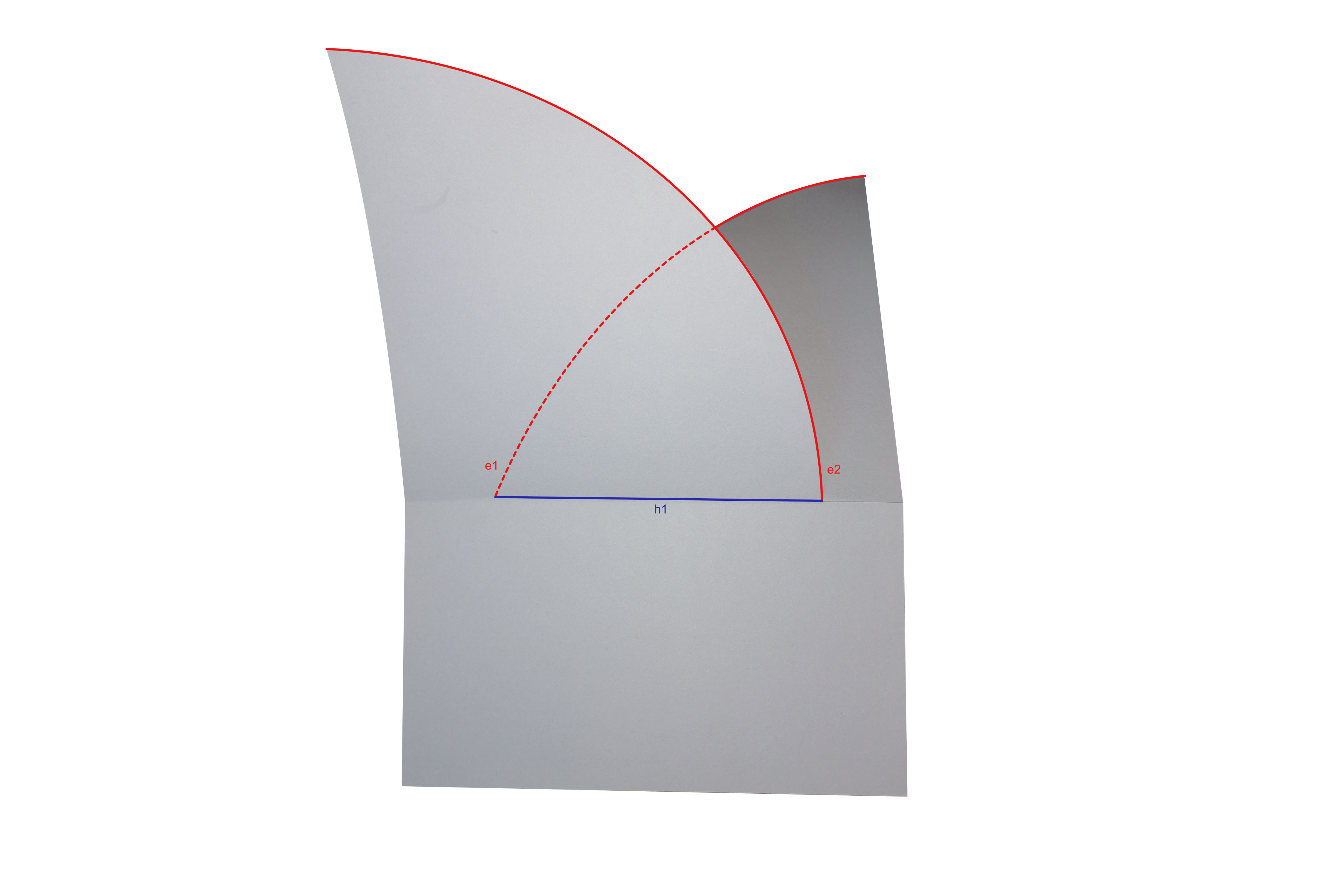}
    \caption{A swallowtail displayed by an unfolded bifurcation diagram.}
    \label{fig:swallowtailunf}
\end{subfigure}
\caption{Flaps and Swallowtails: the red line consists of hyperbolic-regular values, the green lines of elliptic-regular values and the dark dot is an elliptic-elliptic value.}
\label{fig:flapAndSwallow}
\end{figure}

Neither of these two classes of systems admit singularities with hyperbolic components. Unfortunately, there is not yet any symplectic classification for more general classes up to our knowledge.
The aim of the present paper is to gain more understanding of `semitoric systems with hyperbolic singularities' in order to prepare the way towards a symplectic classification in the future. This is done by suitable perturbations of the toric system underlying the semitoric system studied by De Meulenare $\&$ Hohloch \cite{meulenaere2019family}. This provides us with explicit examples of hyperbolic phenomena like flaps and swallowtails, see Figure \ref{fig:flapAndSwallow}.

\begin{figure}[hb]
    \centering
\begin{subfigure}[t]{.48\textwidth}
\centering
   \includegraphics[scale = .55]{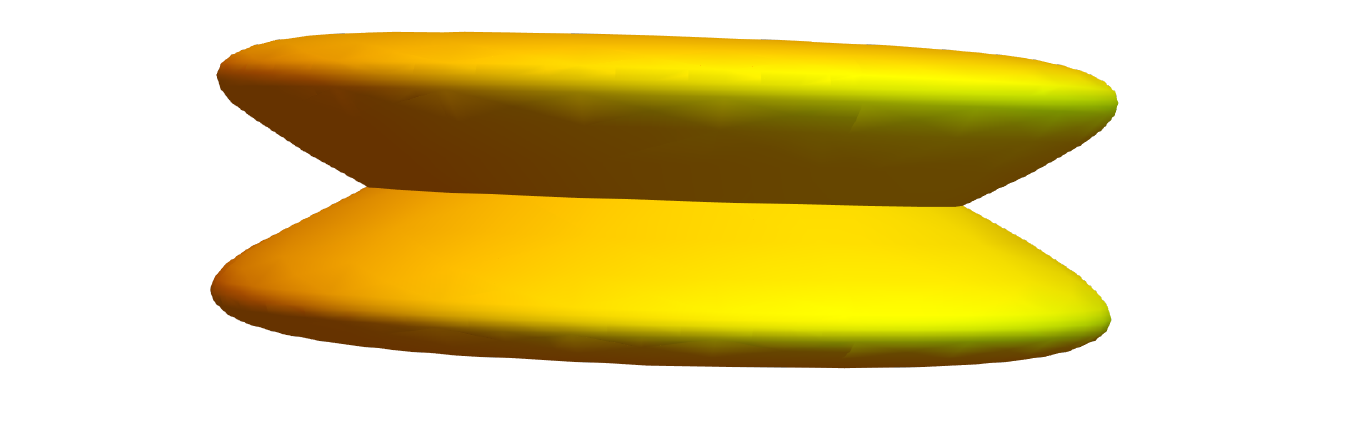}
    \caption{A $2$-stacked torus.}
    \label{fig:doubleTorus}
 \end{subfigure}\quad
\begin{subfigure}[t]{.48\textwidth}
\centering
\includegraphics[scale = .35]{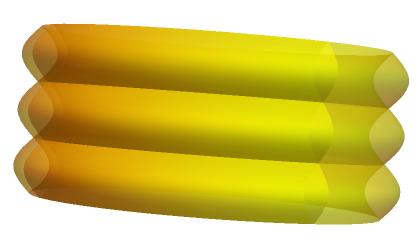}
\caption{A $3$-stacked torus.}
 \label{fig:tritorus}
\end{subfigure}
\caption{Hyperbolic-regular fibres.
}
\label{fig:hypRegFibres}
\end{figure}

Moreover, we can visualise the occuring hyperbolic-regular fibres, finding explicit examples of $k$-stacked tori for $k \in \{2, 3,4\}$, see Figures \ref{fig:doubleTorus}, \ref{fig:tritorus},  \ref{fig:levelSetDutorus}, \ref{fig:levelSetTritorus}, and \ref{fig:quadtorus}.

Now let us be more precise: Consider the octagon $\De$ in Figure \ref{fig:octagon}. Using Delzant's construction \cite{Delzant1988HamiltoniensPE}, De Meulenare $\&$ Hohloch \cite{meulenaere2019family} built the toric system associated with this octagon on a 4-dimensional, compact, connected, symplectic manifold $(M, \om):= (M_\De, \om_\De)$ by means of symplectic reduction by a Hamiltonian $\T^6$-action of the 10-dimensional preimage of a certain map from $\C^8 \to \R^6$ (more details are given in Section \ref{sec:theoctsyst} which summarises De Meulenare $\&$ Hohloch \cite{meulenaere2019family}). Points in $(M, \om)$ are written as equivalence classes of the form $[z]=[z_1, \dots, z_8]$ where $z_k = x_k + i y_k \in \C$ for $1 \leq k \leq 8$. The momentum map of the toric system on $(M, \om)$ is given by 
$$
F=(J, H): (M, \om) \to \R^2 \quad \mbox{with} \quad J([z_1, \dots, z_8])= \frac{1}{2}\abs{z_1}^2 \quad \mbox{and} \quad H([z_1, \dots, z_8])= \frac{1}{2} \abs{z_3}^2
$$
and satisfies $F(M)= \De$. The induced Hamiltonian $2$-torus action is effective.

\begin{figure}[h]
\centering
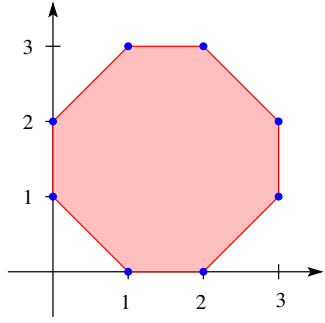
\caption{The octagon $\De \subset \R^2$ with vertices $(0,2)$, $ (0,1)$, $ (1, 0)$, $ (2, 0)$, $ (3, 1)$, $ (3, 2)$, $(2, 3)$ and $ (1, 3)$ is a Delzant polytope.}
\label{fig:octagon}
\end{figure}

We now perturb the second component of the system $F=(J, H)$. To this aim, we define the function $H_t: M \to \R$ via
\begin{align*}
    H_t:= H_{(t_1,t_2,t_3,t_4)} := (1-2 t_1)H + \sum_{j = 1}^4 t_j\gamma_j 
\end{align*}
where $t:=(t_1, t_2, t_3, t_4) \in \R^4$ and $\ga_1, \dots, \ga_4$ are given by
\begin{align*}
    \gamma_1([z])  := \frac{1}{100} \ \bigl(\overline{z_2z_3z_4}z_6z_7z_8+z_2z_3z_4\overline{z_6z_7z_8}\bigr) = \frac{1}{50} \ \Re(\overline{z_2z_3z_4}z_6z_7z_8) 
\end{align*}    
where $\Re$ denotes the real part of a complex number and
 \begin{align*}   
     \gamma_{2}([z])  : =\frac{1}{50} \ |z_5|^4|z_4|^4, \quad 
    \gamma_{3}([z])  : = \frac{1}{50} \ |z_4|^4|z_7|^4,\quad 
    \gamma_{4}([z])    : = \frac{1}{100} \ |z_5|^4|z_7|^4.
\end{align*}
The term $\gamma_1$ comes from the perturbation performed in De Meulenaere $\&$ Hohloch \cite{meulenaere2019family}.
The terms $\ga_2$, $\ga_3$, $\ga_4$ are inspired by the work of Palmer $\&$ Le Floch \cite{floch2019semitoric}. 

\begin{theorem}
 For all $t \in \R^4$, the system $F_t:=(J, H_t): (M, \om) \to \R^2$ is completely integrable and has an effective Hamiltonian $\mbS^1$-action generated by $J$.  
\end{theorem}

This statement summarises the results of Theorem \ref{th:toricOctagonSys}, Proposition \ref{prop:gammacommute}, and Proposition \ref{prop:gammadependend} which are proven in Section \ref{sec:theoctsyst} and Section \ref{sec:defSystem}.

Central for us is a good knowledge of the whereabouts and types of the singular points of the system $F_t:=(J, H_t)$:

\begin{theorem}
The four points
\begin{align*}
& \left[\sqrt{2}, \ 0, \ 0,\sqrt{2}, \ 2, \ 2\sqrt{2},\sqrt{6},\sqrt{6}\right], && \left[2, \ 2\sqrt{2},\sqrt{6},\sqrt{6},\sqrt{2}, \ 0, \ 0,\sqrt{2} \right], \\
 & \left[\sqrt{2},\sqrt{6},\sqrt{6}, \ 2\sqrt{2}, \ 2,\sqrt{2}, \ 0, \ 0 \right],  && \left[2,\sqrt{2}, \ 0, \ 0,\sqrt{2},\sqrt{6},\sqrt{6}, \ 2\sqrt{2}\right]
\end{align*} 
are singular of rank zero for $F_t=(J, H_t)$ for all $t \in \R^4$. These points lie in $J^{-1}(1) \cup\ J^{-1}(2)$. Any other rank zero points of $F_t=(J, H_t)$ can only appear in $J^{-1}(0) \cup\ J^{-1}(3)$.
Rank one singular points are determined by solving a certain polynomial equation.
\end{theorem}

This statement summarises the results of Theorem \ref{the:invariablesingpoint} and  Theorem \ref{the:zerocor} and Corollary \ref{cor:finalsing} which are proven in Section \ref{sec:redmanicritpoints}.

Knowing the whereabouts and types of singular points, we study the appearance and topology of hyperbolic-regular fibres in the manifold:

\begin{theorem}
There are explicit, visualised examples of hyperbolic singular fibres are given by $k$-stacked tori for $k \in \{2, 3,4\}$. Moreover, $k$ can be maximally 13 for this system. Twisted tori (as displayed in Figure \ref{fig:twistDoubleTorus}) do not appear.
\end{theorem}

This statement summarises the Examples \ref{ex:lemniscate}, \ref{ex:tritorus}, and \ref{ex:quadtorus} and Theorem \ref{th:maxfibre} which are proven in Section \ref{sec:examhyp-reg-fib}.

Finally, we focus on the image of the momentum map and the unfolded bifurcationdiagram (see Definition \ref{de:unfbifd}) and study the effects of hyperbolic-regular values there.

\begin{proposition}
 There are explicit, visualised examples for flaps and swallowtails and their collisions.
\end{proposition}

This statement summarises the Examples \ref{ex:singleflap}, \ref{ex:doubleflap}, \ref{ex:swallowtail}, and \ref{ex:flapcol} which are proven in Section \ref{sec:singbif}.

\subsection*{Overview of the paper}
In Section \ref{sec:prelim}, we introduce necessary notions and conventions. 
In Section \ref{sec:theoctsyst}, we recall the toric octagon system from the work of De Meulenaere $\&$ Hohloch \cite{meulenaere2019family}. 
In Section \ref{sec:defSystem}, we define the perturbations of the toric octagon system that form the foundation of this paper. 
In Section \ref{sec:redmanicritpoints}, we find criteria to determine the (type of) singular points of this newly generated family of systems. 
In Section \ref{sec:examhyp-reg-fib}, we determine and analyse the topological shape of the occuring hyperbolic-regular fibres. 
In Section \ref{sec:singbif}, we study explicit examples of the appearance of flaps and swallowtails in our family of systems.

\subsection*{Acknowledgements}
The authors wish to thank Jaume Alonso, Joaquim Brugu\'es, Yohann Le Floch, and Joseph Palmer for helpful discussions and useful comments. Moreover, we thank Wim Vanroose for sharing his computational resources.
The first author was supported by the UA BOF DocPro4 grant with UA Antigoon ProjectID 34812 and the FWO-FNRS Excellence of Science project G0H4518N with UA Antigoon ProjectID 36584. The second author was partially supported by these two grants.


\section{Foundations and conventions}
\label{sec:prelim}

Throughout this paper, we are working mostly on 4-dimensional symplectic manifolds. Thus, in order to keep the notation to a minimum, we will adapt the necessary definitions and facts from the literature directly to our 4-dimensional setting.


\subsection{Completely integrable systems in dimension four}

Let $(M, \om)$ be a 4-dimensional symplectic manifold. Given a smooth function $f: M \to \R$, its {\em Hamiltonian vector field} $X^f$ is defined via $\om(X^f,\cdot \ )= df(\cdot)$ and the flow of the (autonomous) {\em Hamiltonian equation} $z' = X^f(z)$ is denoted by $\Phi^f_t$. In this situation, the function $f$ is usually referred to as the {\em Hamiltonian}.

On $(M, \om)$, the induced {\em Poisson bracket} of two smooth functions $g, h: M \to \R$ is given by $\{g,h\}:= \om(X^g, X^h)$. If $\{g, h\}=0$ the functions $g$ and $h$ are said to {\em Poisson commute}. For the Lie bracket of two smooth vector fields $A$ and $B$ on $M$ and a function $f \in C^\infty(M, \R)$, we use the convention $[A, B]= A(B(f))-B(A(f))$. This leads to the relation $[X^g, X^h]=X^{\{g, h\}}$ for $g,h \in C^\infty(M, \R)$.

\begin{definition}
\label{def:integrable}
A {\em $4$-dimensional completely integrable system} is a triple $(M, \om, \mcF)$ consisting of a $4$-dimensional symplectic manifold $(M, \om)$ and a smooth map $\mcF=(\mcF_1, \mcF_2): M \to \R^2$ such that the derivative $d\mcF$ has maximal rank almost everywhere and the component functions $\mcF_1$ and $\mcF_2$ Poisson commute. 
The functions $\mcF_1$ and $\mcF_2$ are often referred to as the {\em integrals} of the integrable system $(M, \om, \mcF)$ and $\mcF(M)$ as the {\em momentum polytope}.
\end{definition}

Wherever defined, the flow of $(M, \om, \mcF)$ is given by $\Phi_t^\mcF:= \Phi_{t_1}^{\mcF_1} \circ \Phi_{t_2}^{\mcF_2}$ for $t:=(t_1, t_2) \in \R^2 $ and it induces a (local) group action of $ \R^2$ on $ M$ via $t.p:= \Phi_t^\mcF(p)$ for $p \in M$. 

A point $p \in M$ is {\em regular} if the derivative $d\mcF(p)$ has maximal rank and {\em singular} otherwise. The set $\mcF^{-1}(a,b)$ is referred to as the {\em fibre} over $(a,b)\in \R^2$. The connected components of a fibre are called {\em leafs}.

There are several equivalent ways to define non-degeneracy of singular points, cf.\ Bolsinov $\&$ Fomenko~\cite{bolsinov_fomenko_2004}. For us, the following version is the most convenient.

\begin{definition}
\label{nondeg}
Let $(M, \om, \mcF=(\mcF_1, \mcF_2))$ be a $4$-dimensional completely integrable system and $p \in M$ a fixed point.
Denote by $\om_p$ the matrix of the symplectic form with respect to a chosen basis of $T_p M$ and, moreover, by $d^2\mcF_1(p)$ and $d^2\mcF_2(p)$ the matrices of the Hessians of $\mcF_1$ and $\mcF_2$ w.r.t.\ this very basis.
Then $p$ is said to be {\em non-degenerate} if $d^2\mcF_1(p)$ and $d^2\mcF_2(p)$ are linearly independent and, moreover, if there exists a linear combination of $\om_p^{-1}d^2 \mcF_1(p)$ and $\om_p^{-1}d^2 \mcF_2(p)$ that has four distinct eigenvalues.
\end{definition}

If $p$ is a singular point of rank one of a $4$-dimensional completely integrable system $(M,\om, \mcF=(\mcF_1, \mcF_2))$ then there exist constants $c_1, c_2 \in\mathbb{R}$ such that 
$$0 = c_1 d \mcF_1(p) + c_2 d\mcF_2(p).$$ 
The space $L_p:=\Span\{\mcX^{\mcF_1}(p), \mcX^{\mcF_2}(p)\} \subset T_pM$ is the tangent line through $p$ of the orbit generated by the $\mathbb{R}^2$-action. 
Denote by $L^\perp_p$ the symplectic orthogonal complement in $T_pM$ of the tangent line $L_p$. Moreover, keep in mind that $L_p\subset L_p^\perp$. Since $0 = \{\mcF_1, \mcF_2\}$ both integrals are invariant under the $\R^2$-action. Therefore $c_1 d^2\mcF_1(p) + c_2 d^2\mcF_2(p)$ descends to the quotient $L_p^\perp / L_p$.

\begin{definition}
A singular point $p$ of rank one of a $4$-dimensional completely integrable system $(M, \om, \mcF=(\mcF_1, \mcF_2))$ is {\em non-degenerate} if the expression $c_1 d^2\mcF_1(p) + c_2 d^2\mcF_2(p)$ is invertible on the quotient $L^\perp_p/L_p$.
\end{definition}

The following local normal form and its generalisations and/or specializations were established over the years by Colin de Verdi\`ere $\&$ Vey~\cite{Verdire1979LeLD}, R\"ussmann~\cite{ruessmann},
Vey~\cite{Vey1978SurCS}, Eliasson~\cite{eliasson1984HamiltonianSW, eliasson1990}, Dufour \& Molino~\cite{PDML_1988___1B_161_0}, Miranda \& \vungoc~\cite{mirandasan}, V\~{u} Ng\d{o}c \& Wacheux~\cite{vungoc-wacheux}, Chaperon~\cite{chaperon}, and Miranda \& Zung~\cite{Miranda-zung}.

\begin{theorem}[Local normal form for non-degenerate singularities]
\label{th:locNF}
 Consider a $4$-dimensional completely integrable system $(M,\om,\mcF=(\mcF_1, \mcF_2))$ and let $p\in M$ be a non-degenerate singular point. Then
 \begin{enumerate}
  \item 
  there exists an open neighbourhood $U\subset M$ of $p$ and, on it, local symplectic coordinates $(x_1, x_2, \xi_1,\xi_2)$ and smooth functions $q_1,q_2 : U\to\R$ such that $p$ corresponds to the origin in these coordinates and $\{q_i,\mcF_j\}=0$ for all $i,j \in \{1, 2\}$ where $q_1$ and $q_2$ stem from the following list:
 \begin{itemize}
  \item 
  $q_i = (x_i^2+\xi_i^2)/2$ {\em (elliptic component),}
  \item 
  $q_i = x_i \xi_i$ {\em (hyperbolic component),}
  \item 
  $q_i = x_i \xi_{i+1} - x_{i+1} \xi_i$ and $q_{i+1} = x_i \xi_i + x_{i+1} \xi_{i+1}$ {\em (focus-focus component),}
  \item 
  $q_i = \xi_i$ {\em (regular component).}
 \end{itemize}
 
  \item\label{itemLocNF} 
  If there are no hyperbolic components, then the equations $\{q_i,\mcF_j\}=0$ for $ i,j \in \{1, 2\}$ are equivalent to the existence of a local diffeomorphism $g : \R^2\to\R^2$ such that
 $$
  g \circ \mcF = (q_1, q_2) \circ (x_1, x_2, \xi_1,\xi_2).
 $$
 \end{enumerate}
\end{theorem}

Denote by $n_E$, $n_H$ and $ n_{FF}$ the number of elliptic, hyperbolic, and focus-focus components. Then the triple $(n_E, n_H, n_{FF})$ locally classifies a non-degenerate singular point and is referred to as the \emph{Williamson type} of this non-degenerate singular point. 
Thus we conclude that, on $4$-dimensional manifolds, there are exactly six different types of non-degenerate singular points possible:
\begin{itemize} 
 \item
 {\em rank 0:} 
 elliptic-elliptic,
 focus-focus,
 hyperbolic-hyperbolic,
 hyperbolic-elliptic.
 \item
 {\em rank 1:}
 elliptic-regular,
 hyperbolic-regular.
\end{itemize}
It is important to note that the type of a non-degenerate fixed point $p$ can actually also be determined via its eigenvalues (see for instance Bolsinov $\&$ Fomenko \cite[Theorem 1.3]{bolsinov_fomenko_2004}):
Let $\lam_1, \lam_2, \lam_3, \lam_4$ be the distinct eigenvalues of $c_1 \om_p^{-1}d^2\mcF_1(p) + c_2 \om_p^{-1}d^2\mcF_2(p)$, then the type of $p$ is determined as follows:
\begin{itemize}
\item 
\emph{elliptic-elliptic:} $\{\lambda_1, \lambda_2\} = \{\pm i \alpha\}$ and $\{\lambda_3, \lambda_4 \} = \{\pm i\beta\}$,
\item 
\emph{elliptic-hyperbolic:} $\{\lambda_1, \lambda_2\} = \{\pm i \alpha\}$ and $\{\lambda_3, \lambda_4\} = \{\pm \beta\}$,
\item 
\emph{hyperbolic-hyperbolic:} $\{\lambda_1, \lambda_2\} = \{\pm \alpha\}$ and $\{\lambda_3, \lambda_4\} = \{\pm \beta\}$,
\item 
\emph{focus-focus:} $\{\lambda_1, \lam_2, \lam_3, \lam_4\} =  \{\pm  \alpha \pm i\beta\}$
\end{itemize}
with $\alpha, \beta \in \mathbb{R}^{\neq 0}$ and $\alpha \neq \beta$ for the elliptic-elliptic and hyperbolic-hyperbolic cases.


\subsection{Toric systems}
Recall from group theory that a group action is {\em effective} (or {\em faithful}) if the neutral element is the only one that acts trivially. The most accessible class of completely integrable systems is the following.

\begin{definition}
\label{def:toric}
A 4-dimensional completely integrable system $(M,\omega,\mathcal{F})$ is {\em toric} if the flow of $\mathcal{F}$ generates an effective 2-torus action on $M$.
\end{definition}

We will see below that toric systems on compact connected manifolds admit a very nice classification. But to state it properly we first need some notation.

\begin{definition}
A convex polytope $\De \subset \mathbb{R}^2$ is said to be a {\em Delzant polytope} if
\begin{enumerate}
    \item 
    $\De$ is {\em simple}, i.e.\ exactly two edges meet at each vertex.
    
    \item 
    $\De$ is {\em rational}, i.e., all edges have rational slope, meaning, they are of the form $v+bt$  where $v\in \R^2$ is the vertex, $b \in \mathbb{Z}^2$ the directional vector of the given edge, and $t \in \R$ the parameter.
    
    \item 
    $\De$ is {\em smooth}, i.e.\ at each vertex, the directional vectors of the meeting edges form a basis for $\mathbb{Z}^2$.
\end{enumerate}
\end{definition}

The classification of compact toric systems is surprisingly straightforward:

\begin{theorem}[Delzant, \cite{Delzant1988HamiltoniensPE}]
\label{theo:delzant}
Up to symplectic equivariance, any toric system $(M,\omega,\mathcal{F})$ on a compact connected symplectic 4-dimensional manifold $(M, \om)$ is determined by $\mcF(M)$ which is in fact a Delzant polytope. Conversely, for any Delzant polytope $\De$, there exists a compact connected symplectic 4-dimensional manifold $(M, \om)$ and a momentum map $\mcF: M \to \R^2$ such that $(M,\omega,\mathcal{F})$ is toric with $\mcF(M)=\De$.
\end{theorem}

Thus a toric system is determined by a {\em finite} number of points, namely the (coordinates of the) vertices of the image of the momentum map. The construction of the toric system from a Delzant polytope is explicit and not very difficult. We will use it later in Section \ref{sec:theoctsyst}.


\subsection{Semitoric systems}
The following class of completely integrable systems is a natural generalisation of toric systems in dimension four.

\begin{definition}
\label{def:semitoric}
A 4-dimensional completely integrable system $(M,\om, \mcF=(\mcF_1, \mcF_2))$ is {\em semitoric} if
\begin{enumerate}
 \item 
 $\mcF_1$ is proper and generates an effective $\mbS^1$-action on $M$.
 \item
 All singular points of $\mcF=(\mcF_1, \mcF_2)$ are non-degenerate and do not include hyperbolic components.
\end{enumerate}
\end{definition}

Semitoric systems are much more general than toric systems and their behavior is much more complicated due to the presence of focus-focus singularities which cannot occur in toric systems.
In particular, semitoric systems usually depend on an infinite amount of data.
Pelayo \& V\~{u} Ng\d{o}c~\cite{Pelayo2009SemitoricIS,Pelayo2009ConstructingIS} and Palmer \& Pelayo \& Tang~\cite{palmpelaytangsemitoric} achieved a symplectic classification of semitoric systems, hereby generalising Delzant's~\cite{Delzant1988HamiltoniensPE} toric classification.
The semitoric classification is valid on compact as well as non-compact manifolds.


\subsection{Hypersemitoric systems}
The class of systems described in this section is a natural generalisation of semitoric systems.

Intuitively, a parabolic degenerate point can be seen as a singular point where the rank of all relevant derivatives is as maximal as possible without rendering the point non-degenerate. 
A parabolic orbit is the image of a parabolic point under the flow.
Important for us is the geometric interpretation described in the following smooth local normal form. For the original abstract definition of parabolic points, we refer the reader to Bolsinov $\&$ Guglielmi $\&$ Kudryavtseva~\cite[Definition 2.1]{Bolsinov2018}.

\begin{theorem}[{Kudryavtseva $\&$ Martynchuck \cite[Theorem 3.1]{Kudryavtseva_2021}}]
\label{th:LNFparaPoint}
Let $(M, \om, \mathcal{F}=(\mcF_1, \mcF_2))$ be a 4-dimensional completely integrable system with a parabolic orbit $\alpha$. Then there exist
\begin{enumerate}
    \item 
    a small neighbourhood $U \subseteq M$ of $\alpha$ diffeomorphic to $\D^3 \times \mathbb{S}^1$ where $\D^3$ is the open unit ball in $\R^3$,
    
    \item 
    smooth functions $\mcJ, \mcH: U \to \R$ that are constant on the leafs of $\mathcal{F}$,
    
    \item 
    coordinates $(x,y,\lambda,\theta)$ on $\D^3 \times \mathbb{S}^1$ such that
    \begin{align*}
          \mcJ (x,y,\lambda,\theta)= \lambda  \quad \mbox{and} \quad \mcH (x,y,\lambda,\theta)= x^2-y^3+\lambda y   
    \end{align*}   
    and, moreover, the symplectic form becomes
    \begin{align*}
          \omega = g(x, y, \lambda)dx \wedge dy + d\lambda \wedge \bigl(d\theta + A(x, y, \lambda)dx + B(x, y, \lambda)dy\bigr)
    \end{align*}
    where $g$, $A$ and $B$ are smooth functions.
\end{enumerate}
\end{theorem}
Note that the level set $(\mcJ,\mcH)^{-1}(0,0)$ of $(\mcJ, \mcH)$ from Theorem \ref{th:LNFparaPoint} is locally homeomorphic to the geometric shape given by the letter $V$ times a circle which motivates the notion of `cusp' for such degenerate singularities. The local shape as $V$ can be observed in Figure \ref{fig:parabolicExample} in a small neighbourhood of the blue `cusp point'.

\begin{figure}[h]
\centering
\includegraphics[scale =.8]{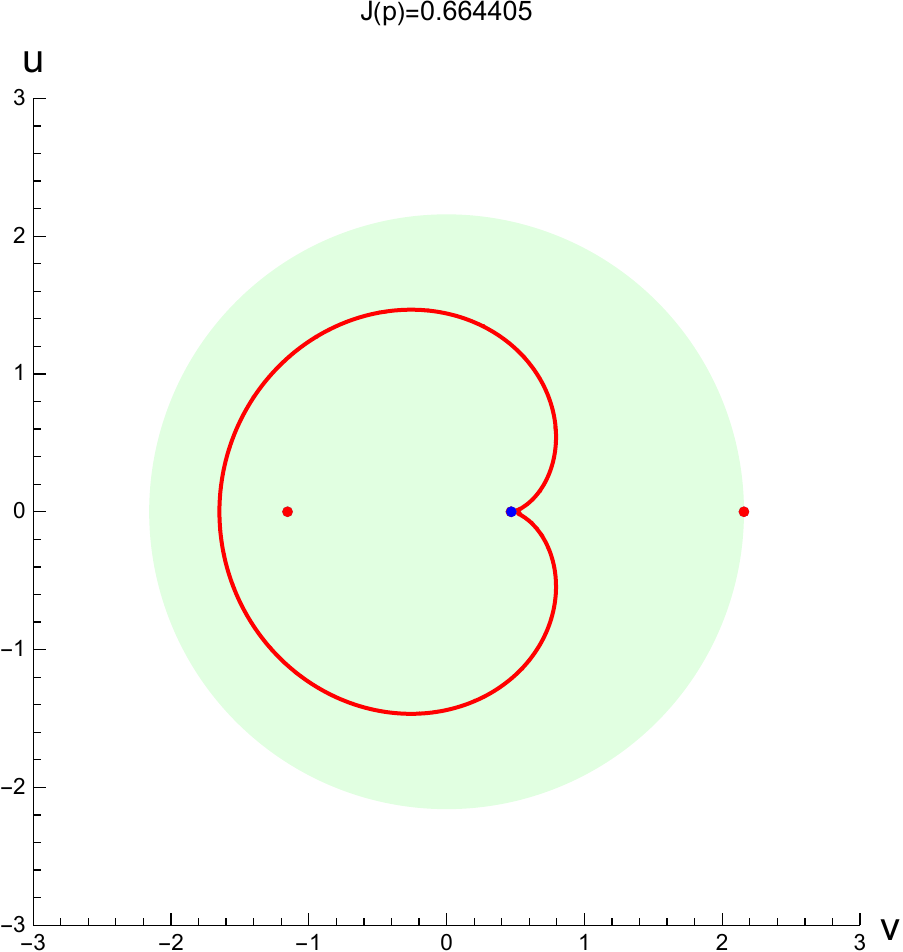}
\caption{The blue point $p$ at the `cusp' of the red curve corresponds to a parabolic point after dividing by the local $\mbS^1$-action. The shape of the red curve near the parabolic point motivates the notion of `cusps' or `cuspidal points' for parabolic points. This plot is done at $t=\left(\frac{1}{2}, \frac{1}{2}, \frac{1}{3}, 1 \right)$ and $J(p)=0.664405=:j$ for the system $(J, H_t)$ introduced and discussed in Section \ref{sec:defSystem}.} 
\label{fig:parabolicExample}
\end{figure}

Now we consider

\begin{definition}[Hohloch $\&$ Palmer \cite{hohloch2021extending}]
A {\em hypersemitoric system} is a 4-dimensional completely integrable system $(M,\omega,\mcF=(J,H))$ such that
\begin{enumerate}
\item 
$J: M \to \R$ is proper.
\item 
$J$ generates an effective $\mathbb{S}^1$-action.
\item 
All degenerate points of $\mcF=(J, H): M \to \R^2$ (if any) are parabolic.
\end{enumerate}
\end{definition}

Hypersemitoric systems are significantly more general than semitoric systems since they admit the full range of combinations of non-degenerate singular components except for hyperbolic-hyperbolic points that are prevented by the existence of the global $\mbS^1$-action (see for instance Hohloch $\&$ Palmer \cite{hohloch2021extending}). The occurrence of parabolic degenerate points in hypersemitoric systems was admitted since they appear naturally in combination with certain hyperbolic-regular points.


\subsection{Proper $\mathbb{S}^1$-systems}

Completing the line of generalisations from toric via semitoric to hypersemitoric systems, we are now dropping the assumption on the type of degeneracy of singular points:

\begin{definition}
A 4-dimensional completely integrable system $(M,\omega,\mathcal{F} = (J,H))$ is a {\em proper $\mathbb{S}^1$-system} if
\begin{enumerate}
    \item 
    $J: M \to \R$ is proper,
    \item 
    $J$ generates an effective $\mathbb{S}^1$-action on $M$.
\end{enumerate}
\end{definition}

The topological structure of the connected components of hyperbolic-regular fibres turned out to be completely determined by the quotient under the $\mbS^1$-action induced by $J$. In the following, we briefly recall the necessary notions for the construction.

\begin{definition}
An {\em undirected generalised bouquet $(G, S,C)$} consists of a compact topological space $G$ and finite disjoint subsets $S, C \subset G$ with $S \cup\ C \neq \emptyset$ such that
\begin{enumerate}
\item
$G \setminus \{ S \cup\ C\}$ is a smooth 1-dimensional manifold.

\item

For all $p \in S$ there exists a small open neighbourhood in $G$ which is homeomorphic to $\{ (x, y) \in \R^2 \mid xy =0, \ x\geq 0, \ y \geq 0 \}$ wherein $p$ corresponds to $(0, 0)$.

\item

For all $p \in C$ there exists a small open neighbourhood in $G$ which is homeomorphic to $\{ (x, y) \in \R^2 \mid xy =0\}$ wherein $p$ corresponds to $(0, 0)$.

\end{enumerate}
\end{definition}

Geometrically, $G \setminus \{ S \cup\ C\}$ is diffeomorphic to a disjoint union of open intervals. Moreover, near points of $S$, $G$ looks like a `corner' and, near points of $C$, like a `cross'. An example is sketched in Figure \ref{fig:teardrop}.

\begin{definition}
\label{de:genbouqfib1}
Let $(G,S,C)$ be an undirected generalised bouquet and consider the space $ G_S:= \faktor{ \bigl(G \times \{0, 1\} \bigr)} { \sim} $ where the equivalence relation $\sim$ is given by 
\begin{equation*}
(g, \si) \sim (\gti, \ti{\si}) 
\quad \IFF\  \quad
(g, \si)  = (\gti,\ti{\si}) \ \mbox{ {\bf or} } \
\si, \ti{\si} \in \{0, 1\} \mbox{ and } g, \gti \in S \mbox{ with } g= \gti.
\end{equation*}
This space carries the natural involution $T : G_S \to G_S$ given by $T(g, 0) = (g, 1)$ and $T(g, 1) = (g, 0)$ for all $g \in G$. We set $\mathcal{G}_S := \faktor{(G_S \times \R)}{ \simeq}$ where the equivalence relation $\simeq $ is given by $ (g, \si, r) \simeq (T(g, \si), r+ \pi) $.
\end{definition}

\begin{figure}[h]
    \centering
\begin{subfigure}[t]{.31\textwidth}
\centering
 \includegraphics[scale =0.15]{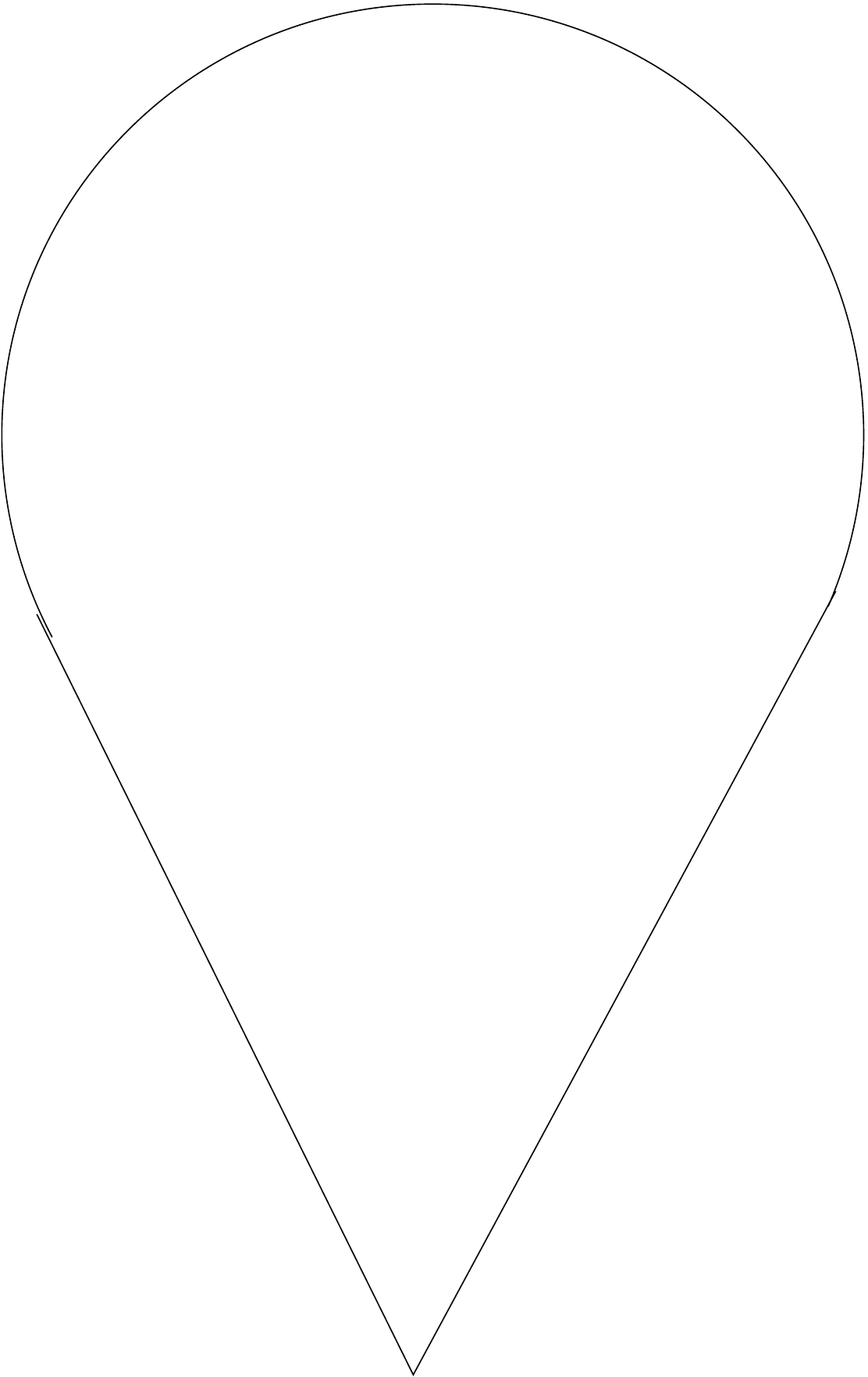}
   \caption{Teardrop.}
    \label{fig:teardrop}
 \end{subfigure}\
\begin{subfigure}[t]{.31\textwidth}
 \centering
\includegraphics[scale =0.6]{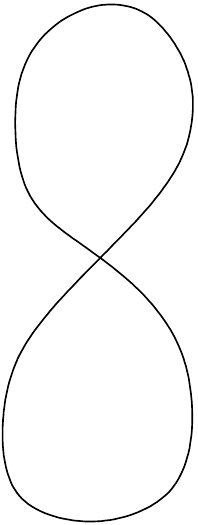}
\caption{Figure eight curve.}
 \label{subfig:FigureEight} 
\end{subfigure} \ 
\begin{subfigure}[t]{.31\textwidth}
\centering
\includegraphics[scale = .35]{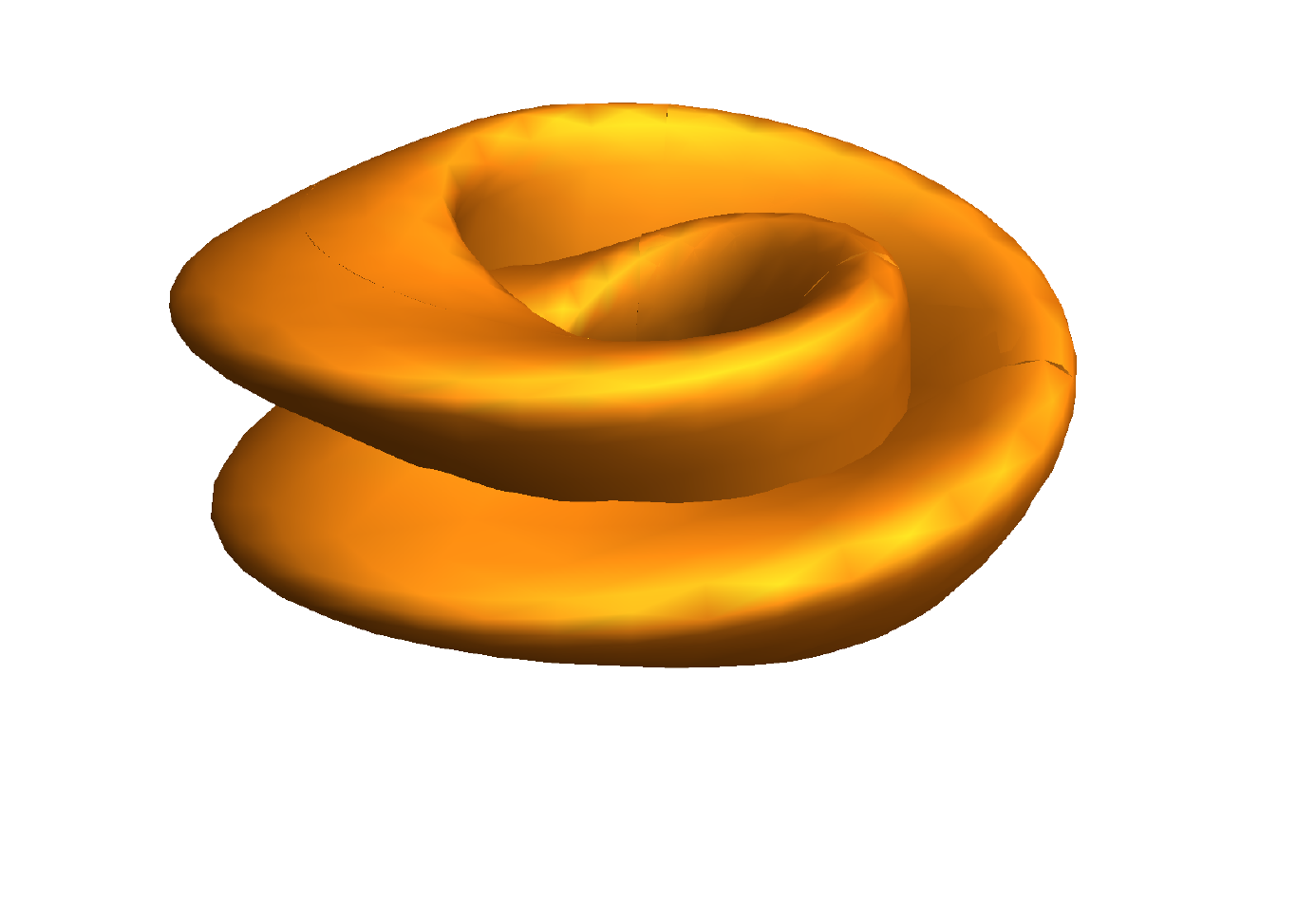} 
\caption{A twisted $2$-stacked torus.}
 \label{fig:twistDoubleTorus}
\end{subfigure}
\caption{The teardrop in (a) can be seen as a generalised bouquet $(G, S, C)$ with $C= \emptyset$ and $S$ consisting of the point at the tip of the teardrop. $G_S$ looks like a figure eight curve, see (b), and $\mcG_S$ is homeomorphic to a twisted $2$-stacked torus, see (c).}
\label{fig:exampleBouquet}
\end{figure}

As sketched in Figure \ref{fig:exampleBouquet}, the transition from $G$ to $G_S$ `turns corners into crosses' by gluing a `mirrored image' of $G$ to $G$. Passing from $G_S$ to $\mcG_S$ yields a twisted mapping torus of $G_S$. The geometric-dynamical meaning of these constructions becomes clear in the following statement:

\begin{theorem}[{Gullentops \cite[Theorem 1.9]{thesis:yannick}, Gullentops $\&$ Hohloch \cite{fiberGullandHohl}}]
\label{theo:main}
Up to homeomorphism, there exists a one-to-one correspondence between undirected generalised bouquets and hyperbolic-regular leafs of proper $\mathbb{S}^1$-systems in the following sense:
\begin{enumerate}
    \item
    Let $(M, \om, \mcF=(J, H))$ be a proper $\mathbb{S}^1$-system. 
    Given a leaf $\F$ of a hyperbolic-regular fibre then the associated undirected generalised bouquet is
    \begin{align*}
        G & :=G(\F):=\faktor{\mathbb{F}}{\mathbb{S}^1} \ \mbox{ where the quotient is taken w.r.t.\ the action induced by } J, \\
        S & := S(\F) := \left\{\left. [x] \in \faktor{\mathbb{F}}{\mathbb{S}^1} \ \right|  \    x  \in \F \text{ singular with } \Phi_\pi^J(x)=x \right\},\\
        C &: = C(\F):= \left\{\left. [x] \in \faktor{\mathbb{F}}{\mathbb{S}^1} \ \right| \ x \in \F \text{ singular with } \Phi_\pi^J(x) \neq x \right\}.
    \end{align*}
        
    \item 
    Given an undirected generalised bouquet $(G,S,C)$ there exists a proper $\mbS^1$-system having a hyperbolic-regular fibre homeomorphic to $\mathcal{G}_S$.
\end{enumerate}
\end{theorem}

Note that the content of Theorem \ref{theo:main}, on the one hand, expands and, on the other hand, partially recovers results from the two following groups of authors. 
\begin{itemize}
 \item 
Bolsinov $\&$ Fomenko \cite{bolsinov_fomenko_2004}, Bolsinov $\&$ Oshemkov \cite{Bolsinov2006SingularitiesOI} and others devised a semi-local classification of non-degenerate points of completely integrable systems based on the notion of atoms. 
 \item
Colin de Verdi\`ere $\&$ \vungoc\ \cite{Verdire2000SINGULARBR} did a semi-local classification of hyperbolic-regular leafs.
\end{itemize}


\subsection{Symplectic reduction}

Delzant's \cite{Delzant1988HamiltoniensPE} construction is based on starting with the symplectic manifold $(\mathbb{C}^n, \frac{i}{2}\sum_{j=1}^n dz_j \wedge d\bar{z}_j)$ for a certain $n \in \N$ and passing to a symplectic quotient by various actions. For the reader's convenience, let us briefly recall the necessary hypothesis and resulting statement for passing to symplectic quotients adapted to our setting: consider the following specialised version of the so-called Marsden-Weinstein theorem which describes in more generality symplectic reduction for Hamiltonianian Lie group actions.

\begin{theorem}[{Audin \cite[Proposition III.2.15]{Audintorus}}]
\label{marsdenWeinsteinToric}
Let $(M, \om)$ be a $2n$-dimensional symplectic manifold and $\mcF: M \to \mathbb{R}^m$ the momentum map of an $m$-torus action. Let $t \in \R^m$ be a regular value and assume that the $m$-torus (denoted by $\T^m$) acts freely on the regular level set $\mcF^{-1}(t)$.
Then the {\em reduced space} (or {\em symplectic quotient}) 
$$M^{red, t}:= (M /\!\!/ \mathbb{T}^m)_t :=   \mcF^{-1}(t) \slash \mathbb{T}^m$$
is a symplectic manifold with symplectic form $\om^{red, t}$. It satisfies $\tau^*\om^{red ,t}= \ka^* \om$ where $\tau: \mcF^{-1}(t) \to M^{red, t}$ denotes the quotient map and $\ka: \mcF^{-1}(t) \hookrightarrow M$ the inclusion. 
\end{theorem}


\section{The toric system constructed from the octagon}
\label{sec:theoctsyst}

In order to study transitions from elliptic-elliptic to focus-focus points and collisions of focus-focus fibres more closely, De Meulenaere $\&$ Hohloch \cite{meulenaere2019family} first constructed via Delzant's \cite{Delzant1988HamiltoniensPE} construction the toric system that has the octagon from Figure \ref{fig:octagon} as image of the momentum map. 
Then they perturbed this toric system in order to obtain a family of systems that is semitoric apart from the parameter values where the transitions and collisions take place. 

The aim of the present paper is to start with the toric system corresponding to the octagon in Figure \ref{fig:octagon} and then find perturbation terms that yield singularities with hyperbolic components and certain topological properties.

We will now briefly recall from De Meulenaere $\&$ Hohloch \cite[Section 3]{meulenaere2019family} the most important facts of the toric system constructed from the octagon in Figure \ref{fig:octagon}.


\subsection{The symplectic manifold}
Denote the octagon displayed in Figure \ref{fig:octagon} by $\De$ and the standard symplectic form on $\C^n$ by $\om_{st}:= \frac{i}{2}\sum_{j=1}^n dz_j \wedge d\bar{z}_j$.
The construction of a symplectic manifold $(M, \om)$ and a toric momentum map $\mcF:(M, \omega) \to \R^2 $ with $\mcF(M)= \De$ is done in De Meulenaere $\&$ Hohloch \cite[Section 3]{meulenaere2019family} and yields the map 
\begin{align*}
    \mathcal{L}:(\mathbb{C}^8, \om_{st}) &\rightarrow \mathbb{R}^6, 
    \qquad \qquad 
    \begin{pmatrix}
    z_1\\
    z_2\\
    z_3\\
    z_4\\
    z_5\\
    z_6\\
    z_7\\
    z_8
    \end{pmatrix}
    \mapsto \frac{-1}{2}
    \left(
    \begin{aligned}
    |z_1|^2+|z_5|^2-6\\
    |z_2|^2+|z_5|^2+|z_7|^2-10\\
    |z_3|^2+|z_7|^2 -6\\
    |z_4|^2-|z_5|^2+|z_7|^2-4\\
    |z_5|-|z_6|+|z_7|^2 -2\\
    |z_5|^2-|z_7|^2+|z_8|-4
    \end{aligned}
    \right)
\end{align*}
which generates a $\mathbb{T}^6$-action. The six equations in the definition of $\mcL$ are also referred to as {\em manifold equations}.
After applying Theorem \ref{marsdenWeinsteinToric} at level zero, we obtain the symplectic quotient  
\begin{align*}
    M:= M^{red, 0}:=   \mcL^{-1}(0) \slash \mathbb{T}^6 \quad \mbox{with} \quad \om:= \om^{red, 0}.
\end{align*}
Elements of $\mathcal{L}^{-1}(0) \subset \C^8$ are denoted by $(z_1, \dots, z_8)$ and elements of the quotient $M =\mcL^{-1}(0) \slash \mathbb{T}^6$ by $[z_1, \dots, z_8]$ unless we work with representatives for which we use again the notation $(z_1, \dots, z_8)$.

The desired momentum map $\mcF$ that generates a $\T^2$-action and satisfies $\mcF(M) = \De$ is given by $\mcF:=(J,H): (M, \om) \to \R^2$ with
\begin{align*}
    J(z_1,...,z_8) = \frac{|z_1|^2}{2} 
    \quad \mbox{and} \quad
    H(z_1,...,z_8) = \frac{|z_3|^2}{2}.
\end{align*}
We recall and summarise

\begin{theorem}[{De Meulenaere $\&$ Hohloch~\cite[Theorem 3.6]{meulenaere2019family}}]
\label{th:toricOctagonSys}
The above constructed symplectic manifold $(M, \om)$ is 4-dimensional, compact, and connected and $(M,\omega,\mathcal{F}=(J,H))$ is a toric system. In particular, $J$ generates an effective $\mathbb{S}^1$-action.
\end{theorem}


\subsection{Coordinate charts}
\label{section:coordCharts}

The manifold $M$ can be covered by explicit coordinate charts that are constructed by means of the six manifold equations and the $\T^6$-action on $\mcL^{-1}(0)$. More precisely, De Meulenaere $\&$ Hohloch~\cite[Lemma 3.3 and discussion afterward]{meulenaere2019family} show that $M =\bigcup_{\nu=1}^8 U_\nu $ where 
$$ 
U_\nu := \left\lbrace \ [z_1, \ldots, z_8] \in M \mid z_k \neq 0 \mbox{ for all } k \in \{\nu+2, \ldots, \nu+7\} \Mod 8 \ \right\rbrace 
$$
for $1 \leq \nu \leq 8$. Note that $U_\nu$ is the only subset of $M$ where $z_\nu$ {\em and} $z_{\nu+1}$ may vanish.
Since there are at least six variables nonzero among $[z_1, \dots, z_8] \in U_\nu$ one can use the $ \T^6$-action to choose strictly positive real numbers as representatives for them. If we write $z_k=x_k+i y_k$ for $1 \leq k \leq 8$, this means $y_k=0 $ and $x_k >0$ for these six representatives. Thus, for example, we may represent $U_1$ by points of the form
$$(z_1, \dots, z_8) = (x_1, y_1, x_2, y_2, x_3, 0, x_4, 0, x_5, 0, x_6, 0, x_7, 0, x_8, 0)$$
with $x_3, \dots, x_8>0$.
By means of the manifold equations and by abbreviating $\vert z_1 \vert^2 = x_1^2 + y_1^2$ and $\vert z_2 \vert^2 = x_2^2 + y_2^2$, the variables $x_3, \dots, x_8$ can be expressed in $U_1$ as functions of $x_1$, $y_1$, $x_2$, $y_2$ via 
\begin{align*}
x_3 &= \sqrt{2 - \vert z_1 \vert^2 + \vert z_2 \vert^2},   &&&   x_5 &= \sqrt{6 - \vert z_1 \vert^2}, &&&  x_7 &= \sqrt{4 + \vert z_1 \vert^2 - \vert z_2 \vert^2}, \\ 
x_4 &= \sqrt{6 - 2\vert z_1 \vert^2 + \vert z_2 \vert^2}, &&& x_6 & = \sqrt{8 - \vert z_2 \vert^2},  &&& x_8 &= \sqrt{2 + 2\vert z_1 \vert^2 - \vert z_2 \vert^2}.
\end{align*}
This leads to charts $\phi_\nu: \C^2 \to U_\nu$ for all $1 \leq \nu \leq 8$ of which we exemplarily write down the first one:
\begin{equation*}
\phi_1: \C^2 \to U_1, 
\qquad 
\phi_1(x+iy,u+iv) = 
\left(
\begin{aligned}
    & x+iy\\
    & u+iv\\
    & \sqrt{2-x^2-y^2+u^2+v^2}\\
     & \sqrt{6-2x^2-2y^2+u^2+v^2}\\
    & \sqrt{6-x^2-y^2}\\
    & \sqrt{8-u^2-v^2}\\
    & \sqrt{4+x^2+y^2-u^2-v^2}\\
    & \sqrt{2+2x^2+2y^2-u^2-v^2}
\end{aligned}
\right)
\end{equation*}
The set $U_\nu$ consists thus of those points in $\mathbb{C}^2$ for which the expressions under the square roots are strictly positive. In particular, $U_\nu$ is completely determined by its image under the `coordinate distance map' $ (x+iy,u+iv) \mapsto (|x+iy|,|u+iv|)$ plotted in Figure \ref{fig:imp1domain}.

\begin{figure}[ht]
    \centering
    \includegraphics[scale =.4]{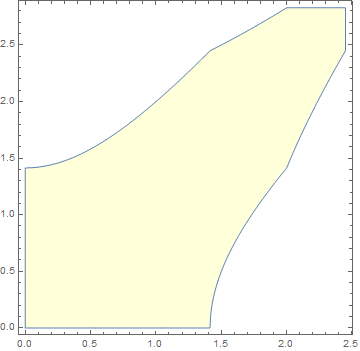}
    \caption{The `coordinate distances' of the points in $U_1$ to the origin are displayed by the image of the map $ (x+iy,u+iv) \mapsto (|x+iy|,|u+iv|)$ for all points with $\phi_1(x+iy,u+iv) \in U_1$. Later on, we refer to this image as $D_1$.}
     \label{fig:imp1domain}
\end{figure}


\begin{lemma}[{De Meulenaere $\&$ Hohloch \cite[Section 3.3]{meulenaere2019family}}]
\label{lem:sympform}
The symplectic form $\om$ on $M$ becomes the standard symplectic form in the charts $(U_\nu, \phi_\nu)$ for all $1 \leq \nu \leq 8$, more precisely, we have for all $1 \leq \nu \leq 8$
\begin{align*}
    \phi_\nu^*\omega =  dx \wedge dy + du\wedge dv. 
\end{align*}
\end{lemma}

Recall that, for all $1 \leq j, k \leq 8$, the functions $|z_j|^2$ and $|z_k|^2$ Poisson commute. Moreover, if functions Poisson commute with $|z_j|^2$ then linear combinations of them also Poisson commute with $|z_j|^2$. Thus every component of $\mathcal{L}$ Poisson commutes with $|z_j|^2$ for all $j$. 

\begin{remark}
The $\mathbb{T}^6$-action generated by $\mathcal{L}$ preserves the norm of $z_j$ for all $j$.
\end{remark}


\section{Our family of proper $\mathbb{S}^1$-systems}
\label{sec:defSystem}

The idea is to perturb the toric octagon system $\bigl(M, \om , \mcF=(J, H)\bigr)$ from Theorem \ref{th:toricOctagonSys} to obtain a family of proper $\mbS^1$-systems that display various singularities with hyperbolic components. Since $M$ is compact, $J$ is certainly proper. Moreover, it generates an effective $\mathbb{S}^1$-action. So in order to construct a suitable family of proper $\mbS^1$-systems, we keep $(M, \om)$ and $J$ as they are and only perturb $H$ to a family $(H_t)_{t \in \R^{4}}$ such that $\left\{J,H_t\right\}=0$ for all $t \in \R^4$. The family $(J,H_t)_{t \in \R^4}$ will be our candidate.


\subsection{The explicit family}
\label{subsec:defSystem}
Denote by $\Re$ the real part of a complex number.
We consider $t:=(t_1, t_2, t_3, t_4) \in \R^4$ and define the function 
\begin{align*}
   H_t: M \to \R, \qquad  H_t:= H_{(t_1,t_2,t_3,t_4)} := (1-2 t_1)H + \sum_{j = 1}^4 t_j\gamma_j 
\end{align*}
where $\gamma_j:M \rightarrow \mathbb{R}$ for $1 \leq j \leq 4$ are given by
\begin{align*}
    \gamma_1(z_1,...,z_8)   := \frac{1}{100} \ \bigl(\overline{z_2z_3z_4}z_6z_7z_8+z_2z_3z_4\overline{z_6z_7z_8}\bigr) = \frac{1}{50} \ \Re(\overline{z_2z_3z_4}z_6z_7z_8), 
\end{align*}    
and
 \begin{align*}   
     \gamma_{2}(z_1,...,z_8) & : =\frac{1}{50} \ |z_5|^4|z_4|^4,\\
    \gamma_{3}(z_1,...,z_8) & : = \frac{1}{50} \ |z_4|^4|z_7|^4,\\
    \gamma_{4}(z_1,...,z_8) & : = \frac{1}{100} \ |z_5|^4|z_7|^4.
\end{align*}

The map $\gamma_1$ is inspired by the perturbation in De Meulenaere $\&$ Hohloch \cite{meulenaere2019family} that forced certain elliptic-elliptic points to turn into focus-focus points. The maps $\ga_2$, $\ga_3$, $\ga_4$ are inspired by the work of Palmer $\&$ Le Floch \cite{floch2019semitoric}. 
Note in addition that these maps are compatible with the $\T^6$-action defined at the end of Section \ref{section:coordCharts} and therefore also descend to the symplectic quotient.

These perturbations will cause the occurence of hyperbolic-regular fibres.
The coefficients $\frac{1}{50}$ and $\frac{1}{100} $ in $\ga_1, \dots, \ga_4$ rescale the functions and parameter values to a more convenient (in particular plottable) range.

Now we will show that $F_t:=(J, H_t): (M, \om) \to \R^2 $ forms a proper $\mbS^1$-system for all $t \in \R$. We start with

\begin{proposition}\label{prop:gammacommute}
$\ga_1$, $\gamma_2,\gamma_3$ and $\gamma_4$ Poisson commute with $J$. In particular, $H_t$ also Poisson commutes with $J$ for all $t \in \R^4$. 
\end{proposition}

\begin{proof}
The pair $(J,\gamma_1)$ was already shown in De Meulenaere $\&$ Hohloch \cite[Proposition 4.6]{meulenaere2019family} to be a completely integrable system.

Since $\gamma_2$, $\ga_3$ and $\ga_4$ pass to the symplectic quotient it is enough to prove that $\gamma_2$, $\ga_3$ and $\ga_4$ commute with $J$ already on $(\mathbb{C}^8, \om_{st})$.
Recall that $\partial_{z_j} := \frac{1}{2}(\partial_x - i\partial_y)$ and $\partial_{{\zbar_j}} := \frac{1}{2}(\partial_x + i\partial_y)$ and calculate 
\begin{align*}
    \{\ga_2, J\}& = \omega_{st}(X^{\gamma_2},X^J) = \frac{i}{2}\sum_{n=1}^8 (dz_n\wedge d{\zbar_n})(X^{\gamma_2},-iz_1\partial_{z_1}+i{\zbar_1}\partial_{{\zbar_1}})\\
    &= \frac{i}{2}(dz_1\wedge d{\zbar_1})(X^{\gamma_2},-iz_1\partial_{z_1}+i{\zbar_1}\partial_{{\zbar_1}}) + \frac{i}{2}\sum_{n=2}^8 (dz_n\wedge d{\zbar_n})(X^{\gamma_2}, -iz_1\partial_{z_1}+i{\zbar_1}\partial_{{\zbar_1}})\\
    &= 0+0 = 0.
\end{align*}
The first zero is a consequence of $\gamma_2$ not depending on $z_1$. The second zero follows since the vector field $-iz_1\partial_{z_1}+i{\zbar_1}\partial_{{\zbar_1}}$ only contains $\partial_{z_1}$ and $\partial_{{\zbar_1}}$ components. 
The calculations for $\ga_3$ and $\ga_4$ are analogous.
\end{proof}

It remains to show 

\begin{proposition}\label{prop:gammadependend}
For all $t \in \R^4$, the Hamiltonian vector fields $X^J$ and $X^{H_t}$ are linearly independent almost everywhere.
\end{proposition}

\begin{proof}
W.l.o.g.\ we will proof it explicitly only for the chart $(U_1 , \phi_1)$. For $1 \leq k \leq 4$, denote by $X^{J \circ \phi_1}_k$ and $X_k^{H_t \circ \phi_1}$ the $k$th coordinate component of the Hamiltonian vector fields $X^{H_t \circ \phi_1}$ and $X^{J \circ \phi_1}$ and consider, for $1 \leq j < k \leq 4$, the functions 
\begin{equation*}
    f_{jk}: U_1  \rightarrow \mathbb{R}, 
    \qquad 
    x \mapsto
    \det
    \begin{pmatrix}
    X^{J \circ \phi_1}_j(x) & X^{H_t \circ \phi_1}_j(x) \\
     X^{J \circ \phi_1}_k(x)  & X^{H_t \circ \phi_1}_k(x)
    \end{pmatrix}.  
\end{equation*}
Since $X^{J \circ \phi_1}$ and $X^{H_t \circ \phi_1}$ are linearly dependent on a subset of $f^{-1}_{jk}(0)$ it is enough to show that there exists $j,k$ such that $f_{jk}^{-1}(0)$ has measure zero w.r.t.\ the Lebesgue measure. $f_{12}$ and $f_{34}$ vanish, but the others do not. Thus, apart from the combinations $(j,k) \in \{ (1, 2), (3, 4)\}$ we can consider any of the functions $f_{jk}$. Let us start with $f_{14}$. As it turns out, we then need not consider the others. 

More precisely, we will now show that $ f^{-1}_{14}(0) $ is a zero set w.r.t.\ the Lebesgue measure. 
Recall that the zero set of a nonzero analytic function has Lebesgue measure zero. The function $f_{14}$ is analytic on $U_1$ since it consists of polynomials and square roots in $x, y, u, v$ where the square roots are evaluated away from their poles. It remains to show that no choice of the parameters $t_1, \dots, t_4$ can make $f_{14}$ vanish completely. Arguing by contradiction, assume that there is $(t_1,...,t_4) \in \mathbb{R}^4$ such that $f_{14} \equiv 0$. By plugging the point $\left(0, \frac{1}{2}, 0, 0\right) $ into $f_{14}$ we obtain
$$
0 =  f_{14}\left(0, \frac{1}{2}, 0, 0\right)  =  - \frac{t_1}{40} \sqrt{\frac{1309}{10}} .
$$
Then plugging the (admissible) points $$
(x, y, u, v) \in \left\{   \left(\frac{1}{2}, \frac{1}{2}, \frac{1}{2}, 0 \right), \left(0, \frac{1}{2}, \frac{1}{3}, 0 \right), \left(0, \frac{1}{2}, \frac{1}{4}, 0\right), \left(0, \frac{1}{2}, \frac{1}{5}, 0 \right) \right\}
$$
into $f_{14}$ we get via a short {\em Mathematica} calculation
\begin{align*}
 f_{14} \left(\frac{1}{2}, \frac{1}{2}, \frac{1}{2}, 0 \right) & =  \frac{-400 - 42 t_2 + 714 t_3 + 
  17 t_4}{1600} , \\
 f_{14}\left(0, \frac{1}{2}, \frac{1}{3}, 0 \right) & = \frac{-1166400 - 
  32724 t_2 + 3190388 t_3 + 12069 t_4}{6998400} , \\
 f_{14}\left(0, \frac{1}{2}, \frac{1}{4}, 0\right) & =  \frac{-25600 - 712 t_2 + 
  65593 t_3 + 268 t_4}{204800} ,\\
 f_{14}\left(0, \frac{1}{2}, \frac{1}{5}, 0 \right)& = \frac{-692500 t_2 + 
  62040244 t_3 + 625 (-40000 + 421 t_4)}{250000000} .
\end{align*}
We verify with {\em Mathematica} that there are no values of $t_1, \dots, t_4$ that satisfy $f_{14} \equiv 0$ under the above conditions. Thus $f_{14} $ cannot be identically zero. 
Therefore $X^{H_t}$ and $X^J$ are linearly independent almost everywhere on $M$. 
\end{proof}


\subsection{Some technical properties}

This subsection consists of two technical results that we need later. First consider

\begin{lemma}
\label{lem:omega}
If $\gamma_1(\phi_1(x+iy,u+iv)) = 0$ for $ (x+iy,u+iv) \in U_1$ then $u = 0$.
\end{lemma}

\begin{proof}
First, recall that, on $U_1$,
$$ \ga_1(z_1, \dots, z_8) = \frac{1}{50} \ \Re(\overline{z_2z_3z_4}z_6z_7z_8) \quad \mbox{and} \quad (z_1, \dots, z_8) = \phi_1(x+iy,u+iv). $$ 
Second, recall that our preferred representatives of $z_3,...,z_8$ are in fact real on $U_1$. This implies
$$ 
\ga_1(z_1, \dots, z_8) = \frac{1}{50} \ \Re(\overline{z_2z_3z_4}z_6z_7z_8) = \frac{1}{50} z_3z_4 z_6z_7z_8\Re(\overline{z_2}).
$$
For a product to vanish, one of the factors needs to be zero. Thus if $\gamma_1(\phi_1(x+iy,u+iv)) = 0$, one of the $\Re(\overline{z_2})$, $z_3$, $z_4$, $z_6$, $z_7$, $z_8$ must vanish. By definition of $U_1$, de factors $z_3$, $z_4$, $z_6$, $z_7$, $z_8$ are strictly greater than zero and only $z_1$ and/or $z_2$ may vanish. Thus $u=\Re(\overline{z_2}) $ must vanish.
\end{proof}

When constructing the system $(J,H_{t})$, we took $J$ unaltered from the toric system $(J,H)$. Thus the system $(J, H_t)$ always contains a global $\mathbb{S}^1$-action. 
We now show that all regular points of $J$ have the same period:

\begin{theorem}
\label{the:equalperiod}
For all $p \in M$ with $dJ(p) \neq 0 $ the periods of regular points of $J$ coincide, i.e., we have $2 \pi = P^J(p) := \inf\left\{t>0 \ \left| \ \Phi^J_t(p)=p\right. \right\}$.  
\end{theorem}

\begin{proof}
Let us assume for simplicity that $p$ lies on the coordinate chart $(U_1, \phi_1)$, all other cases are analogous. 
Whenever 
$$d(J \circ \phi_1)(x+iy,u+iv) = xdx+ydy \neq 0, $$ 
we have $\abs{x+iy}^2 =x^2+y^2 \neq 0$ and thus $z_1= x +iy \neq 0$.
Denote the domain in Figure \ref{fig:imp1domain} (that is associated with $U_1$) by $D_1$ and consider polar coordinates given by 
\begin{equation}
 \label{eq:polarcoordinates}
 \begin{aligned}
  &  \EuScript{P}_1 : \  D_1 \ \times  \ ]-\pi,\pi[ \ \times\ ]-\pi,\pi[ \ \to U_1, \\
  & \EuScript{P}_1(r_1,r_2,\theta_1,\theta_2) := \bigl(r_1 cos(\theta_1)+i r_1 sin(\theta_1),\ r_2 cos(\theta_2)+i r_2 sin(\theta_2)\bigr).
  \end{aligned}
\end{equation}
In these coordinates, the symplectic form becomes
$$r_1dr_1 \wedge d\theta_1+r_2dr_2 \wedge d\theta_2.$$ 
Evaluating the equation
\begin{align*}
     d(J \circ \phi_1 \circ \EuScript{P}_1) (\ \cdot\ ) = (r_1dr_1 \wedge d\theta_1+r_2dr_2 \wedge d\theta_2)(X^{J\circ \phi_1 \circ \EuScript{P}_1},\cdot\ )
\end{align*}
on $\partial_{r_1}$, $ \partial_{\theta_1}$, $\partial_{r_2}$ and $\partial_{\theta_2}$ yields
 \begin{align*}
     &d (J\circ \phi_1 \circ \EuScript{P}_1)(\partial_{r_1})  =r_1d\theta_1(X^{J\circ \phi_1 \circ \EuScript{P}_1 }) = r_1, && d (J\circ \phi_1 \circ \EuScript{P}_1)(\partial_{r_2})  = r_2d\theta_2(X^{J\circ \phi_1 \circ \EuScript{P}_1})  = 0,\\
      & d (J\circ \phi_1 \circ \EuScript{P}_1)(\partial_{\theta_1})  = -r_1dr_1(X^{J\circ \phi_1 \circ \EuScript{P}_1 } ) = 0, && d (J\circ \phi_1 \circ \EuScript{P}_1)(\partial_{\theta_2})  =-r_2dr_2(X^{J\circ \phi_1 \circ \EuScript{P}_1 }) =  0.
 \end{align*}
which implies $X^ {J\circ \phi_1 \circ \EuScript{P}_1 } = \partial_{\theta_1}$. This allows us to calculate the flow map of $J$ in polar coordinates explicitly as
$$
\Phi^{J\circ \phi_1 \circ \EuScript{P}_1 }_t(r_1,r_2,\theta_1,\theta_2) = (r_1,r_2,\theta_1 +t,\theta_2)
$$
or, equivalently,
$$
\Phi^{J\circ \phi_1  }_t(z_1, z_2) = ( e^{it} z_1, z_2).
$$
Therefore, if $dJ(p) \neq 0$ for some $p \in U_1$, we find for the period $P^J(p) = 2 \pi $.
\end{proof}


\section{Singular points and the reduced manifold}
\label{sec:redmanicritpoints}

In this section, we will calculate the singular points on the proper $\mathbb{S}^1$-system $\left(J,H_t\right)$ for certain $t \in \mathbb{R}^4$. In Section \ref{sec:singbif}, we will observe how these points change when $t\in \R^4$ varies.


\subsection{Criteria for rank zero singular points}

Recall that the toric octagon system $(J, H)$ has eight singular points of rank zero, namely the preimages of the vertices of the octagon. Moreover, each semitoric system of the semitoric transition family in De Meulenaere $\&$ Hohloch \cite{meulenaere2019family} also has exactly eight singular points of rank zero, of which four undergo bifurcations from elliptic-elliptic to focus-focus and back when the perturbation parameter passes from zero to one.

The following statements will show that the rank zero singular points can only appear in the preimage of the blue points and lines in Figure \ref{fig:sides}.

\begin{figure}[ht]
    \centering
    \includegraphics{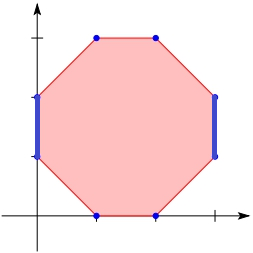}
    \caption{Rank zero singular points can only appear in the preimage of the blue points and lines which are given by $(1,0)$, $(1,3)$, $(2,0)$, $(2,3)$ and $\{0\} \times [1,2]$ and  $\{3\} \times [1, 2]$.}
    \label{fig:sides}
\end{figure}

\begin{theorem}
\label{the:invariablesingpoint}
For all $t \in \R^4$, the points 
\begin{equation*}
\left\{ \
\begin{aligned}
 \quad \phi_2(0,0) & = \left[\sqrt{2}, \ 0, \ 0,\sqrt{2}, \ 2, \ 2\sqrt{2},\sqrt{6},\sqrt{6}\right], \\
 \quad \phi_3(0,0) & = \left[2,\sqrt{2}, \ 0, \ 0,\sqrt{2},\sqrt{6},\sqrt{6}, \ 2\sqrt{2}\right], \\
 \quad \phi_6(0,0) & = \left[2, \ 2\sqrt{2},\sqrt{6},\sqrt{6},\sqrt{2}, \ 0, \ 0,\sqrt{2} \right], \\
 \quad \phi_7(0,0) & = \left[\sqrt{2},\sqrt{6},\sqrt{6}, \ 2\sqrt{2}, \ 2,\sqrt{2}, \ 0, \ 0 \right]
\end{aligned} 
\right.
\end{equation*}
are singular points of rank zero of $(J,H_t)$ with values 
$$
J\left (\left [\sqrt{2}, \ 0, \ 0,\sqrt{2}, \ 2, \ 2\sqrt{2},\sqrt{6},\sqrt{6} \right ] \right)=1= J \left(\left [\sqrt{2},\sqrt{6},\sqrt{6}, \ 2\sqrt{2}, \ 2,\sqrt{2}, \ 0, \ 0 \right] \right)
$$
and 
$$
J \left(\left [ 2, \ 2\sqrt{2},\sqrt{6},\sqrt{6},\sqrt{2}, \ 0, \ 0,\sqrt{2} \right ] \right)= 2 = J \left ( \left [2,\sqrt{2}, \ 0, \ 0,\sqrt{2},\sqrt{6},\sqrt{6}, \ 2\sqrt{2} \right ] \right ).
$$
They are referred to as the {\em invariant singular points} since their coordinates do not depend on the parameter $t$.
\end{theorem}

\begin{proof}
We will exemplarily prove it for $(U_2,\phi_2)$. Recall $ J(z_1,...,z_8) = \frac{|z_1|^2}{2}$ and
\begin{align*}
       \phi_2(x+ iy,u+ iv) &= 
    \begin{pmatrix}
    \sqrt{2 + x^2 + y^2 - u^2 - v^2}\\
    x+iy\\
    u+iv\\
    \sqrt{2 - x^2 - y^2 + 2 u^2 + 2 v^2}\\
    \sqrt{4 - x^2 - y^2 + u^2 + v^2}\\
    \sqrt{8 - x^2 - y^2}\\
    \sqrt{6 - u^2 - v^2}\\
    \sqrt{6 + x^2 + y^2 - 2 u^2 - 2 v^2}
    \end{pmatrix}.
\end{align*}
For $(x+iy,u+iv) \in \C^2$, calculate 
\begin{align*}
    J(\phi_2(x+iy,u+iv)) & = \frac{1}{2} (2+x^2+y^2-u^2-v^2), \\
    d(J\circ \phi_2)(x+iy,u+iv) & = (x, \ y, \ -u, \ -v).
\end{align*}
Thus $dJ =0$ is equivalent with $x=y=u=v=0$. Thus the origin $(0, 0) \in \C^2$ is the only singular point of the chart $\phi_2$.
The argument for the charts $(U_3,\phi_3)$, $(U_6,\phi_6)$ and $(U_7,\phi_7)$ is similar.
Therefore the rank of $\phi_\ell(0,0)$ for $\ell \in \{2, 3, 6, 7\}$ can be maximally one.
Now we show that it is actually rank zero, i.e., we also have $d(H_t \circ \phi_\ell)(0,0) =0$ for $\ell \in \{2, 3, 6, 7\}$. 
We give the argument in detail only for $\ell = 2$ since the other cases are similar. From Section \ref{sec:defSystem}, recall
\begin{align*}
    H_t:= H_{(t_1,t_2,t_3,t_4)} = (1-2 t_1)H + \sum_{n = 1}^4 t_n \gamma_n 
\end{align*}
and the definitions of $\ga_1$, $\ga_2$, $\ga_3$ and $\ga_4$. We calculate
\begin{align*}
    (\gamma_1 \circ \phi_2)(x+iy,u+iv) &= \frac{1}{50}\Re(\overline{z_2z_3})z_4z_6z_7z_8 = \frac{1}{50}(ux-vy)z_4z_6z_7z_8\\
    &= ux \ f(x+iy,u+iv) -vy \ g(x+iy,u+iv)
\end{align*}
where $f, g : U_2 \to \R$ are suitable smooth functions.
Now we compute the partial derivatives for $x, y, u, v$ of $(\gamma_1 \circ \phi_2)$ in $(x+iy,u+iv)$:
\begin{align*}
    \partial_x (\gamma_1 \circ \phi_2) &= u f + ux \ \partial_xf - vy \ \partial_xg,
    &&& 
    \partial_u (\gamma_1 \circ \phi_2) & =  x f + ux \ \partial_uf - vy \ \partial_ug, \\
    \partial_y (\gamma_1 \circ \phi_2) & = ux \ \partial_uf - v g - vy \ \partial_yg ,  
    &&&
    \partial_v (\gamma_1 \circ \phi_2) &=  ux \ \partial_vf - y g - vy \ \partial_vg.
\end{align*}
Evaluating at $(x, y, u, v) = (0,0,0,0)= (0+i0, 0+i0)$, we find for all $k \in \{x, y, u, v\}$
\begin{align*}
    \partial_k (\gamma_1 \circ \phi_2)(0+i 0,0+ i0) &= 0.
\end{align*}
Now calculate
\begin{align*}
 (\gamma_2 \circ \phi_2)(x+iy,u+iv) & = \frac{1}{50}(2-x^2-y^2+2u^2+2v^2)^2(4-x^2-y^2+u^2+v^2)^2, \\
 (\gamma_3 \circ \phi_2)(x+iy,u+iv) & =  \frac{1}{50}(2-x^2-y^2+2u^2+2v^2)^2(6-u^2-v^2)^2, \\
 (\gamma_4 \circ \phi_2)(x+iy,u+iv) & =  \frac{1}{100}(4-x^2-y^2+u^2+v^2)^2(6-u^2-v^2)^2.
\end{align*}
Note that these formulas depend only on $x^2+y^2$ and $u^2+v^2$. Therefore calculating the derivatives for $k \in \{x, y, u, v\}$ and evaluating them at $(x , y, u, v) = (0,0,0,0)= (0+i0, 0+i0)$ yields
\begin{align*}
     \partial_k (\gamma_k \circ \phi_2)(0+0i,0+0i) &= 0
\end{align*}     
for all $k \in \{x, y, u, v\}$. 
It remains to consider
$$
(H \circ \phi_2) (x+iy, u+iv) = \frac{u^2+v^2}{2}.
$$
We compute
\begin{align*}
 \partial_x (H \circ \phi_2)  = 0, \quad 
 \partial_y (H \circ \phi_2)   = 0, \quad 
 \partial_u (H \circ \phi_2)   = u, \quad 
 \partial_v (H \circ \phi_2)   = v.
\end{align*}
Evaluating at $(x, y, u, v) = (0,0,0,0)= (0+i0, 0+i0)$, we find for all $k \in \{x, y, u, v\}$
\begin{align*}
    \partial_k (H \circ \phi_2)(0+i 0,0+i 0) &= 0.
\end{align*}
Altogether we therefore obtain for all $k \in \{x, y, u, v\}$
\begin{align*}
    \partial_k  (H_{t} \circ \phi_2) (0+i 0,0+i 0)&  = 0.
\end{align*}
\end{proof}

Note that Theorem \ref{the:invariablesingpoint} shows a property of $(J, H_t)$ similar to the one in Le Floch $\&$ Palmer \cite[Lemma 3.2]{floch2019semitoric}.

Now we are interested in the whereabouts of the singular points of $(J, H_t)$ that are not invariant.

\begin{theorem}\label{the:zerocor}
If a rank zero singular point $[z_1, \dots, z_8]$ of $(J, H_t)$ is not an invariant singular point then either its first or its fifth coordinate is zero. If the first coordinate is zero then its value under $J$ is zero and if the fifth coordinate is zero then its value under $J$ is three.
\end{theorem}

\begin{proof}
Given a singular point of rank zero, the proof of Theorem \ref{the:invariablesingpoint} showed that the only points in $U_2$, $U_3$, $U_6$ and $U_7$ with $dJ=0$ are the invariants singular points listed in the statement of Theorem \ref{the:invariablesingpoint}. Thus any further singular points of rank zero must lie in $(U_1, \phi_1)$, $(U_4, \phi_4)$, $(U_5, \phi_5)$ and/or $(U_8, \phi_8)$.
Recall $ J(z_1,...,z_8) = \frac{|z_1|^2}{2}$ and consider
\begin{equation*}
\phi_1(x+iy,u+iv) = 
\left(
\begin{aligned}
    &  \ x+iy\\
    &  \ u+iv\\
    & \sqrt{2-x^2-y^2+u^2+v^2}\\
     & \sqrt{6-2x^2-2y^2+u^2+v^2}\\
    & \sqrt{6-x^2-y^2}\\
    & \sqrt{8-u^2-v^2}\\
    & \sqrt{4+x^2+y^2-u^2-v^2}\\
    & \sqrt{2+2x^2+2y^2-u^2-v^2}
\end{aligned}
\right).
\end{equation*}
and calculate
\begin{align*}
    J(\phi_1(x+iy,u+iv)) & = \frac{1}{2} (x^2+y^2), \\
    d(J \circ \phi_1)(x+ iy,u+ iv) & = (x, \ y, \ 0, \ 0 ).
\end{align*}
Thus $dJ=0$ on $(U_1, \phi_1)$ is equivalent with $x+iy=0$, i.e.\ with vanishing of the first coordinate in $\phi_1$. For $(U_8, \phi_8)$ it is the same. For $(U_4, \phi_4)$ and $(U_5, \phi_5)$ it is the same w.r.t.\ the fifth coordinate instead of the first one.

If the first coordinate vanishes, then, by definition, $J(0, z_2, \dots, z_8)= \frac{\abs{0}^2}{2}=0$. If the fifth coordinate vanishes, we obtain from the definition of the charts via $x_5 =\sqrt{6- \abs{z_1}^2}$ immediately $\abs{z_1}^2 = 6$ and thus $J(z_1,z_2, z_3, z_4, 0, z_6, z_7, z_8)  = 3$.  
\end{proof}


\subsection{Criteria for rank one singular points}

Now let us study the singular points of rank one of $(J, H_t)$. 

\begin{lemma}
\mbox{ \ }
\begin{enumerate}
 \item 
 All rank one singular points with $dJ=0$ lie in $J^{-1}(0)$ and $J^{-1}(3)$. 
 \item
 All points in $J^{-1}(0)$ and $J^{-1}(3)$ satisfy $dJ=0$ and are thus rank zero or rank one singular points.
\end{enumerate}
\end{lemma}

\begin{proof}
{\em 1)} Let $p$ be a singular point with $dJ(p)=0$. If $p$ lies in $ U_2$ or $U_3$ or $U_6$ or $U_7$ then, by the proof of Theorem \ref{the:invariablesingpoint}, $p$ is an invariant singular point and thus, by Theorem \ref{the:invariablesingpoint}, is not of rank one. Therefore $p$ must lie in $U_1$ or $U_4$ or $U_5$ or $U_8$. From the proof of Theorem \ref{the:zerocor}, we deduce that $J(p)$ equals either $0$ and $3$.

{\em 2)}
The value of $J$ in its global maximum is $3$ and in its global minimum $0$. Thus we automatically have $dJ=0$ on $J^{-1}(0)$ and $J^{-1}(3)$.
\end{proof}

Now we will investigate the case of rank one singular points with $dJ \neq 0$. First deduce

\begin{proposition}
\label{prop:dj}
Let $(x+iy,u+iv) \in \phi_1^{-1}(U_1)$ be a singular point of $(J, H_t= H_{(t_1, t_2, t_3, t_4)})$ with $d(J \circ \phi_1)(x+iy,u+iv) \neq 0$ and $t_1 \neq 0$. Then $v=0$. 
\end{proposition}

\begin{proof}
It is convenient to work with the polar coordinates $\EuScript P_1$ as defined in \eqref{eq:polarcoordinates}. Sometimes we identify $\cos(\theta_1) +i  \sin (\theta_1) = e^{i \theta_1}$ etc.
We compute
\begin{align*}
    & (J \circ \phi_1 )\bigl(r_1 \cos(\theta_1) + i \ r_1 \sin(\theta_1), \ r_2 \cos(\theta_2) + i \ r_2 \sin(\theta_2)\bigr) = \frac{r_1^2}{2},\\
    & (H_{t} \circ \phi_1)\bigl(r_1 \cos(\theta_1) +i \ r_1 \sin(\theta_1), \ r_2 \cos (\theta_2) +i \ r_2 \sin (\theta_2) \bigr) = \Gamma_{t}(r_1,r_2) + t_1 \cos(\theta_2) \ \overline{\Omega}(r_1,r_2)
\end{align*}
where $t_1$ is the first coordinate of $t=(t_1, t_2, t_3, t_4)$ and $\Ga_t: D_1 \to \R$ and $\overline{\Omega}: D_1 \to \R$ are defined by the equation above in the sense that $\overline{\Omega}$ is the unique function that satisfies
\begin{align*}
    (\gamma_1 \circ \phi_1)(r_1e^{i\theta_1},r_2e^{i\theta_2}) = \cos ( \theta_2) \ \overline{\Omega}(r_1,r_2)
\end{align*}
and can be explicitly computed as
$$
\overline{\Omega}(r_1,r_2)  =\frac{r_2}{50} \sqrt{(
 8 - r_2^2) (2 + 2 r_1^2 - r_2^2) (6 - 2 r_1^2 + r_2^2) (4 + r_1^2 - r_2^2) (2 - r_1^2 + r_2^2)}
 $$
and $\Gamma_{t}$ is given by
\begin{align*}
    \Gamma_{t}(r_1,r_2) &= \frac{1}{100} \left(100 - 200 t_1 + 1152 t_3 - 50 r_1^2 + 100 t_1 r_1^2 - 192 t_3 r_1^2 +   72 t_2 r_1^4 - 184 t_3 r_1^4  \right. \\
   & \qquad \quad + 16 t_4 r_1^4   - 48 t_2 r_1^6 + 16 t_3 r_1^6 +    8 t_4 r_1^6 + 8 t_2 r_1^8 + 8 t_3 r_1^8 + t_4 r_1^8 + 50 r_2^2 -  100 t_1 r_2^2  \\
   & \qquad \quad   - 192 t_3 r_2^2 + 304 t_3 r_1^2 r_2^2    + 24 t_2 r_1^4 r_2^2 -    8 t_3 r_1^4 r_2^2 - 8 t_4 r_1^4 r_2^2 - 8 t_2 r_1^6 r_2^2  -    24 t_3 r_1^6 r_2^2  \\
   & \qquad \quad - 2 t_4 r_1^6 r_2^2 - 88 t_3 r_2^4 - 16 t_3 r_1^2 r_2^4 + 
   2 t_2 r_1^4 r_2^4  + 26 t_3 r_1^4 r_2^4 + t_4 r_1^4 r_2^4 + 8 t_3 r_2^6  \\
   & \qquad \quad  \left. -   12 t_3 r_1^2 r_2^6 + 2 t_3 r_2^8 \right).
\end{align*}
Recall that $dJ \neq 0$ by assumption.
Thus a singular point is of rank one if there exists $s \in \mathbb{R}$ such that for $\al \in \{r_1, r_2, \theta_1, \theta_2\}$
\begin{align*}
     s \ \partial_{\al} (J \circ \phi_1 \circ \EuScript P_1) =  \partial_{\al} (H_t \circ \phi_1 \circ \EuScript P_1).
\end{align*}
This leads to the equations
\begin{align*}
   s  r_1 &=  \partial_{r_1}(H_t\circ \phi_1 \circ \EuScript P_1),  && 
    0 =  \partial_{r_2}\Gamma_{t_1,t_2,t_3,t_4}(r_1,r_2) + t_1 \cos (\theta_2) \ \partial_{r_2}\overline{\Omega}(r_1,r_2),\\
    0 &=  0, && 
    0 = - t_1\sin (\theta_2) \ \overline{\Omega}(r_1,r_2).
\end{align*}
The above equations reduce to 
\begin{align*}
    0 = \partial_{r_2}\Gamma_{t_1,t_2,t_3,t_4}(r_1,r_2) + t_1 \cos (\theta_2) \ \partial_{r_2}\overline{\Omega}(r_1,r_2), \quad\quad 
     0 = \sin (\theta_2) \ \overline{\Omega}(r_1,r_2).
\end{align*}

By Lemma \ref{lem:omega},  $\overline{\Omega}(r_1,r_2)$ is non zero on its domain of definition. Therefore the only way to have $0 = \sin (\theta_2) \ \overline{\Omega}(r_1,r_2)$ is to require either $\theta_2 =0$ or $\theta_2 =\pi$. Both cases result in $v=0$.
\end{proof}

In particular, we obtain

\begin{corollary}
\label{cor:finalsing}
 Let $r_1 \neq 0$ and $t_1 \neq 0$ and let $(r_1, r_2) \in D_1$ be a solution of 
\begin{equation}
\label{eq:finalsing}
    0 = (\partial_{r_2}\Gamma_{t_1,t_2,t_3,t_4}(r_1,r_2))^2 - t_1^2(\partial_{r_2}\overline{\Omega}(r_1,r_2))^2.
\end{equation}
Then for arbitrary $\theta_1$ and either $\theta_2 = 0$ or $\theta_2 = \pi$, we find that $ (\phi_1 \circ \EuScript P_1) (r_1, r_2, \theta_1, \theta_2) $ is a singular point of rank one.  
\end{corollary}

\begin{proof}
The choice of $\theta_2 \in \{0, \pi\}$ at the end of the proof of Proposition \ref{prop:dj} yields the two equations
$$ 0 = \partial_{r_2}\Gamma_{t_1,t_2,t_3,t_4}(r_1,r_2) \pm t_1 \partial_{r_2}\overline{\Omega}(r_1,r_2) $$
which can be combined to 
$$ 0 = (\partial_{r_2}\Gamma_{t_1,t_2,t_3,t_4}(r_1,r_2))^2 - t_1^2(\partial_{r_2}\overline{\Omega}(r_1,r_2))^2. $$
\end{proof}

The advantage of equation \eqref{eq:finalsing} is that it can easily be reduced to a polynomial since there are no square roots involved in contrast to the original equations. 
A disadvantage of this combined equation is that after having found a solution we still need to check whether these solutions satisfy 
$$0 = \partial_{r_2}\Gamma_{t_1,t_2,t_3,t_4}(r_1,r_2) + t_1 \partial_{r_2}\overline{\Omega}(r_1,r_2)
\quad \mbox{ or } \quad 
0 = \partial_{r_2}\Gamma_{t_1,t_2,t_3,t_4}(r_1,r_2) - t_1 \partial_{r_2}\overline{\Omega}(r_1,r_2).
$$


\subsection{Criteria for singular points via symplectic reduction}
\label{sec:Reducedmanifold}

Recall the notion of a symplectic quotient from Theorem \ref{marsdenWeinsteinToric} and apply it to $J: M \to \R$ and its induced $\mbS^1$-action to obtain $ J^{-1}(j) \slash \mbS^1=: M^{red, j}$. 
Denote by $H_t^{red, j}: M^{red, j} \to \R$ the Hamiltonian $H_t :M \to \R$ descended to the symplectic quotient $M^{red, j}$.
In this section, we show how to reduce the problem of finding and determining singular points of $(J, H_t): M \to \R^2$ on $M$ to finding and determining singular points of $H_t^{red, j}: M^{red, j} \to \R$. The latter problem is easier to solve since it is lower dimensional.

The reduced space $M^{red, j}$ is plotted for various values of $J$ in Figure \ref{fig:reducedmanifold}. 
For details about the topology and geometry of $M^{red, j}$, we refer the reader to the papers by De Meulenaere $\&$ Hohloch \cite[Section 4]{meulenaere2019family} and Le Floch $\&$ Palmer \cite{floch2019semitoric}. In short, for all regular values of $J$, the reduced spaces $M^{red, j}$ are smooth manifolds diffeomorphic to $\mbS^2$. The $J$-values $j \in \{0, 1, 2, 3\}$ are singular and the shape of the reduced space depends on the $J$-value being at a maximum/minimum of $J$ or not: according to Le Floch $\&$ Palmer \cite[Lemma 2.12]{floch2019semitoric} and De Meulenaere $\&$ Hohloch \cite[Lemma 4.1]{meulenaere2019family}, the reduced spaces for $j \in \{0,3\}$ are diffeomorphic to $\mbS^2$ but for $j \in \{1,2\}$ only homeomorphic to $\mbS^2$.

\begin{figure}[ht]
    \centering
    \includegraphics[scale = .25]{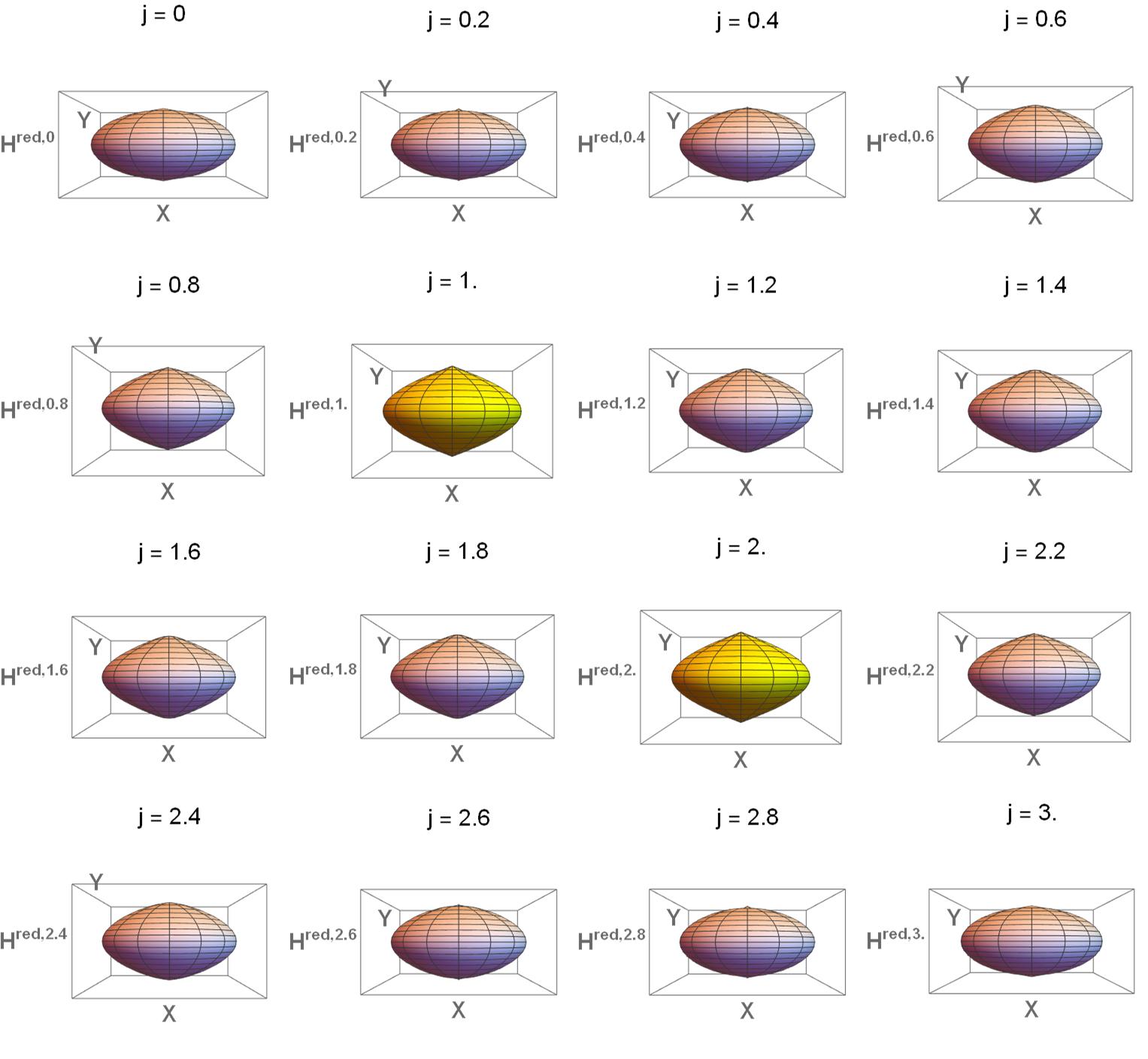}
    \caption{The reduced manifold $M^{red, j}$ for different values $j$ of $J$.}
    \label{fig:reducedmanifold}
\end{figure}

For regular levels of $J$, the following statement explains the relation between rank one singular points on $M$ and singular points on the symplectic quotient. Recall that a group action is {\em free} if it has no fixed points. Acting freely on a point means that the stabiliser of this point is trivial.

\begin{lemma}
\label{lem:red}
Let $p \in M$ be a singular point of $(J, H_t)$ such that the $\mbS^1$-action induced by $J$ acts freely on $p$ and denote by $[p] \in M^{red, J(p)}=J^{-1}(J(p)) \slash \mbS^1$ the equivalence class of $p$. Then
\begin{enumerate}
    \item 
    $p$ is non-degenerate if and only if $[p] \in M^{red, J(p)}$ is non-degenerate (in the sense of Morse theory) for $H_t^{red, J(p)}$.
    
     \item 
     $p$ is hyperbolic-regular if and only if $[p] \in M^{red, J(p)}$ is a hyperbolic fixed point for $H_t^{red, J(p)}$.
     
     \item 
     $p$ is elliptic-regular if and only if $[p] \in M^{red,J(p)}$ is an elliptic fixed point for $H_t^{red, J(p)}$.
\end{enumerate}
\end{lemma}

\begin{proof}
 See for example Le Floch $\&$ Palmer \cite[Lemma 2.6]{floch2019semitoric}.
\end{proof}

\begin{lemma}
\label{lem:dJnotZero}
The $\mbS^1$-action induced by $J$ acts freely on the set of points with $dJ \neq 0$.
\end{lemma}

\begin{proof}
In our convention, the maximal period of the $\mbS^1$-action induced by $J$ equals $2\pi$. 
By Theorem \ref{the:equalperiod}, the period of points with $dJ \neq 0$ is $2 \pi$, thus the action on this set is free.
\end{proof}

Given a ($2 \times 2$)-matrix $A=\left( \begin{smallmatrix} a_{11} & a_{12} \\ a_{21} & a_{22} \end{smallmatrix} \right) $, its characteristic polynomial is given by
$$
\chi_A(\lam):= \det(A-\lam \Id)= \lam^2 -(a_{11} + a_{22}) \lam + a_{11}a_{22} - a_{12} a_{21}.
$$
By means of $\det(A) = a_{11}a_{22} - a_{12} a_{21}$ and $\tr(A)= a_{11} + a_{22}$, we can reformulate $\chi_A$ as
$$
\chi_A(\lam)= \lam^2 - \tr(A) \lam + \det(A)
$$
so that we obtain the following formula for the eigenvalues of $A$:
$$
\lam_\pm = \tfrac{1}{2} \left( \tr(A) \pm \sqrt{\tr(A)^2 - 4 \det(A)} \ \right).
$$
Now consider

\begin{lemma}
\label{lem:hypcondition}
\label{lem:ellipcondition}
Let $p \in U_1$ be a rank one singular point with $dJ(p) \neq 0$.
\begin{enumerate}
 \item 
 $p$ is hyperbolic-regular for $(J, H_t)$ if and only if 
\begin{align*}
    \partial^2_u H_{t}^{red, J(p)} \partial^2_v H_{t}^{red, J(p)} - \left( \partial_u\partial_v H_{t}^{red, J(p)} \right )^2<0.
\end{align*}

\item
$p$ is elliptic-regular for $(J, H_t)$ if and only if 
\begin{align*}
    \partial^2_u H_{t}^{red, J(p)} \partial^2_v H_{t}^{red, J(p)} - \left( \partial_u\partial_v H_{t}^{red, J(p)} \right )^2 > 0.
\end{align*}
\end{enumerate}
\end{lemma}

\begin{proof}
We only prove the first claim since the second one follows analogously.
By Lemma \ref{lem:red}, the point $p$ is hyperbolic-regular if and only if
\begin{align*}
    \om^{-1}_p d^2H^{red, J(p)}(p)
    = 
    \begin{pmatrix}
    -\partial_u\partial_v H_{t}^{red, J(p)} &  -\partial^2_u H_{t}^{red, J(p)} \\
   \partial^2_v H_{t}^{red, J(p)} & \partial_u\partial_v H_{t}^{red, J(p)}
    \end{pmatrix}
\end{align*}
has two non-zero real eigenvalues.
Since this matrix is traceless this is equivalent with  
\begin{align*}
    \det\begin{pmatrix}
    -\partial_u\partial_v H_{t}^{red, J(p)} &  -\partial^2_u H_{t}^{red, J(p)}  \\
   \partial^2_v H_{t}^{red, J(p)} & \partial_u\partial_v H_{t}^{red, J(p)}
    \end{pmatrix} < 0.
\end{align*}
This gives us 
\begin{align*}
\partial^2_u H_{t}^{red, J(p)} \partial^2_v H_{t}^{red, J(p)} - \left( \partial_u\partial_v H_{t}^{red, J(p)} \right )^2<0.
\end{align*}
\end{proof}


\subsection{Criteria for degenerate singular points}

Recall that we obtained $(J, H_t)$ from the toric system $(J, H)$ by perturbing $H$ to $H_t$ while leaving $J$ unchanged. This allows us to give a fairly simple criterion in terms of the Hessian of $H_t$ for degenerate singular points whenever $dJ \neq 0$, as we will see below.

\begin{theorem}
Consider $(J, H_t)$ and assume $t_1 \neq 0$ in $t=(t_1, t_2, t_3, t_4)$.
Recall the polar coordinates $\EuScript P_1 (r_1, r_2, \theta_1, \theta_2)= (r_1e^{i \theta_1},r_2 e^{i \theta_2})$ from \eqref{eq:polarcoordinates}. 
Let $(r_1e^{i \theta_1},r_2 e^{i \theta_2}) \in \phi_1^{-1}( U_1)$ be a rank one singular point of $(J, H_t)$ with $dJ \neq 0$. Then 
$$
 (r_1e^{i \theta_1},r_2 e^{i \theta_2}) \ \mbox{ degenerate } 
\quad \IFF \quad  
\bigl(\partial^2_{r_2}\Gamma_t(r_1,r_2)\bigr)^2 + t_1 \cos(\theta_2) \ \partial^2_{r_2}\overline{\Omega}(r_1,r_2)=0.
$$
In particular, these degenerate singular points satisfy
\begin{equation}
\label{eq:criterionDeg}
     \left(\partial^2_{r_2}\Gamma_{t}(r_1,r_2)\right)^2- t_1^2\left(\partial^2_{r_2}\overline{\Omega}(r_1,r_2)\right)^2=0.
\end{equation}
\end{theorem}

\begin{proof}
Let the rank of the singular point $z:=(r_1e^{i \theta_1},r_2 e^{i \theta_2})$ be one but $d (J \circ \phi_1 ) \neq 0$.
Please note that, for sake of readability, we write in the following computation $H_t^{red,j}$ instead of $H_t^{red,j} \circ \phi_1 \circ \EuScript P_1$ where $j:=J(\phi_1(z))$.
Recall that degeneracy of a singular point $z$ with $d(J \circ \phi_1) (z) \neq 0$ is, according to Lemma \ref{lem:red} and Lemma \ref{lem:dJnotZero}, equivalent with  
\begin{equation}
\label{eq:deg}
0 = \det \left(d^2 H_t^{red,j} \right) = \det  
    \begin{pmatrix}
    \partial^2_{r_2} H_t^{red,j} && \partial_{r_2}\partial_{\theta_2} H_t^{red,j}\\
     \partial_{r_2}\partial_{\theta_2} H_t^{red,j} && \partial^2_{\theta_2} H_t^{red,j}
    \end{pmatrix}
    .    
\end{equation}
Using the formulas computed in the proof of Proposition \ref{prop:dj} and in particular $\sin(\theta_2)=0$, equation \eqref{eq:deg} becomes
\begin{align*}
    0 &= \det 
    \begin{pmatrix}
    \partial^2_{r_2}\Gamma_{t}\left(r_1,r_2 \right) + t_1 \cos ( \theta_2) \ \partial^2_{r_2}\overline{\Omega}\left(r_1,r_2\right) &&  -t_1 \sin (\theta_2) \ \partial_{r_2}\overline{\Omega}\left(r_1,r_2\right) \\
    -t_1 \sin (\theta_2) \ \partial_{r_2}\overline{\Omega}\left(r_1,r_2\right) & &-t_1\cos (\theta_2) \ \overline{\Omega}\left(r_1,r_2 \right)
    \end{pmatrix}\\
    &=\det \begin{pmatrix}
    \partial^2_{r_2}\Gamma_{t}\left(r_1,r_2 \right) + t_1 \cos (\theta_2) \ \partial^2_{r_2}\overline{\Omega}\left(r_1,r_2\right) && 0 \\
    0 &&-t_1\cos (\theta_2) \ \overline{\Omega}\left(r_1,r_2\right)
    \end{pmatrix}\\
    &= -t_1\cos (\theta_2) \ \overline{\Omega}\left(r_1,r_2\right) \ \partial^2_{r_2}\Gamma_{t}\left(r_1,r_2\right) -t _1^2 \cos^2 ( \theta_2) \ \overline{\Omega}\left(r_1,r_2\right) \ \partial^2_{r_2}\overline{\Omega}\left(r_1,r_2\right).
\end{align*}
Using the fact that $t_1 \neq 0$ and Lemma \ref{lem:omega}, we obtain
\begin{equation}
\label{eq:options}
    \partial^2_{r_2}\Gamma_{t}(r_1,r_2)+ t_1\cos (\theta_2) \ \partial^2_{r_2}\overline{\Omega}(r_1,r_2)=0.
\end{equation}
Moreover, $\sin (\theta_2) = 0$ implies $\cos (\theta_2) = \pm 1$. Inserting both options in equation \eqref{eq:options} and multiplying them, yields 
\begin{align*}
     \left(\partial^2_{r_2}\Gamma_{t}(r_1,r_2)\right)^2- t_1^2\left(\partial^2_{r_2}\overline{\Omega}(r_1,r_2)\right)^2=0.
\end{align*}
\end{proof}

Note that equation \eqref{eq:criterionDeg} can be written as a polynomial fraction.
In particular, when working with {\em Mathematica}, determining the zeros of a polynomial is much faster than determining the zeros of an arbitrary equation that contains square roots. 
In addition, instead of solving two equations, finding degenerate points can be done by solving only one equation.


\section{Hyperbolic-regular fibres}
\label{sec:examhyp-reg-fib}

In this section, we will study the occurrence and topology of hyperbolic-regular fibres in the systems $(J, H_t)$ for certain values of the parameter $t$ by means of symplectic reduction.
Note that the coordinate charts $(U_\nu, \phi_\nu)$ of $(M, \om)$ for $1 \leq \nu \leq 8$ are closed w.r.t.\ the actions induced by $J$ and $H$ but not always w.r.t.\ the action induced by $H_t$. In particular, the coordinate charts descend to the reduced space $M^{red, j}$. For example, $\phi_1$ descends to $M^{red, j}$ as 
\begin{align*}
 \phi_1^{red, j} : \C \to U_1^{red, j}:= \bigl(U_1 \cap\ J^{-1}(j) \bigr) \slash \mbS^1 \subset M^{red, j}, \quad 
 \phi^{red, j}_1(u,v):=\phi_1\left(\sqrt{2j}, \ 0, \ u, \ v\right).
\end{align*}
In case the reduced spaces are not smooth, one may need certain additional assumptions which we will introduce when needed.

Now recall some notions and results from Gullentops \cite{thesis:yannick} and Gullentops $\&$ Hohloch \cite{fiberGullandHohl} concerning hyperbolic-regular fibres of proper $\mbS^1$-systems.

\begin{definition}
If a space is homeomorphic to the product of a figure eight loop (cf.\ Figure \ref{subfig:twoLoops}) with a circle we call it a {\em $2$-stacked torus} (see Figure \ref{fig:doubleTorus}). 
Moreover, if a space is homeomorphic to the product of a curve as in Figure \ref{subfig:threeLoops} with a circle then it is said to be a {\em $3$-stacked torus} (see Figure \ref{fig:tritorus}). Analogously we define a $4$-stacked torus and so on.
\end{definition}

Intuitively, a $2$-stacked torus can be seen as two tori `glued on top of each other', a $3$-stacked torus as three tori `glued on top of each other' and so on for the $4$-stacked torus etc.


\subsection{Appearing hyperbolic-regular fibres}

Note that the symplectic quotient $M^{red, j}$ has certain features in common with the generalised bouquet $(G,S,C)$ in Theorem \ref{theo:main}: $G$ is defined as the quotient of a given hyperbolic-regular fibre w.r.t.\ the action induced by $J$. Therefore $G$ appears naturally as a subset of some reduced space $M^{red, j}$. Because of Lemma \ref{lem:red}, $G$ is in fact a hyperbolic fibre of $(M^{red, j},H^{red,j})$. The set $S$ is the quotient of the set of points with $J$-period $\pi$. Note that $S = \emptyset$ here due to Theorem \ref{the:equalperiod}.

We want to use this to deduce the topology of a hyperbolic-regular fibre of $H_t$ by looking at the level sets of $H_{t}^{red, j}$ on $M^{red, j}$.

\begin{example}
\label{ex:lemniscate}
Let $t= (t_1, t_2, t_3, t_4)= \left( \frac{1}{4},  \frac{1}{3}, \frac{1}{3}, 1 \right)$ and $j=2$.
Then there exists $h \in \R$ such that the fibre $(H_t^{red,2})^{-1}(h)$ is homeomorphic to a figure eight (see Figure \ref{subfig:levelSetDutorus}) and, as a result, the fibre $(J,H_t)^{-1}(2, h)$ is homeomorphic to a $2$-stacked torus (see Figure \ref{fig:doubleTorus}). 
\end{example}

\begin{proof}
To determine the singular points on $M^{red,2}$, we employ Corollary \ref{cor:finalsing}: For $t= (t_1, t_2, t_3, t_4)= \left( \frac{1}{4},  \frac{1}{3}, \frac{1}{3}, 1 \right)$ and $j=2$, the following polynomial is the numerator of the expression \eqref{eq:finalsing} evaluated with {\em Mathematica}: 
\begin{align*}
   0 &=  655.36 -286855 \ r^2+797691 \ r^4 - 972680 \ r^6+684747 \ r^8  \\ 
     & \quad -309212 \ r^{10}+94171.9 \ r^{12} -19783.5 \ r^{14}+2874.65 \ r^{16} \\
    & \quad - 283.91 \ r^{18}+18.1967 \ r^{20}-0.682667 \ r^{22}+0.0113778 \ r^{24}
\end{align*}
Solving this equation gives us exactly two singular points in the domain of $\phi^{red,2}_1$, namely $\xi_1:= (1.48116, 0)$ and $\xi_2:= (-1.66216, 0)$.  
Using Lemma \ref{lem:hypcondition}, we conclude that $\xi_1 \in M^{red,2}$ is a singular point of hyperbolic type of $H_t^{red,2}$ and $\xi_2 \in M^{red,2}$ is a singular point of elliptic type of $H_t^{red,2}$. 
Now set $h:= H_t^{red,2}(\xi_1)$. 
Plotting $(H^{red,2}_t)^{-1}(h)$, see Figure \ref{subfig:levelSetDutorus}, we find it to be homeomorphic to a figure eight curve, see Figure \ref{subfig:twoLoops}.
The associated generalised bouquet is given by 
$$\bigl(G := (H^{red,2}_t)^{-1}(h), \ S := \emptyset, \  C:=\left\{\xi_1\right\} \bigr)$$ 
and therefore, by Theorem \ref{theo:main}, the fibre $(J,H_t)^{-1}(2, h)$ is homeomorphic to a $2$-stacked torus. 
\end{proof}

\begin{figure}[h]
    \centering
\begin{subfigure}[t]{.48\textwidth}
\centering
   \includegraphics[scale =.35]{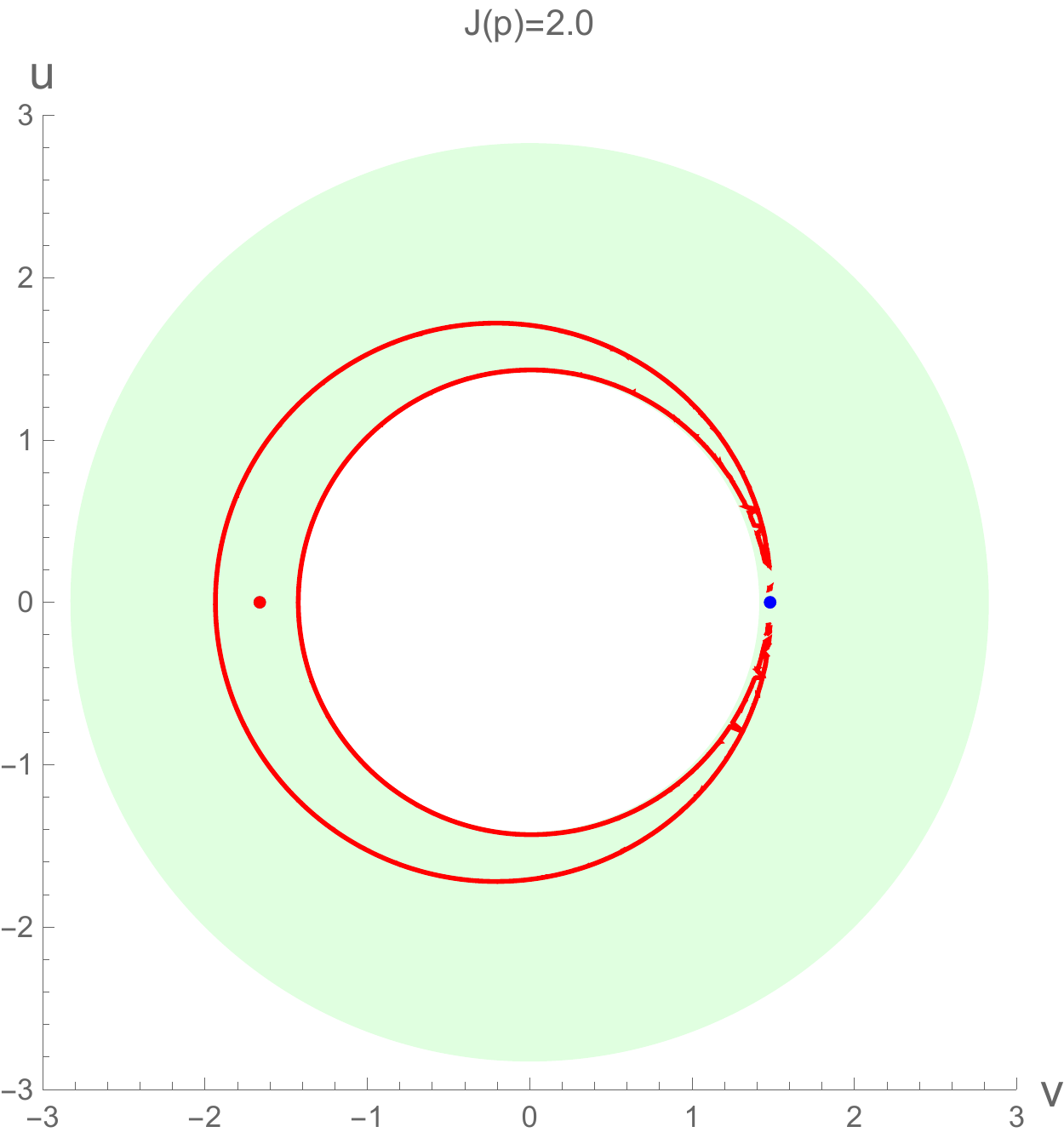}
    \caption{Hyperbolic fibre on the reduced chart $\phi_1^{red,2}$ with one elliptic and one hyperbolic singular point visible.}
    \label{subfig:levelSetDutorus}
 \end{subfigure}\quad
\begin{subfigure}[t]{.48\textwidth}
\centering
\includegraphics[scale =.85]{twoLoops.pdf}
\caption{Figure eight curve.}
 \label{subfig:twoLoops}
\end{subfigure}
\caption{$2$-stacked torus on the reduced space $M^{red, 2}$ from Example \ref{ex:lemniscate}.}
\label{fig:levelSetDutorus}
\end{figure}

\begin{example}
\label{ex:tritorus}
Let $t= (t_1, t_2, t_3, t_4)= \left( \frac{1}{2},  \frac{1}{2}, \frac{1}{3}, 1 \right)$ and $j=2$.
Then there exists $h \in \R$ such that the fibre $(H_t^{red,2})^{-1}(h)$ is homeomorphic to the curves plotted in Figure \ref{fig:levelSetTritorus} and, as a result, the fibre $(J,H_t)^{-1}(2, h)$ is homeomorphic to a $3$-stacked torus (see Figure \ref{fig:tritorus}). 
\end{example}

\begin{proof}
To determine the singular points on $M^{red,2}$, we employ Corollary \ref{cor:finalsing}: For $t= (t_1, t_2, t_3, t_4)= \left( \frac{1}{2},  \frac{1}{2}, \frac{1}{3}, 1 \right)$ and $j=2$, the following polynomial is the numerator of the expression \eqref{eq:finalsing} evaluated with {\em Mathematica}: 
\begin{align*}
0 & =  2621.44 -366594 \ r^2+966351 \ r^4- (1.12567*10^6) \ r^6 + 761534 \ r^8 \\
& \quad - 332674 \ r^{10} + 98714.3 \ r^{12} -20343.5 \ r^{14} + 2917.25 \ r^{16} - 285.733 \ r^{18} \\
& \quad + 18.2302 \ r^{20} - 0.682667 \ r^{22} + 0.0113778 \ r^{24}
\end{align*}
Solving this equation gives us exactly four singular points in the domain of $\phi^{red,2}_1$, namely 
$$
\xi_1:= (1.56842, 0) , \quad \xi_2 :=  (2.23607, 0) , \quad \xi_3 := (-2.23607,0)  ,\quad \xi_4:= (2.74592, 0) .
$$
Using Lemma \ref{lem:hypcondition}, we conclude that $\xi_1, \xi_4 \in M^{red,2}$ are singular points of hyperbolic type of $H_t^{red,2}$ and $\xi_2, \xi_3 \in M^{red,2}$ are singular points of elliptic type of $H_t^{red,2}$. 
Now set $h:= H_t^{red,2}(\xi_1) = H_t^{red,2}(\xi_4)$. 
Plotting $(H^{red,2}_t)^{-1}(h)$, see Figure \ref{subfig:levelSetTritorus}, we find it to be homeomorphic to the curve plotted in Figure \ref{subfig:threeLoops}.
The associated generalised bouquet is given by 
$$\bigl(G := (H^{red,2}_t)^{-1}(h), \ S := \emptyset, \ C:=\left\{\xi_1, \xi_4 \right\}\bigr)$$ 
and therefore, by Theorem \ref{theo:main}, the fibre $(J,H_t)^{-1}(2, h)$ is homeomorphic to a $3$-stacked torus. 
\end{proof}

\begin{figure}[h]
    \centering
\begin{subfigure}[t]{.48\textwidth}
\centering
   \includegraphics[scale =.35]{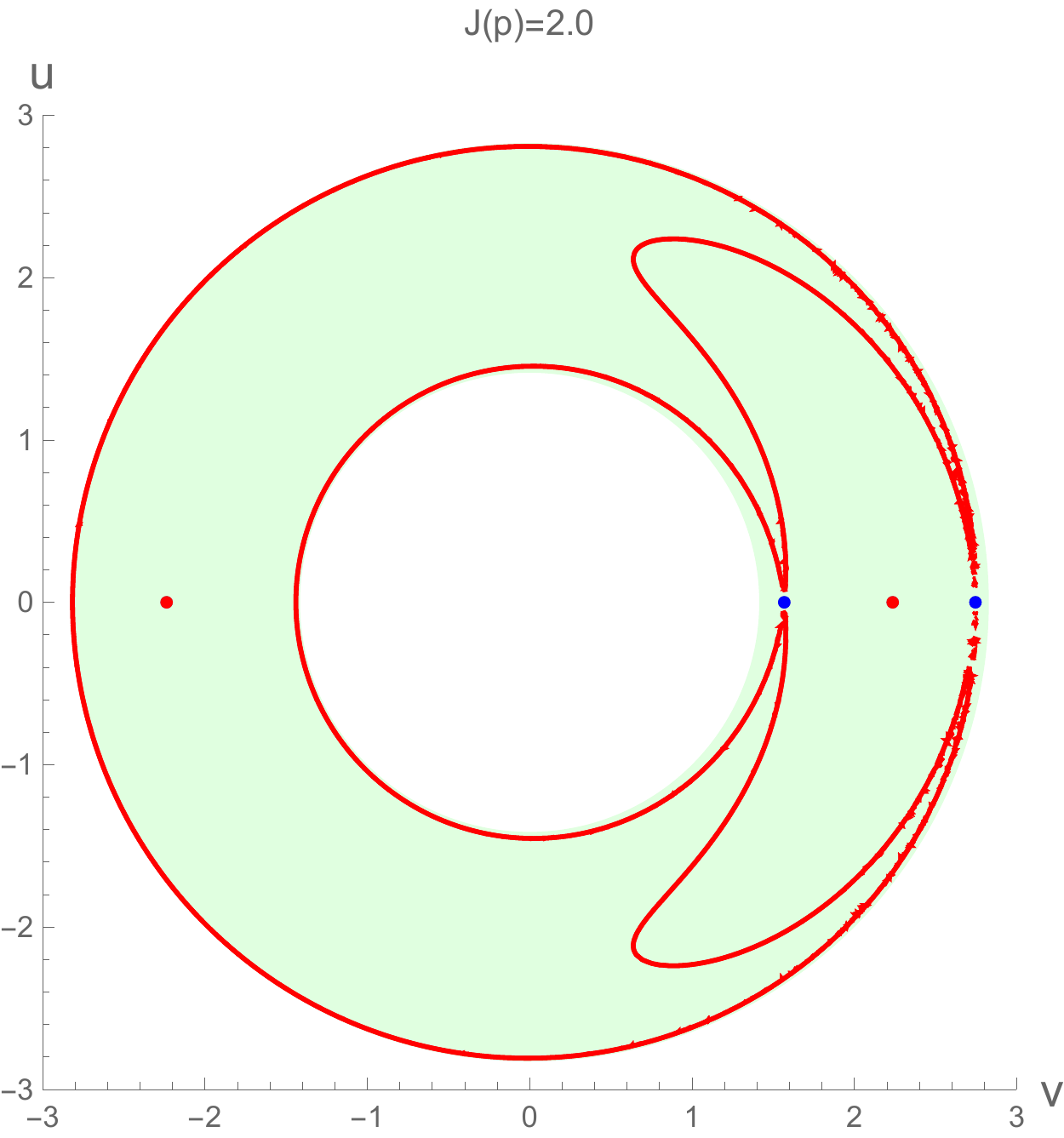}
    \caption{Hyperbolic fibre on the reduced chart $\phi_1^{red,2}$ with two hyperbolic and two elliptic singular points visible.}
    \label{subfig:levelSetTritorus}
 \end{subfigure}\quad
\begin{subfigure}[t]{.48\textwidth}
\centering
\includegraphics[scale =.8]{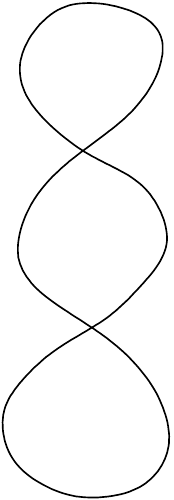}
\caption{Homeomorphic drawing of the plotted red curve in Figure \ref{subfig:levelSetTritorus}.}
 \label{subfig:threeLoops}
\end{subfigure}
\caption{$3$-stacked torus on the reduced space $M^{red, 2}$ from Example \ref{ex:tritorus}.}
\label{fig:levelSetTritorus}
\end{figure}


\subsection{Collisions of $2$-stacked torus fibres}

Example \ref{ex:lemniscate} and Example \ref{ex:tritorus} are in fact part of a larger picture in the sense that transition from one to the other happens naturally for certain values of $J$ and $t$ as we will see now. Moreover, we observe again the importance of Theorem \ref{theo:main} when we deduce the topology of the fibre from its topological shape in the reduced space.

\begin{example}
\label{ex:gridsmooth}
Consider $(J, H_t)$ for $t=\left(\frac{1}{2},\ \frac{1}{2}, \ \frac{1}{3},\ 1\right)$.
In Figure \ref{fig:Jdomains}, the domains of $\phi^{red, j}_1$ for certain values of $J$ are plotted. Figure \ref{fig:Jdomains} displays the fibres of $H^{red, j}_t$. Thus, by Theorem \ref{theo:main}, the connected components of the hyperbolic-regular fibres are $2$-stacked tori except for $j=2$ where the collision of two $2$-stacked tori creates a $3$-stacked torus, see Figure \ref{fig:Jcollisions}. Note that this is the only fibre of $\left(J,H_{\left(\frac{1}{2},\ \frac{1}{2}, \ \frac{1}{3},\ 1\right)}\right)$ that is not a $2$-stacked torus.
\end{example}

\begin{proof}
Let $t=\left(\frac{1}{2},\ \frac{1}{2}, \ \frac{1}{3},\ 1\right)$.
Figure \ref{fig:Jdomains} and Figure \ref{fig:Jcollisions} are obtained similarly as the figures used for Example \ref{ex:lemniscate} and Example \ref{ex:tritorus}. The sequence of subfigures of Figure \ref{fig:Jdomains} shows the fibres of $H^{red, j}_t$ transitioning from two elliptic singular points via one figure eight curve to two figure eight curves and back to two elliptic points. Figure \ref{fig:Jcollisions} zooms in between $j= 1.90125$ and $j= 2.10125$ and displays how the two figure eight curves collide into a curve underlying a $3$-stacked torus and then revert back into a figure eight curve.
\end{proof}

\begin{figure}[h]
\begin{center}
    \includegraphics[scale = .6]{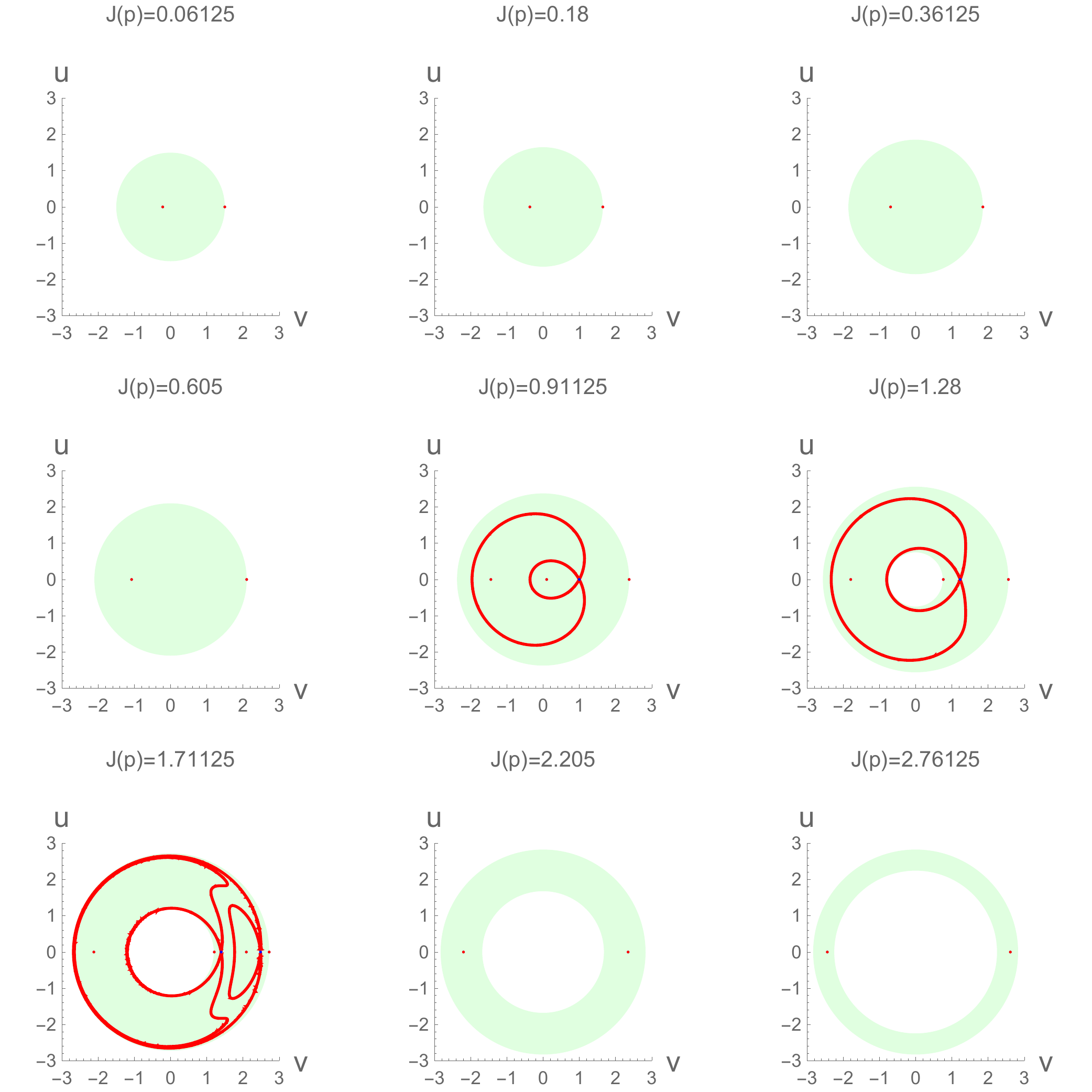}
\end{center}
\caption{Domains of $\phi_1^{red, j}$ for certain values of $j=J(p)$ where $p$ is some representative point in the domain to trace the value of $J$. 
The subfigures show the fibres of $H^{red, j}_{\left(\frac{1}{2},\ \frac{1}{2}, \ \frac{1}{3},\ 1\right)}$ transitioning from two elliptic singular points via one figure eight curve to two figure eight curves and back to two elliptic points.}
\label{fig:Jdomains}
\end{figure}

\begin{figure}[h]
\begin{center}
    \includegraphics[scale = .72]{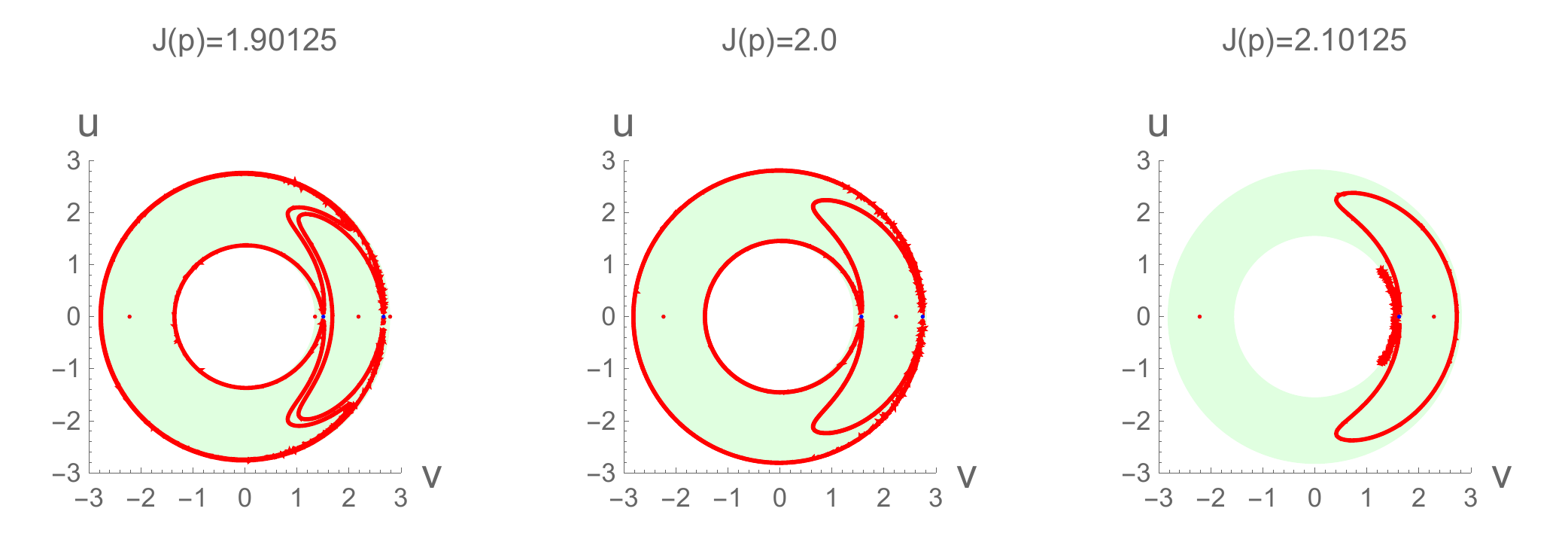}
\end{center}
\caption{Quotient of the transition from two $2$-stacked tori via one $3$-stacked torus to one $2$-stacked torus by the $\mbS^1$-action of $J$.}
\label{fig:Jcollisions}
\end{figure}

Example \ref{ex:gridsmooth} can also be observed via the bifurcation diagram: 

\begin{remark}
Each point of a blue line in Figure \ref{fig:aroundtritorus} has a $2$-stacked torus as fibre. At the intersection point of the blue lines, the fibres of the $2$-stacked tori collide into a $3$-stacked torus. 
\end{remark}

\begin{figure}[hb]
    \centering
    \includegraphics[trim={0 1.5cm 0 2.9cm},clip,scale = .85]{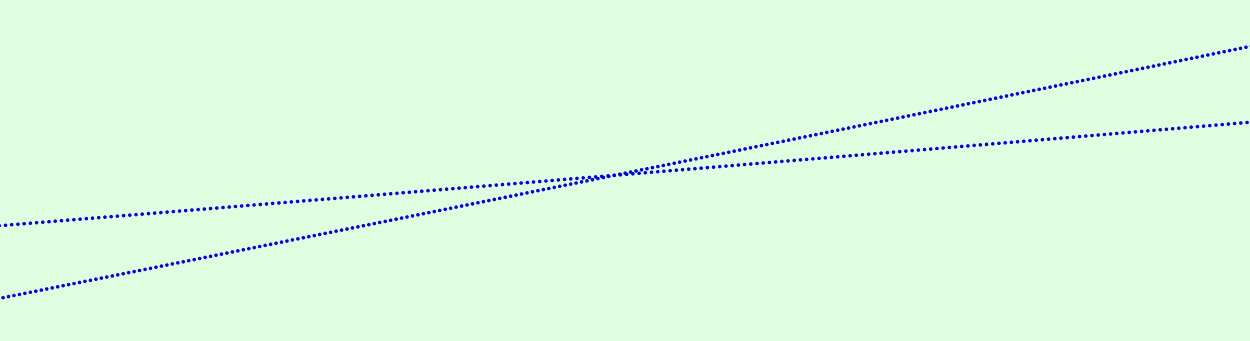}
    \caption{Bifurcation around the $3$-stacked torus.}
    \label{fig:aroundtritorus}
\end{figure}


\subsection{Upper bound}

We can also easily produce a $4$-stacked torus. It is generated by a collision of a $3$-stacked torus with a $2$-stacked torus:

\begin{example}
\label{ex:quadtorus}
Let $\tau = -12.045$ and $t= (t_1, t_2, t_3, t_4)= \left(\frac{1}{2}, \frac{\tau}{2} , \frac{\tau}{3} , \tau \right)$ and consider $\left(J, H_t \right)$ for $j=2$. Then there exists $h \in \R$ such that the fibre $(H_t^{red,2})^{-1}(h)$ is homeomorphic to the curve plotted in Figure \ref{fig:quadtorus} and, as a result, the fibre $(J,H_t)^{-1}(2, h)$ is homeomorphic to a $4$-stacked torus. 
\end{example}

\begin{proof}
Let $\tau = -12.045$ and $t= (t_1, t_2, t_3, t_4)= \left(\frac{1}{2}, \frac{\tau}{2} , \frac{\tau}{3} , \tau \right)$. 
To determine the singular points on $M^{red,2}$, we use Corollary \ref{cor:finalsing} and obtain for the numerator of the expression \eqref{eq:finalsing} evaluated with {\em Mathematica}: 
\begin{align*}
     0 & =  2621.44 - (5.11466*10^7) \  r^2 + (1.35995*10^8) \ r^4 - (1.59003*10^8) \ r^6 \\ 
     & \quad + (1.07988*10^8) \ r^8 - (4.73907*10^7) \ r^{10} + (1.41302*10^7) \  r^{12} \\
     & \quad - (2.92511*10^6) \ r^{14} + 421026 \ r^{16} - 41351 \ r^{18} + 2642.8 \ r^{20} \\
     & \quad - 99.0427 \ r^{22} + 1.65071 \ r^{24}
 \end{align*}
Solving this equation gives us exactly six singular points in the domain of $\phi^{red,2}_1$, namely 
\begin{align*}
 & \xi_1  := ( -1.78207, 0)  , && \xi_2 := (2.23607, 0) , && \xi_3 := (-2.61232, 0), \\
 & \xi_4  := (-2.23607, 0)  , && \xi_5 := (1.83516, 0) , && \xi_6 := (2.5753, 0).
\end{align*}
Using Lemma \ref{lem:hypcondition}, we conclude that $\xi_1, \xi_2, \xi_3 \in M^{red,2}$ are singular points of hyperbolic type of $H_t^{red,2}$ and $\xi_4, \xi_5, \xi_6 \in M^{red,2}$ are singular points of elliptic type of $H_t^{red,2}$. 
Now set $h:= H_t^{red,2}(\xi_1) = H_t^{red,2}(\xi_2)= H_t^{red,2}(\xi_3) = -14.7267$. 
Plotting $(H^{red,2}_t)^{-1}(h)$ we find it to be the curve plotted in Figure \ref{fig:quadtorus}.
The associated generalised bouquet is given by 
$$\left(G := \bigl( H^{red,2}_t\bigr)^{-1}(h), \ S := \emptyset, \ C:=\left\{\xi_1, \xi_2, \xi_3 \right\} \right)$$ 
and therefore, by Theorem \ref{theo:main}, the fibre $(J,H_t)^{-1}(2, h)$ is homeomorphic to a $4$-stacked torus. 
\end{proof}

\begin{figure}[h]
    \centering
    \includegraphics[scale=.7]{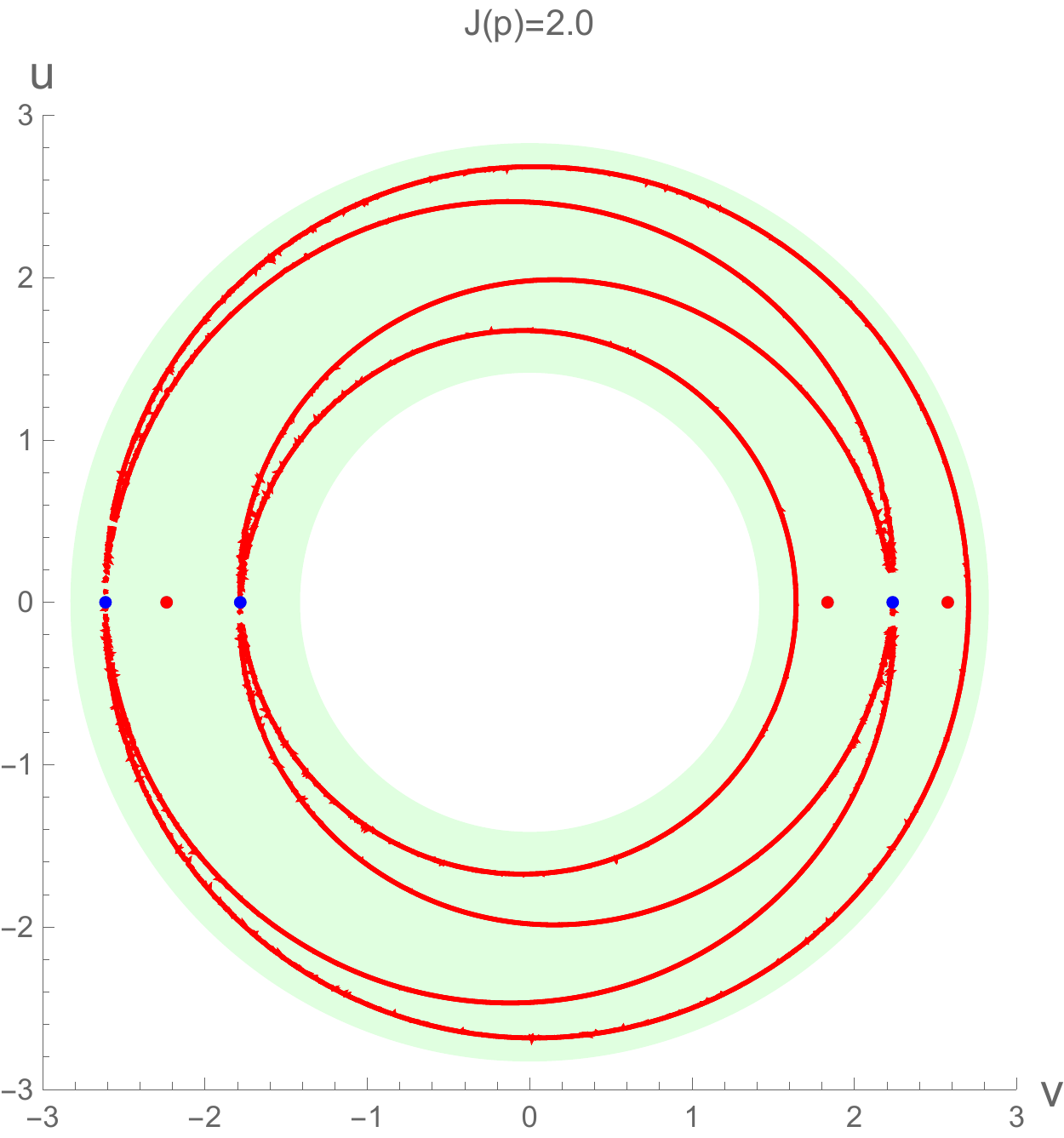}
    \caption{Appearance of a $4$-stacked torus as explained in Example \ref{ex:quadtorus} where $t= (t_1, t_2, t_3, t_4)= \left(\frac{1}{2}, \frac{\tau}{2} , \frac{\tau}{3} , \tau \right)$ with $\tau = -12.045$.}
    \label{fig:quadtorus}
\end{figure}

The natural question now is how far we can push this --- can we get a $5$-stacked torus, a $6$-stacked torus and so on? To answer this questions, we need, among others, the following result: 

\begin{proposition}[{Gullentops \cite[Theorem 3.16]{thesis:yannick}}]
 \label{theo:Hittingbranches}
Let $(N,\om)$ a compact symplectic 2-dimensional manifold and $H: N \to \R$ a smooth Hamiltonian function. 
For every point $x$ in a hyperbolic fibre, $\lim_{t \rightarrow \pm \infty}\Phi^H_t(x)$ is a hyperbolic fixpoint.
\end{proposition}

Now we determine an upper bound $k$ for $k$-stacked tori in our family $(J, H_t)$ of proper $\mbS^1$-systems.

\begin{theorem}
\label{th:maxfibre}
Let $t_1 \neq 0$. Then the number of hyperbolic points in a fibre of $\left( M^{red,j},H_t^{red,j}\right)$ cannot exceed twelve. In particular, there are no $k$-stacked tori possible for $k>13$.  
\end{theorem}

\begin{proof}
{\em Step 1: Obtaining of an upper bound.} Every hyperbolic point is a solution of equation \eqref{eq:finalsing} which, in our situation, takes the form
\begin{align*}
0  =  f_1(r_1, r_2, t)+ \frac{f_2(r_1, r_2, t) }{f_3(r_1, r_2, t)}
\end{align*}
where $r_1= \sqrt{2j}$, and thus $j = \frac{r_1^2}{2}$, and $f_1$, $f_2$ and $f_3$ are polynomials defined on $D_1$ (see Figure \ref{fig:imp1domain} and Equation \eqref{eq:polarcoordinates}) with values in $\R$ given by
\begin{align*}
    f_1(r_1,r_2, t) & := r_2^2 \bigl(-25 + 50 t_1 + 2 (-6 + r_1^2)^2 (-6 + 2 r_1^2 - r_2^2) t_2 + 96 t_3 - 
   152 t_3 r_1^2  + 4 t_3 r_1^4  +\\
   & \quad 12 t_3 r_1^6 + 88 t_3 r_2^2 +  16 t_3 r_1^2 r_2^2 - 26 t_3 r_1^4 r_2^2 - 12 t_3 r_2^4 + 18 t_3 r_1^2 r_2^4  -  4t_3 r_2^6 + 144 t_4 \\
   & \quad- 12t_4 r_1^2  - 8 t_4 r_1^4 + t_4 r_1^6  - 
   36 t_4 r_2^2 + 12 t_4  r_1^2 r_2^2 -  t_4 r_1^4 r_2^2\bigr)^2, \\
   f_2(r_1,r_2, t) & := (-384 + 160 r_2^2 + 324 r_2^4 - 88 r_2^6 - 15 r_2^8 + 3 r_2^{10} + 
   4 r_1^8 (-4 + r_2^2)  + r_1^6 (96 r_2^2 - 18 r_2^4) \\
   & \quad+ 2 r_1^4 (120 - 46 r_2^2 - 75 r_2^4 + 13 r_2^6) +  r_1^2 (-160 - 600 r_2^2 + 192 r_2^4 + 84 r_2^6 - 15 r_2^8))^2 t_1^2, \\
   f_3(r_1,r_2, t) & := (-8 + r_2^2) (-8 + r_1^4 - 2 r_2^2 + r_2^4 - 2 r_1^2 (-1 + r_2^2)) (-12 +  4 r_1^4 + 4 r_2^2 + r_2^4 - 4 r_1^2 (2 + r_2^2)).
\end{align*}
As in Examples \ref{ex:lemniscate}, \ref{ex:tritorus}, \ref{ex:quadtorus} with associated Figures \ref{fig:levelSetDutorus}, \ref{fig:levelSetTritorus}, \ref{fig:quadtorus} it is enough to work on the reduced space $M^{red, j}$. Since $r_1= \sqrt{2j}$ where $j$ is the level of $J$, we are only interested in the properties concerning $r_2$. 
The highest order term in $r_2$ of the numerator of the rational function $f_1(r_1, r_2, t)+ \frac{f_2(r_1, r_2, t) }{f_3(r_1, r_2, t)}$ is $16 t_3^2 r_2^{24} $ which is of degree 24 in the variable $r_2$.

Thus the expression $f_1(r_1, r_2, t)+ \frac{f_2(r_1, r_2, t) }{f_3(r_1, r_2, t)}$ can maximally have 24 zeros in the variable $r_2$ for every value of $r_1$. Therefore $U^{red, j}_1$ can maximally contain 24 singular points.
The chart $U^{red,j}_1$ covers all of $M^{red,j}$ apart from a finite number of points: for $j \leq 1$, the chart misses exactly one point, and, for $j > 1$, the chart misses exactly two points. Therefore $M^{red,j}$ contains maximally 26 singular points.

{\em Step 2: Determining the number of hyperbolic points.} 
Let $(J, H_t)^{-1}(j, h)$ be a hyperbolic-regular fibre and assume that it is connected (otherwise see {\em Step 11}).
The connected components of $ M^{red,j} \setminus \bigl(H^{red,j}_t\bigr)^{-1}  (h)$ are open in $ M^{red,j}$ and are referred to as {\em faces}. We consider $\bigl(H^{red,j}_t \bigr)^{-1}    (h)$ as graph where the vertices are given by the hyperbolic singular points $\Hyp \left(H_t^{red,j}\right)$ and the edges by the connected components of $\bigl( H^{red,j}_t \bigr)^{-1} (h) \setminus \Hyp \left(H_t^{red,j}\right)$.

Now we will show that each face contains at least one singular point of $H_t^{red,j}$. 
Let $\mcC$ be a face and $\overline{\mcC}$ its closure. Since $M^{red,j}$ is homeomorphic to a 2-sphere $\overline{\mcC}$ is compact. Therefore $H_t^{red,j}$ attains a maximum and a minimum on $\overline{\mcC}$. 
By definition, $H_t^{red,j}$ is constant on $\bigl(H^{red,j}_t \bigr)^{-1}(h)$. Since the boundary $\overline{\mcC} \setminus \mcC$ is a subset of $\bigl(H^{red,j}_t \bigr)^{-1}(h)$ the function $H_t^{red,j}$ is either constant on $\overline{\mcC} $ or has at least one extremum in $\mcC$. If $H^{red,j}_t$ is smooth in the extremum then the extremum is attained in a singular point of $H^{red, j}_t$. The only points where $H^{red,j}_t$ is not smooth are the invariant singular points (see Theorem \ref{the:invariablesingpoint}). Therefore the face $\mcC$ must contain at least one singular point.

Euler's formula for connected graphs with finite number of faces, edges, and vertices in $\R^2$ is given by
\begin{equation}
 \label{eq:eulerformula}
 \abs{faces} - \abs{edges} + \abs{vertices} = 2.
\end{equation}
By definition, the vertices of the graph induced by $\bigl(H^{red,j}_t \bigr)^{-1}(h)$ are the hyperbolic singular points. Since a hyperbolic point $p$ in the plane has exactly one stable and one unstable manifold a small neighbourhood of $p$ in $\bigl(H^{red,j}_t \bigr)^{-1}(h) \setminus \{p\}$ has 4 connected components of which two come from the stable manifold (`stable branches') and two from the unstable one (`unstable branches'). These are the ends of the edges meeting in $p$. By Theorem \ref{theo:Hittingbranches} the two ends of an edge are always branches. Therefore 
$$
\abs{edges} = \frac{4 \ \abs{vertices}}{2} = 2 \ \abs{vertices}.
$$
By definition of the graph, we obtain 
$$
2 \ \left| \Hyp \left(H_t^{red,j}\right) \right| = \abs{edges}.
$$
Together with Euler's formula \eqref{eq:eulerformula}, this leads to
$$
2 = \abs{faces} - 2 \ \left| \Hyp \left(H_t^{red,j}\right) \right| + \left| \Hyp \left(H_t^{red,j}\right) \right| = \abs{faces} -  \ \left| \Hyp \left(H_t^{red,j}\right) \right|
$$
and therefore 
\begin{equation}
\label{eq:numberFaces}
\abs{faces} = \left| \Hyp \left(H_t^{red,j}\right) \right| + 2.
\end{equation}
Since each face contains at least one singular point and since the total of singular points is maximally 26, we deduce
$$
  26 \geq \abs{faces} +  \left| \Hyp \left(H_t^{red,j}\right) \right| = 2 \ \left| \Hyp \left(H_t^{red,j}\right) \right| + 2.
$$
which results in 
$$
12 \geq \left| \Hyp \left(H_t^{red,j}\right) \right|.
$$

{\em Step 3: Location of singular points.} Recall that $M^{red, j}$ is diffeomorphic to a 2-sphere for $j \neq \{1, 2\}$ and homeomorphic to a 2-sphere for $j \in \{1, 2\}$. 
The chart $(U_1^{red, j}, \phi^{red, j}_1)$ covers all but one point of $M^{red, j}$ if $j< 1$ and all but two points if $j \geq 1$. By Proposition \ref{prop:dj}, all singular points $(x + iy, u + iv) $ in the chart $(U_1,\phi_1)$ satisfy $v=0$. This property descends to the reduced chart $\bigl(U_1^{red, j},  \phi^{red, j}_1\bigr)$. Therefore all these points lie on a line $L \subset \C$, more precisely in $L \cap\ \left(\phi_1^{red, j}\right) ^{-1}\bigl( U_1^{red, j} \bigr)$. The image $\phi_1(L)$ consists of a circle minus one or two points in $M^{red, j}$. We denote by $\mcS_L$ the closure of $\phi_1(L)$, i.e., it includes these missing points. Note that the points missing in $\phi_1(L)$ coincide with the points missed by the chart $\phi_1$. Thus any singular point on $M^{red, j}$ must lie on $\mcS_L$. 

{\em Step 4: Whereabouts of the faces.} Recall that each face contains a singular point. Thus each face has to intersect $\mcS_L$. 

{\em Step 5: Reflection along $\mcS_L$.} From the definition of $H_t$ and $H_t^{red,j}$, we obtain that $H^{red,j}_t(\phi_1^{red,j}(u+iv)) = H^{red,j}_t(\phi_1^{red,j}(u-iv))$, meaning, the function $ H^{red,j}_t \circ \phi_1^{red,j}$ is invariant under reflection along the line $L$ containing the singular points. Thus $ H^{red,j}_t$ is invariant under reflection along the circle $\mcS_L$ containing the singular points. Therefore the `graph' $\bigl(H^{red,j}_t \bigr)^{-1} (h)$ is invariant under this reflection.

{\em Step 6: Transverse intersections with $\mcS_L$.} 
A Hamiltonian solution $\ga$ of $H^{red,j}_t$ is said to be {\em heteroclinic} from the hyperbolic point $p^-$ to the hyperbolic point $p^+$ if $\lim_{s \to \pm \infty} \ga(s) = p^\pm$. A Hamiltonian solution is said to be {\em homoclinic} if it is heteroclinic with $p^- = p^+$.

Endow the line $L$ with the ordering induced by the ascending real coordinate and induce a (circular) ordering on $\mcS_L$.

Let $ \pti \in \mcS_L \cap\ \bigl(H^{red,j}_t \bigr)^{-1} (h)$ and let $V$ be a sufficiently small neighbourhood of $\pti$ in $\bigl(H^{red,j}_t \bigr)^{-1} (h)$. 
The number of connected components of $V \setminus\ \{\pti\}$ must be even since $\pti$ is either hyperbolic (causing four components) or regular (causing two components).

We now show that these connected components of $V \setminus\ \{\pti\}$ can never lie on $\mcS_L$. We argue by contradiction: Assume that $\pti \in \mcS_L \cap\ \bigl(H^{red,j}_t \bigr)^{-1} (h)$ where two of the components of $V \setminus\ \{ p'\} $ are contained in $\mcS_L$. 
If $\pti$ is regular then the flow line $s \mapsto \Phi_s^{H^{red,j}_t}(\pti)$ must stay in $\mcS_L$ for all $s \in \R$ due to the invariance of $\bigl(H^{red,j}_t \bigr)^{-1} (h)$ under reflection (otherwise the solution would have to split into two branches in contradiction to uniqueness of ODEs). Since $\dim_\R (\mcS_L)=1$ it is either heteroclinic or homoclinic to some hyperbolic points. 

In particular, $s \mapsto \Phi_s^{H^{red,j}_t}(\pti)$ lies on one branch of the 1-dimensional stable and unstable manifolds of these hyperbolic points. These branches therefore have to lie on $\mcS_L$. Since $\bigl(H^{red,j}_t \bigr)^{-1} (h)$ is invariant under the above mentioned reflection, the other branch has to lie on $\mcS_L$, too. Therefore the whole stable or unstable manifolds of the hyperbolic points adjacent to $\pti$ on the circle lie on $\mcS_L$. Iterating this argument, we obtain that $\mcS_L$ must be a subset of $\bigl(H^{red,j}_t \bigr)^{-1} (h)$. Therefore no face intersects $\mcS_L$. This leads to a contradiction with the result of {\em Step 4}.

Therefore all intersections of $\bigl(H^{red,j}_t \bigr)^{-1} (h)$ with $\mcS_L$ are transverse in the following sense: if the intersection point is hyperbolic, the stable and unstable manifolds both intersect the circle transversely. If the intersection point is regular then, due to invariance under reflection, the intersection is in fact perpendicular and the intersection point must lie on a homoclinic orbit.   

{\em Step 7: Existence of a homoclinic orbit.}
We want to show that there exists at least one homoclinic orbit. We argue by contradiction: assume that there is no homoclinic orbit. This means in particular that the only intersections of $\bigl(H^{red,j}_t \bigr)^{-1} (h)$ with $\mcS_L$ are precisely the hyperbolic points. These cut $\mcS_L$ into as many segments as there are hyperbolic points. Moreover, each segment between two adjacent hyperbolic points lies completely in a face. By equation \eqref{eq:numberFaces}, we have two more faces than hyperbolic points. Therefore there exist at least two faces which do not intersect the circle which contradicts the result of {\em Step 4}. 

{\em Step 8: The generalised bouquet.}
By Theorem \ref{theo:Hittingbranches}, every flow line begins and ends at a hyperbolic point, i.e., all flow lines are heteroclinic or homoclinic. Consider a hyperbolic point with a homoclinic orbit. Then either the remaining stable and unstable branch are connected by another homoclinic orbit or not. In the latter case, both are connected via heteroclinic orbits to hyperbolic points. Due to invariance under reflection, these hyperbolic points must coincide. Now repeat this argument for this hyperbolic point. Since there are only finitely many hyperbolic points this procedure terminates after finitely many steps with an homoclinic orbit.
Intuitively, $\bigl(H^{red,j}_t \bigr)^{-1} (h)$ thus looks like a `loop chain' as in Figure \ref {subfig:threeLoops}.

{\em Step 9: The shape of the fibre.}
By Theorem \ref{theo:main} and Theorem \ref{the:equalperiod}, we obtain the fibre $(J, H_t)^{-1}(j, h)$ by taking the product of the `loop chain' $\bigl(H^{red,j}_t \bigr)^{-1} (h)$ with $\mbS^1$. 

A `loop chain' containing precisely one hyperbolic point leads to a $2$-stacked torus. More generally, a `loop chain' with $k$ hyperbolic points leads to a $(k+1)$-stacked torus, see Figure \ref{fig:hypRegFibres} for $2$-stacked and $3$-stacked ones.

{\em Step 10: Working on $\phi_1$ is sufficient.}
Note that we need not look at other charts than $\phi_1$ since we accounted for all points missed by the chart $\phi_1$ in the steps above. 

{\em Step 11: Case of graphs with several connected components.}
Concerning {\em Step 2}, if a graph has $N$ connected components, Eulers formula becomes  
$$\abs{faces} - \abs{edges} + \abs{vertices} = N+1,$$
see Wilson \cite[Corollary 13.3]{wilson10}. Thus our estimates continue to hold true since $2 \leq N+1$ for $ 1 \leq N$. In all subsequent {\em Steps}, one may always work componentwise.
\end{proof}

The content of the proof of Theorem \ref{th:maxfibre} gives rise to an interesting observation:

\begin{corollary}
Each connected components of $ M^{red,j} \setminus \bigl(H^{red,j}_t\bigr)^{-1}(h)$ contains at least one singular point of $H_t^{red,j}$. 
\end{corollary}

\begin{proof}
 See {\em Step 2} in the proof of Theorem \ref{th:maxfibre}.
\end{proof}


\section{Bifurcations of singular points}
\label{sec:singbif}

In this section, we will study how the position and type of singular points of $(J, H_t)$ change when the parameter $t=(t_1,t_2,t_3,t_4)$ varies. Here both the rank zero singular points and the rank one singular points display interesting behavior. We will not only consider non-degenerate but also degenerate singular points.


\subsection{Bifurcation diagram}

Given a $4$-dimensional completely integrable system $(M, \om, \mcF)$, the image of the momentum map $\mcF(M)$ is a subset of $\R^2$. When decorated with the singular values of the system, we rather refer to it as {\em bifurcation diagram}.   
The {\em leaf space} of $(M, \om, \mcF)$ is the space $\faktor{M}{\sim}$, where $p \sim q$ if and only if $p$ and $q$ belong to the same connected component of a fibre of $\mcF$. 

In the classification of toric and semitoric systems, the image of the momentum map plays a crucial role: for toric systems, it is the only invariant and, for semitoric systems, it is used to construct the so-called polygon invariant. For hypersemitoric systems or even more general $\mbS^1$-systems, there is not yet any symplectic classification in the spirit of the ones for toric and semitoric systems. In particular, it is not yet clear how the image of the momentum map would/will be incorporated: The main problem is that fibres of hypersemitoric systems or even more general $\mbS^1$-systems are not necessarily connected.

Since the fibres of toric and semitoric systems on connected manifolds are connected, the image of the momentum map can be identified with the leaf space.

\begin{definition}
\label{de:unfbifd}
    An {\em unfolded bifurcation diagram (short UBD)} of a $4$-dimensional completely integrable system $(M, \om, \mathcal{F})$ is a path connected topological space $\mathcal{UBD}$ together with a projection $\pi:\mathcal{UBD} \rightarrow \mathcal{F}(M)$ such that, for all $r \in \R^2$, the number of points in $\pi^{-1}(r)$ equals the number of connected components of $\mathcal{F}^{-1}(r)$.
\end{definition}

If there exists a continuous embedding $i:\mathcal{F}(M) \hookrightarrow M$ that is a right inverse of $\mathcal{F}$ then the leaf space forms a {\em canonical} unfolded bifurcation diagram. 

Unfolded bifurcation diagrams and leaf spaces are essential tools when studying the dynamics of hypersemitoric and proper $\mathbb{S}^1$-systems since they allow for visualization of features that the image of the momentum map cannot provide (for instance displaying so-called flaps and swallowtails, see Figure \ref{fig:flapAndSwallow} and Section \ref{subsec:flaps}).


\subsection{Flaps and swallowtails}
\label{subsec:flaps}

Intuitively, a flap (cf.\ Figure \ref{fig:unfolded flap} and Figure \ref{fig:flapInBifurDiagr}) can be thought of as triangular shaped momentum map image of a toric system that is attached to another integrable system along one side consisting of hyperbolic-regular values and ending in parabolic values. The other two sides of the triangle consist of elliptic-regular values that become an elliptic-elliptic value at the common vertex of these sides.

Thus flaps are easier visualised in an {\em unfolded} bifurcation diagram than just in the usual bifurcation diagram. The hyperbolic-regular values are thus located where the unfolded bifurcation diagram splits into different branches.
On the image of the momentum map, a flap looks like in Figure \ref{fig:flapInBifurDiagr}.

\begin{figure}[h]
\centering
    \includegraphics[scale = 0.5]{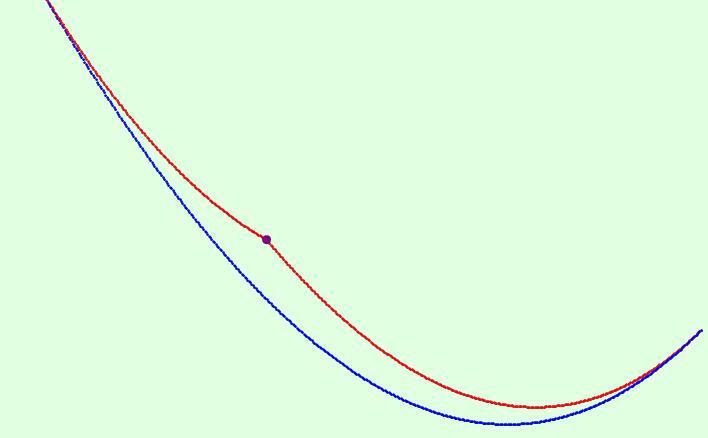}
\caption{This plot shows a flap in the bifurcation diagram of $(J,H_t)$ for $t=\bigl(\frac{7}{20},\frac{7}{20},\frac{7}{30}, \frac{7}{10}\bigr)$. The blue line consists of hyperbolic-regular values, the red lines of elliptic-regular values and the dark dot is an elliptic-elliptic value. The elliptic-elliptic point is $\left[\sqrt{2}, \ 0, \ 0, \sqrt{2}, \ 2, \ 2 \sqrt{2}, \sqrt{6}, \sqrt{6}\right]$.}
\label{fig:flapInBifurDiagr}
\end{figure}

\begin{example}
    \label{ex:singleflap}
For the parameter value $t=\left(\frac{7}{20},\frac{7}{20},\frac{7}{30}, \frac{7}{10} \right)$, the system $(J, H_t)$ has the flap plotted in Figure \ref{fig:flapInBifurDiagr}.
\end{example}

Dullin $\&$ Pelayo \cite{Dulinhyper} show how to create a small flap from a focus-focus point in a $4$-dimensional integrable system in the presence of an $\mbS^1$-action. Flaps play an essential role in Hohloch $\&$ Palmer \cite{hohloch2021extending} when extending $\mbS^1$-actions to (as nice as possible) $4$-dimensional integrable systems.
Note that, by using so-called blow-up or cutting techniques (see for instance Hohloch $\&$ Palmer \cite{hohloch2021extending}), a flap can be modified to contain several elliptic-elliptic points. This has a natural impact on its shape in the sense that the flap gets more vertices, see Figure \ref{fig:doubleflapunf}.

\begin{figure}[h]
    \centering
    \includegraphics[trim={3.5cm 4.5cm 3.5cm 3.5cm},clip,scale=.45]{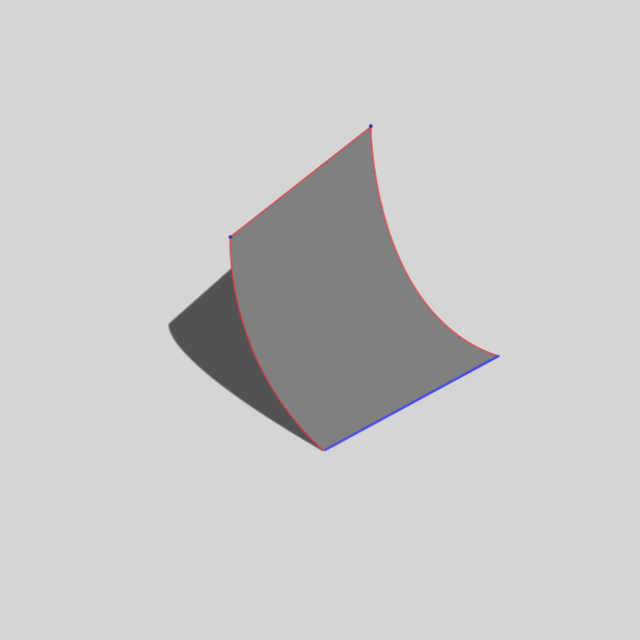}
    \caption{A flap displayed via the unfolded bifurcation diagram with two elliptic-elliptic values as vertices along the green line consisting of elliptic-regular values.}
    \label{fig:doubleflapunf}
\end{figure}

We will now see that the family $(J, H_t)$ contains, for certain values of $t$, flaps of various shapes:

\begin{example}
\label{ex:doubleflap}
For the parameter value $t=\left(\frac{ 9}{20},\frac{9}{20},\frac{9}{30}, \frac{9}{10} \right)$, the system $(J, H_t)$ has a flap with two elliptic-elliptic vertices as plotted in Figure \ref{fig:Bigflap}. This flap appears as the result of the collision of two flaps.
\end{example}

De Meulenaere $\&$ Hohloch \cite{meulenaere2019family} perturb the toric octagon system $(J, H)$ by means of a 1-dimensional family of parameters to generate focus-focus points from elliptic-elliptic ones and to observe the collision of two fibres with one focus-focus point contained in each in order to create one fibre containing 2 focus-focus points. Since their perturbation is included as $t=(t_1, 0, 0, 0) \in \R \times\ \{0\} \times\ \{0\}\times\ \{0\}$ in our more general family parametrised by $t=(t_1, t_2, t_3, t_4 ) \in \R^4$ it is only natural to expect the creation of focus-focus points also in our family $(J, H_t)$ for certain values of $t$.

\begin{figure}[h]
    \centering
    \includegraphics[scale = 1.]{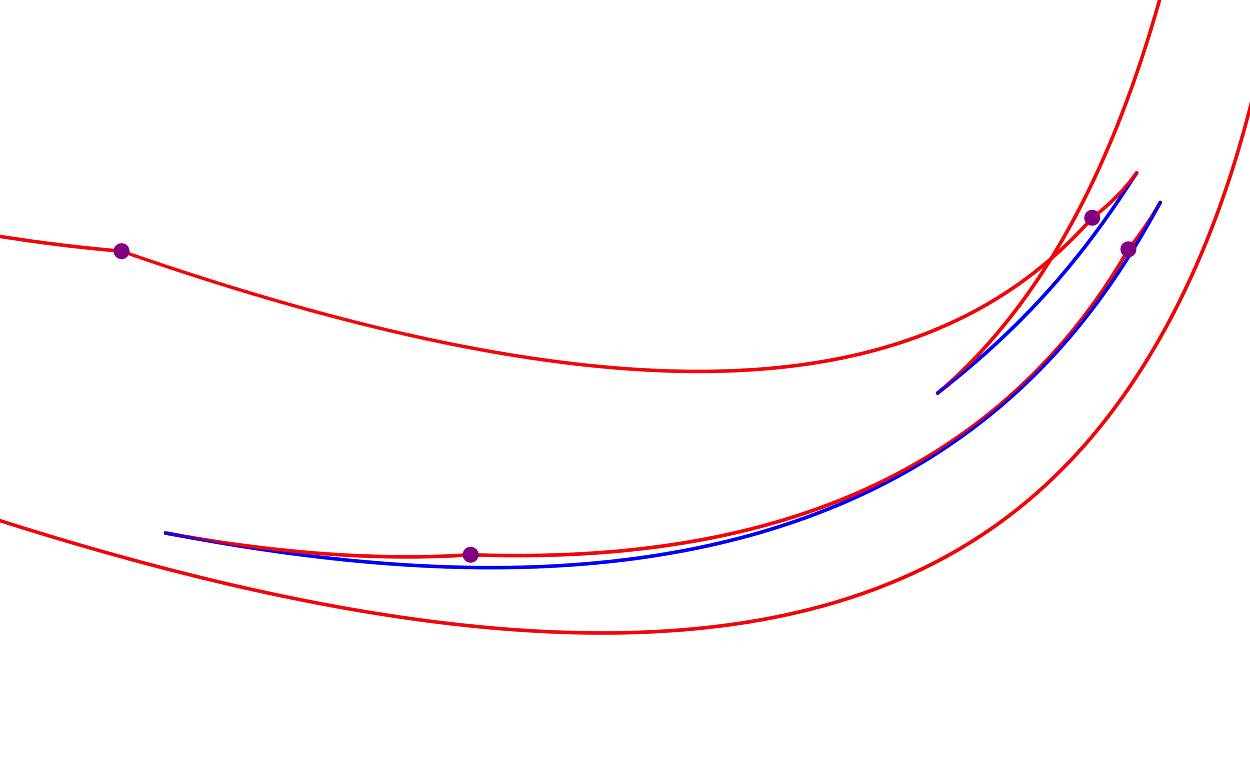}
    \caption{Part of the bifurcation diagram of $(J, H_t)$ for $t=\left(\frac{ 9}{20},\frac{9}{20},\frac{9}{30}, \frac{9}{10} \right)$ where two smaller flaps collided and a flap with two elliptic-elliptic vortices appears. The red lines consist of elliptic-regular values and the blue ones of hyperbolic-regular values and the dots mark elliptic-elliptic values.}
    \label{fig:Bigflap}
\end{figure}

In our system $(J, H_t)$, singular points with $J$-value equal to zero or three can never be of focus-focus type since their image never lies in the interior of the image of the momentum map (cf.\ Theorem \ref{th:locNF}).

Transitions between elliptic-elliptic and focus-focus singularities are often called {\em Hamiltonian Hopf bifurcations} and are, for instance, discussed in Van der Meer \cite{vandermeer1985hamiltonian}: Such a bifurcation is called {\em supercritical} when the focus-focus value is generated by passing of an elliptic-elliptic value from a boundary of the bifurcation diagram (see Figure \ref{fig:transfocfoc}) to the interior and {\em subcritical} when the focus-focus point transitions to an `elliptic-elliptic point on a flap' (see Figure \ref{fig:transhyp}).

\begin{figure}[h]
    \centering
\begin{subfigure}[t]{.98\textwidth}
\centering
   \includegraphics[scale = 1.]{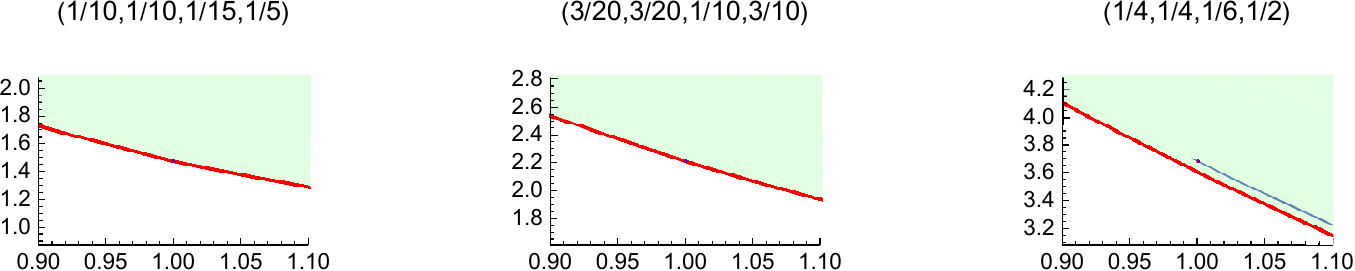}
    \caption{Transition from elliptic-elliptic to focus-focus.}
    \label{fig:transfocfoc}
 \end{subfigure}
\vspace{4mm}
 
\begin{subfigure}[t]{.98\textwidth}
\centering
\includegraphics[scale = 1.]{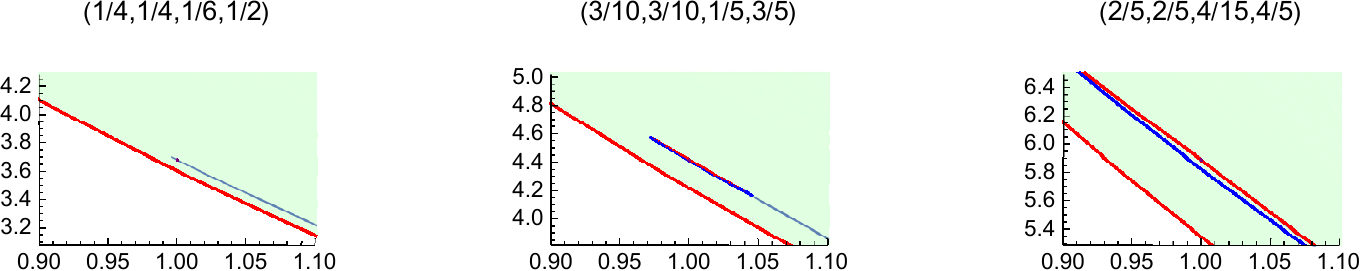}
\caption{Transition from focus-focus to elliptic-elliptic with flap.}
 \label{fig:transhyp}
\end{subfigure}

\caption{Transitions of the point $\phi_2(0,0)$. The associated values of the parameter $t=(t_1,t_2,t_3,t_4)$ are displayed above each subfigure.}
\label{fig:transitionsWithParam}
\end{figure}

Now we consider an example where all the invariant singular points become focus-focus type for a certain interval and then, when they transition back to elliptic-elliptic points, in addition a flap will be attached to them.

\begin{example}
Let $p:=\left[\sqrt{2}, \ 0, \ 0, \sqrt{2}, \ 2, \ 2 \sqrt{2}, \sqrt{6}, \sqrt{6}\right]$ (which is an invariant singular point) and consider the 1-parameter subfamily $\left(J,H_{\left(\frac{\tau}{2},\frac{\tau}{2},\frac{\tau}{3}, \tau \right)} \right)$ for $\tau \in \R$. 
For $\tau < \frac{25}{69}$, the point $p$ is elliptic-elliptic and transitions at $\tau = \frac{25}{69}$ to a focus-focus point (see Figure \ref{fig:transfocfoc}). Eventually at $\tau = \frac{5}{9}$, the point $p$ transitions to a elliptic-elliptic point but with a flap attached (see Figure \ref{fig:transhyp}). 
\end{example}

\begin{proof}
 To find the values of $\tau$ such that $\phi_2(0,0)$ is a degenerate singular point we calculate the eigenvalues of $\omega^{-1}_{\phi_2(0,0)}d^2 H_{\left(\frac{\tau}{2},\frac{\tau}{2},\frac{\tau}{3}, \tau \right)}$ at $\phi_2(0,0)$. These eigenvalues are given by
 {\small
 \begin{align*}
 \lambda_1 &= -\frac{1}{25} \sqrt{-\frac{625}{2} + 6000 \tau - \frac{91017 \tau^2}{2} + \left(
  \frac{25}{2}  - \frac{423}{2} \tau \right) \sqrt{625 - 2850 \tau + 3105 \tau^2}},\\
  \lambda_2 &= \frac{1}{25} \sqrt{-\frac{625}{2} + 6000 \tau - \frac{91017 \tau^2}{2} + 
  \left(  \frac{25}{2} - \frac{423}{2} \tau \right) \sqrt{625 - 2850 \tau + 3105 \tau^2}},\\
  \lambda_3 &= -\frac{1}{25} \sqrt{-\frac{625}{2} + 6000 \tau - \frac{91017 \tau^2}{2} -  
  \left( \frac{25}{2} -  \frac{423}{2} \tau \right) \sqrt{625 - 2850 \tau + 3105 \tau^2}},\\
  \lambda_4 &= \frac{1}{25} \sqrt{-\frac{625}{2} + 6000 \tau - \frac{91017 \tau^2}{2}
  - \left(  \frac{25}{2} - \frac{423}{2} \tau \right) \sqrt{625 - 2850 \tau + 3105 \tau^2}}.
 \end{align*}
}
These eigenvalues only vanish for parameter value $\tau =0$ where the system actually satisfies $(J,H_0)=(J, H)$, where, due to the contribution of $d^2 J$, the point remains nondegenerate. Thus $\phi_2(0,0)$ can only be degenerate if two or more of these eigenvalues coincide.
This is equivalent with $625 - 2850 \tau + 3105 \tau^2 =0$ of which the solutions are $\tau = \frac{25}{69}$ and $\tau=\frac{5}{9}$.

To verify that these two values corresponds to degenerate singularities we calculate the eigenvalues of $\om^{-1}(c_1 d^2J+ c_2 d^2H_t )$. The eigenvalues for $\tau = \frac{25}{69}$ are 
$$ \frac{\sqrt{-1250 c_1^2 + \frac{205000 c_1 c_2}{23} - \frac{8405000 c_2^2}{529}}}{25 \sqrt{2}}
\quad \mbox{and} \quad 
-\frac{\sqrt{-1250 c_1^2 + \frac{205000 c_1 c_2}{23} - \frac{8405000 c_2^2}{529}}}{25 \sqrt{2}}
$$
and for $\tau = \frac{5}{9}$ 
$$
\frac{\sqrt{-1250 c_1^2 + 13000 c_1 c_2 - 33800 c_2^2}}{ 25 \sqrt{2}}
\quad \mbox{and} \quad
-\frac{\sqrt{-1250 c_1^2 + 13000 c_1 c_2 - 33800 c_2^2}}{ 25 \sqrt{2}}.
$$
All the above eigenvalues appear in multiplicity two, thus proving degeneracy.
\end{proof}

Now consider the following statement which is of more general nature:

\begin{proposition}
\label{prop:noHypEll}
The invariant singular points of rank zero $\phi_2(0,0),\phi_3(0,0),\phi_6(0,0)$ and $\phi_7(0,0)$ from Theorem \ref{the:invariablesingpoint} can never be of hyperbolic-elliptic type, i.e., there is no parameter value $t \in \R^4$ such that these points are hyperbolic-elliptic fixed points of $(J, H_t)$.
\end{proposition}

\begin{proof}
Recall from the discussion after Theorem \ref{th:locNF} that a non-degenerate singular point is hyperbolic-elliptic if the eigenvalues of 
$\om^{-1}(c_1 d^2J+ c_2 d^2H_t )$ are of the form $\alpha$, $-\alpha$, $i \beta$,  $- i \beta$ with $\al, \be \in \mathbb{R}^{\neq 0}$. 
Thus $\det\left(\om^{-1}(c_1 d^2J+ c_2 d^2H_t )\right)= -\al^2 \be^2 <0$ is strictly negative. 

Now consider $p \in \{\phi_2(0,0),\phi_3(0,0),\phi_6(0,0), \phi_7(0,0) \}$ and find the value of the determinant $\det\left(\om^{-1}_p \bigl(c_1 d^2J(p)+ c_2 d^2H_t(p) \bigr)\right)$. We get for $k \in \{2, 3, 6, 7\}$
\begin{align*}
    \det\left(\om^{-1}_{\phi_k(0,0)}  \bigl(c_1 d^2J+ c_2 d^2H_t \bigr)(\phi_k(0,0)) \right)
    & = \frac{1}{390625} \ \bigl(f_k(t_1,t_2,t_3,t_4,c_1,c_2)\bigr)^2 \geq 0
\end{align*}    
where $f_k$ are real polynomials for $k \in \{2, 3, 6, 7\}$ given by
\begin{align*}
   f_2 (t_1,t_2,t_3,t_4,c_1,c_2) 
   & = 625 c_1^2+25 c_1 c_2 (-25 +50 t_1-256t_2-384t_3-192 t_4)  \\
   & \quad +48 c_2 ^2(3t_1^2+320t_2^2+75t_3+720t_3^2  +75t_4+864t_3t_4 \\
   & \quad +144t_4^2  + t_2(50+960t_3+576t_4)-50t_1(2t_2+3(t_3+t_4))) , \\
 f_3(t_1,t_2,t_3,t_4,c_1,c_2)  
    &= 625 c_1^2-25 c_1 c_2 (-25+50t_1+168t_4) \\
    & \quad  +72 c_2^2(2t_1^2+50t_1t_4+t_4(-25+96t_4)), \\   
f_6(t_1,t_2,t_3,t_4,c_1,c_2)    
    &= 625 c_1^2+25 c_1 c_2 (-25+50t_1-432t_4) \\
    & \quad +48 c_2^2(3t_1^2-200t_1t_2+20t_2(5+48t_2)), \\
f_7(t_1,t_2,t_3,t_4,c_1,c_2)    
   &= 625 c_1^2 - 125 c_1 c_2 (-5 + 10 t_1 + 256 t_2)  \\
   & \quad +   48 c_2^2 (3 t_1^2 + 800 t_1 t_2 + 16 t_2 (-25 + 512 t_2)).
\end{align*}
Note that, for all $k \in\{2, 3, 6, 7\}$, the determinant above is always non-negative. Therefore $\phi_2(0,0)$, $\phi_3(0,0)$, $\phi_6(0,0)$ and $\phi_7(0,0)$ can never be of hyperbolic-elliptic type.
\end{proof}

After discussing flaps, let us now consider another bifurcation phenomenon. Intuitively, a {\em swallowtail} (or {\em pleat}) of a $4$-dimensional completely integrable system is created by `pulling a part of the elliptic-regular boundary' of the image of the momentum map over another part of the `elliptic-regular boundary' resulting in the situation displayed in Figure \ref{fig:swallowtail}.

\begin{figure}[h]
    \centering
    \includegraphics[angle = 0,scale = 1.21]{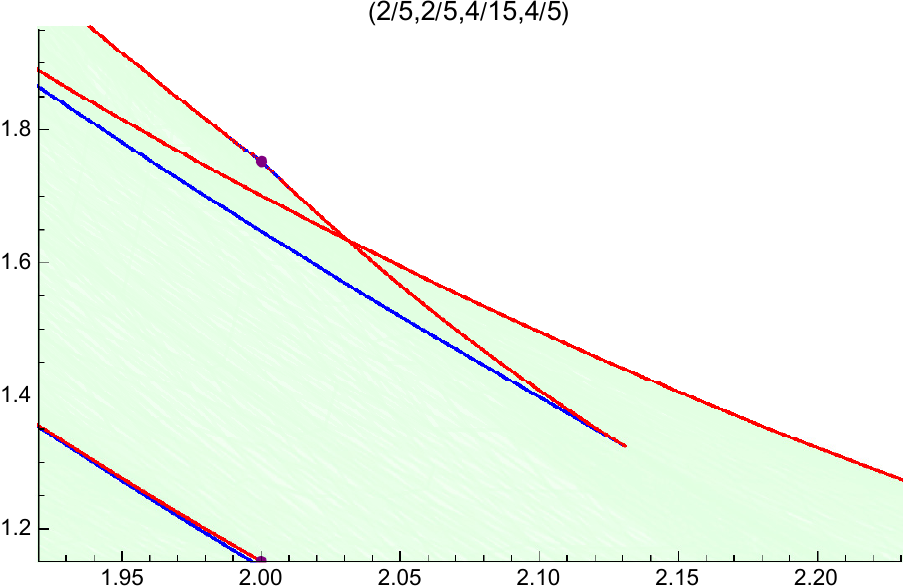}
    \caption{A swallowtail: red lines consist of elliptic-regular values and the blue one of hyperbolic-regular values. There is no elliptic-elliptic value where the two red lines pass over each other.}
    \label{fig:swallowtail}
\end{figure}

In a swallowtail, the point where the lines of elliptic-regular values seem to intersect in the bifurcation diagram is no elliptic-elliptic value but just the overlapping of two elliptic-regular values, i.e., in the unfolded bifurcation diagram, there is no intersection at all. The alternative name {\em pleat} becomes clear when considering a swallowtail in the unfolded bifurcation diagram as drawn in Figure \ref{fig:swallowtailunf}. 
Swallowtails are discussed in more detail in the works by Sadovskii $\&$ Zhilinskii \cite{swallowtailsSadov} and Efstathiou $\&$ Sugney \cite{Efstathiou_2010}.
Our family $(J, H_t)$ contains swallowtails for certain values of the parameter $t$:

\begin{example}
\label{ex:swallowtail}
 In $(J, H_t)$, swallowtails appear as plotted in Figure \ref{fig:swallowtailgeneration} in the 1-parameter family $\tau \mapsto t(\tau):= \left( \frac{\tau}{2} , \frac{\tau}{2} , \frac{\tau}{3} ,\tau \right)$. The swallowtail is found by observing the boundary of the bifurcation diagram and tracing where it gets pulled over itself. 
\end{example}

\begin{figure}[h]
    \centering
    \includegraphics[scale = 1.1]{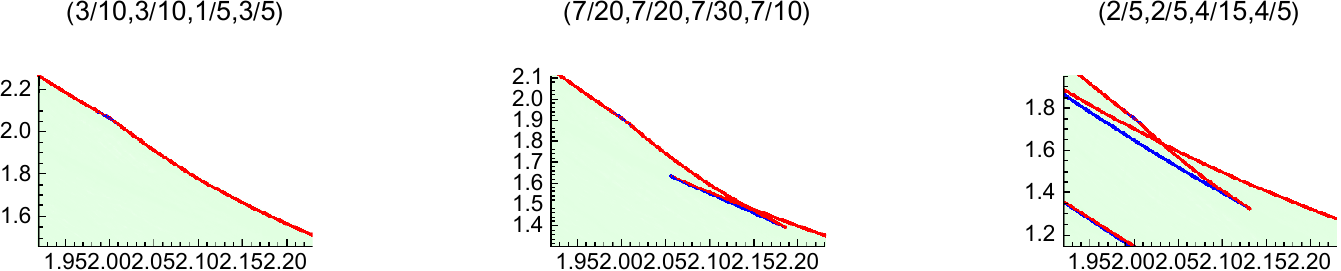}
    \caption{A swallowtail appearing along the family $\tau \mapsto t(\tau):= \left( \frac{\tau}{2} , \frac{\tau}{2} , \frac{\tau}{3} ,\tau \right)$.}
    \label{fig:swallowtailgeneration}
\end{figure}

Now we are interested in what happens when a flap and a swallowtail collide. In particular, we would like to know if this collision result in a swallowtail, a flap or maybe something entirely different.

\begin{example}
\label{ex:flapcol}
 Looking at Example \ref{ex:doubleflap} and Figure \ref{fig:Bigflap}, we see that there two smaller flaps collided and, after the collision, formed a bigger flap with more vertices. Later on, the swallowtail will meet the flap and the bifurcation diagram becomes very complicated, see Figure \ref{fig:collisions}. This collision happens between the parameter values $t_l = \bigl(\frac{119}{200},\frac{119}{200},\frac{119}{300},\frac{119}{100}\bigr)$ and $t_r = \bigl(\frac{3}{5},\frac{3}{5},\frac{2}{5},\frac{6}{5}\bigr)$.
\end{example}

\begin{figure}[h]
    \centering
    \includegraphics[scale=.07]{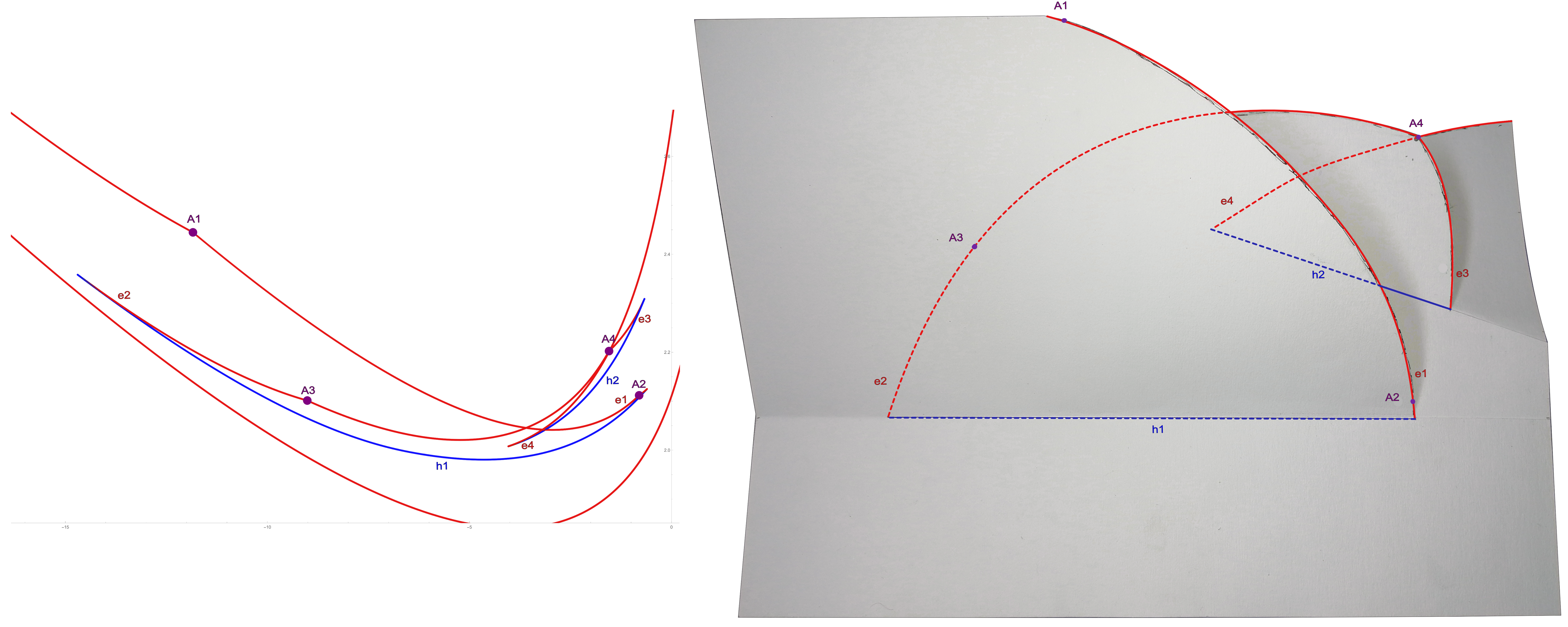}
    \caption{On the right, a schematic sketch of a part of the unfolded bifurcation diagram at $t=\left( \frac{3}{5}, \frac{3}{5}, \frac{2}{5},\frac{6}{5}\right)$ of $(J,H_t)$. On the left, a {\em Mathematica} plot of the associated singular points.}
    \label{fig:collisions}
\end{figure}


\subsection{Occurence of rank one singular points}

Recall that, on four-dimensional manifolds, there are only two types of non-degenerate rank one singular points possible, namely elliptic-regular and/or hyperbolic-regular ones. 
We will now show that both types appear for certain values of $t \in \R^4$ in $(J, H_t)$.

\begin{example}
 For $t=\left(\frac{7}{20}, \frac{7}{20},\frac{7}{30},\frac{7}{10} \right)$, the system $(J, H_t)$ has (among others) elliptic-regular and hyperbolic-regular singular points as sketched in Figure \ref{fig:singpoint}.
\end{example}

\begin{proof}
Fill the parameter values $t=\left(\frac{7}{20}, \frac{7}{20},\frac{7}{30},\frac{7}{10} \right)$ into equation \ref{eq:finalsing} which leads to 
 \begin{align*}
     0 = (\partial_{r_2}\Gamma_{\left(\frac{7}{20}, \frac{7}{20},\frac{7}{30},\frac{7}{10} \right)}(r_1,r_2))^2 - \left(\frac{7}{20}\right)^2(\partial_{r_2}\overline{\Omega}(r_1,r_2))^2.
 \end{align*}
Now calculate its zeros with {\em Mathematica}. This yields singular points which, by Theorem \ref{prop:dj}, lie on the hypersurface $u=0$. The situation is plotted in Figure \ref{fig:singpoint}.
\end{proof}

\begin{figure}[h]
    \centering
    \includegraphics[scale = 1.265]{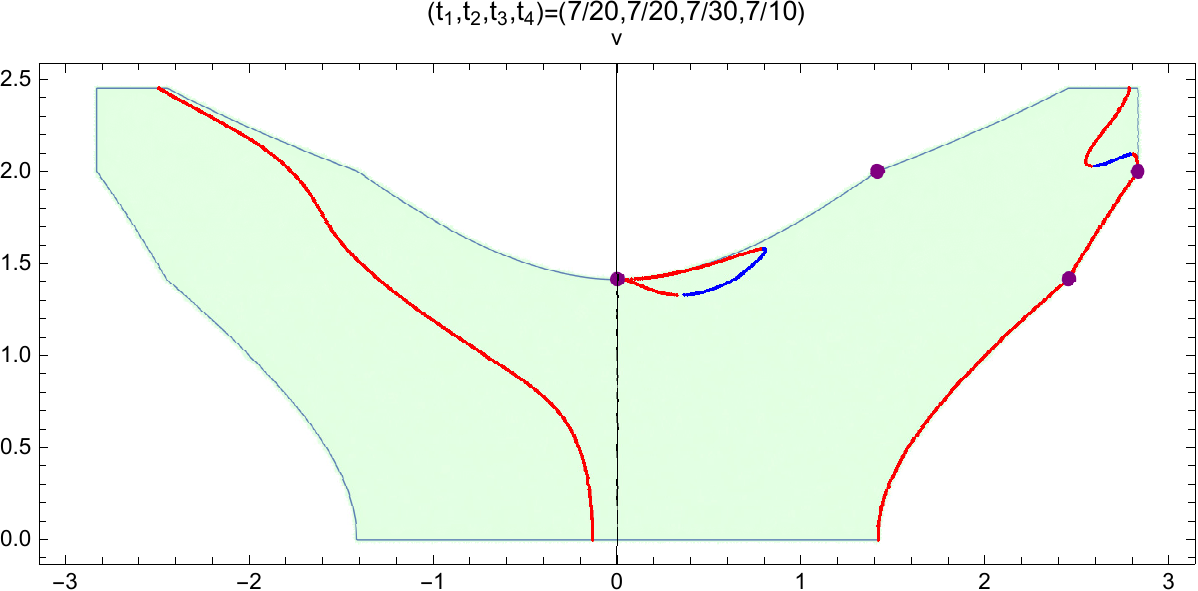}
    \caption{We consider the system for $t=\left(\frac{7}{20},\frac{7}{20}, \frac{7}{30}, \frac{7}{10} \right)$. The light green set is a plot of $\left(\phi_1^{-1}(U_1) \cap \{u=0\}\right) \slash \mbS^1$. The choice of the section $\{u=0\}$ is due to Lemma \ref{lem:omega}. The set of singular points is disconnected and consists of three connected components (one focus-focus point, one `loop' of different types of singular points, all remaining singular points). More precisely, the red lines represent the elliptic-regular points, the blue lines the hyperbolic-regular points, and the purple points are the invariant singular points.}
    \label{fig:singpoint}
\end{figure}


\subsection{Occurence of degenerate singular points}

In this subsection, we turn our attention to degenerate singular points. 
In hypersemitoric systems, the only admissible type of degenerate singular points are parabolic ones.
We will now show that $(J, H_t)$ contains, for certain $t \in \R^4$, degenerate points that are not parabolic. The family $(J,H_{(t_1,0,0,0)})$ with $t_1 \in \R$ coincides in fact with the perturbation of $(J, H)$ used in De Meulenaere $\&$ Hohloch \cite{meulenaere2019family}, but there the type of degeneracy was not investigated.

\begin{example}
For $t_1=\frac{25}{74}$ and $t_1=\frac{25}{26}$, the point $\phi_2(0,0)$ is a degenerate singular point of rank zero of $(J,H_{(t_1,0,0,0)})$ and therefore in particular not parabolic.
\end{example}

\begin{proof}
According to Theorem \ref{the:invariablesingpoint}, the point $\phi_2(0,0)$ is a rank zero singular point. Now we check if it is (non-)degenerate. To this aim, calculate the eigenvalues of
\begin{align*}
  \om^{-1}_{\phi_2(0,0)}\left(c_1 d^2 J + c_2 d^2H_{\left(\frac{25}{74},0,0,0\right)}\right)\bigl(\phi_2(0,0)\bigr) 
\end{align*}
which turn out to be, each with multiplicity two, 
\begin{align*}
    \pm \frac{\sqrt{-36 c_2^2 + 444 c_1 c_2 -1369 c_1^2}}{37} .
\end{align*}
Since the multiplicity of these eigenvalues is two and the eigenvalues of a non-degenerate point all have to be distinct it follows that $\phi_2(0,0)$ is degenerate for $\left(\frac{25}{74},0,0,0\right) \in \R^4$.
Similarly, the eigenvalues of
\begin{align*}
  \om^{-1}_{\phi_2(0,0)}\left(c_1 d^2 J + c_2 d^2H_{\left(\frac{25}{26},0,0,0\right)}\right)\bigl(\phi_2(0,0)\bigr) 
\end{align*}
turn out to be, each with multiplicity two, 
\begin{align*}
    \pm \frac{\sqrt{-36 c_2^2-156 c_1 c_2 -169 c_1^2}}{13} .
\end{align*}
Since the multiplicity of the eigenvalues is two it follows as above that $\phi_2(0,0)$ is also degenerate for $\left(\frac{25}{26},0,0,0\right) \in \R^4$.
\end{proof}

Therefore there are $t_1 \in \R$ in the family of systems $(J,H_{(t_1,0,0,0)})$, studied in De Meulenaere $\&$ Hohloch \cite{meulenaere2019family}, for which the system is not hypersemitoric. Thus, the same is certainly true for our family $(J, H_t)$ where $t$ varies in $\R^4$.

\begin{remark}
For $t=\left(\frac{7}{20},\frac{7}{20}, \frac{7}{30}, \frac{7}{10} \right)$, the system $(J, H_t)$ displays elliptic-regular points transitioning into hyperbolic-regular points by passing through a degenerate rank one singular point, see Figure \ref{fig:singpoint}. These degenerate points are located where the blue and red line meet.
\end{remark}

Parabolic degenerate points are sometimes also referred to as {\em cusps} which is due to their geometric shape, see Figure \ref{fig:parabolicExample}.

\begin{remark}
For $t=\left(\frac{1}{2}, \frac{1}{2}, \frac{1}{3}, 1 \right)$ and $j=0.664405$, the system $(J, H_t)$ has a loop of parabolic points $\theta \mapsto \phi_1(1.15274 \ e^{\theta i}, \ 0.470033 \ i)$. 
Descended to the (chart $U_1^{red, j}$ of the) reduced space $M^{red, j}$, the loop becomes the blue point at the `cusp' of the red curve in Figure \ref{fig:parabolicExample}.
\end{remark}


\bibliographystyle{alphaurl} 
\bibliography{refs} 

\vspace{3mm}

\noindent
Yannick Gullentops \\
Department of Mathematics \\
University of Antwerp \\
Middelheimlaan 1 \\
B-2020 Antwerp, Belgium \\
{\tt Yannick.Gullentops@uantwerpen.be} 

\vspace{3mm}

\noindent
Sonja Hohloch \\
Department of Mathematics \\
University of Antwerp \\
Middelheimlaan 1 \\
B-2020 Antwerp, Belgium \\
{\tt sonja.hohloch@uantwerpen.be} 

\end{document}